\documentclass[reqno, a4paper]{amsart}

\usepackage{amsfonts,amsmath,amsthm}
\usepackage{mathrsfs} 
\usepackage{graphicx}
\usepackage{amssymb}
\usepackage{esint} 

\usepackage[T1]{fontenc}
\usepackage[utf8]{inputenc}
\usepackage[british]{babel}
\usepackage{csquotes}
\usepackage{lmodern}

\usepackage{microtype} 

\usepackage{enumitem} 
\usepackage{thmtools}
\usepackage{setspace}

\usepackage{url} 

\usepackage{tikz,pgfplots}
\usetikzlibrary{calc,arrows}
\pgfplotsset{compat=1.4}

\usepackage[hidelinks]{hyperref}
\usepackage{doi}
\usepackage[noabbrev]{cleveref}
\usepackage{titlecaps}

\DeclareTextCommandDefault{\textbullet}{\ensuremath{\bullet}}

\newcommand{\N}{\ensuremath{\mathbb{N}}}
\newcommand{\Z}{\ensuremath{\mathbb{Z}}}
\newcommand{\R}{\ensuremath{\mathbb{R}}}
\renewcommand{\H}{\mathscr{H}} 

\newcommand{\Dt}{\ensuremath{D^t_{1/2}}} 
\newcommand{\D}{\ensuremath{\mathbb{D}}}
\newcommand{\Dn}{\ensuremath{\mathbb{D}_n}}
\newcommand{\tphi}{\ensuremath{\tilde{\phi}}}
\newcommand{\dbmo}{*, \,\text{dyadic}} 

\newcommand{\dx}{\ensuremath{\, \mathrm{d} x}}
\newcommand{\dX}{\ensuremath{\, \mathrm{d} X}}
\newcommand{\dy}{\ensuremath{\, \mathrm{d} y}}
\newcommand{\dY}{\ensuremath{\, \mathrm{d} Y}}
\newcommand{\dt}{\ensuremath{\, \mathrm{d} t}}

\newcommand{\du}{\ensuremath{\, \mathrm{d} u}}
\newcommand{\ds}{\ensuremath{\, \mathrm{d} s}}
\newcommand{\dlambda}{\ensuremath{\, \mathrm{d} \lambda}}
\newcommand{\drho}{\ensuremath{\, \mathrm{d} \rho}}
\newcommand{\dmu}{\ensuremath{\, \mathrm{d} \mu}}
\newcommand{\dnu}{\ensuremath{\, \mathrm{d} \nu}}
\newcommand{\dsig}{\ensuremath{\, \mathrm{d} \sigma}}

\newcommand{\dtau}{\ensuremath{\, \mathrm{d} \tau}}

\renewcommand{\epsilon}{\varepsilon}
\renewcommand{\L}{\ensuremath{\mathcal{L}}}

\DeclareMathOperator*{\supp}{supp}

\DeclareMathOperator{\BMO}{BMO}
\DeclareMathOperator{\VMO}{VMO}
\DeclareMathOperator{\loc}{loc}
\DeclareMathOperator{\Lip}{Lip}
\DeclareMathOperator*{\osc}{osc}
\DeclareMathOperator{\dist}{dist}
\DeclareMathOperator{\diam}{diam}
\DeclareMathOperator*{\di}{div}

\newtheorem{theorem}{Theorem}[section]
\newtheorem{lemma}[theorem]{Lemma}
\newtheorem{proposition}[theorem]{Proposition}
\newtheorem{corollary}[theorem]{Corollary}

\newtheorem{remark}[theorem]{Remark}
\newtheorem{definition}[theorem]{Definition}
\newtheorem{claim}[theorem]{Claim}

\numberwithin{equation}{section}

\crefname{corollary}{Corollary}{Corollaries}
\crefname{theorem}{Theorem}{Theorems}
\crefname{lemma}{Lemma}{Lemmas}
\crefname{proposition}{Proposition}{Propositions}
\crefname{equation}{}{} 

\newlist{steps}{enumerate}{3}
\setlist[steps]{wide} 
\setlist[steps,1]{label={\textbf{Step \arabic*:}},
				  ref={\arabic*}} 
\setlist[steps,2]{label={\textbf{Step \arabic{stepsi}.\alph*:}},
				  ref=\thestepsi{}.\alph*} 
\setlist[steps,3]{label={\textbf{Step \arabic{stepsi}.\alph{stepsii}.\roman*:}},
				  ref={\thestepsi{}.\thestepsii{}.\roman*}}
\crefname{stepsi}{step}{steps}
\crefname{stepsii}{step}{steps}
\crefname{stepsiii}{step}{steps}

\Crefname{stepsi}{Step}{Steps}
\Crefname{stepsii}{Step}{Steps}
\crefname{stepsiii}{Step}{Steps}

\newlist{caselist}{enumerate}{3}
\setlist[caselist]{label*=(\arabic*), ref=(\arabic*)} 
\crefname{caselisti}{case}{caselist}
\Crefname{caselisti}{Case}{Caselist}

\newlist{condition}{enumerate}{5}
\setlist[condition,1]{label=(\arabic*)} 
\setlist[condition,2]{label=(\alph*), ref=(\arabic{conditioni}.\alph*)} 
\crefname{conditioni}{condition}{conditions}
\Crefname{conditioni}{Condition}{conditions}
\crefname{conditionii}{condition}{conditions}
\Crefname{conditionii}{Condition}{conditions}

\newlist{property}{enumerate}{2}
\setlist[property]{label*=(\roman*), ref=(\roman*)} 
\crefname{propertyi}{property}{properties}
\Crefname{propertyi}{property}{properties}

\begin{document}
\title[Parabolic $L^p$ Dirichlet Problem and $\VMO$-type domains]{Parabolic $L^p$ Dirichlet Boundary Value Problem and $\VMO$-type time-varying domains}

\author{Martin Dindo\v{s}}
\address{School of Mathematics \\
	The University of Edinburgh and Maxwell Institute of Mathematical Sciences, UK}
\email{m.dindos@ed.ac.uk}

\author{Luke Dyer}
\address{School of Mathematics \\
	The University of Edinburgh and Maxwell Institute of Mathematical Sciences, UK}
\thanks{Luke Dyer was supported by The Maxwell Institute Graduate School in Analysis and its Applications, a Centre for Doctoral Training funded by the UK Engineering and Physical Sciences Research Council (grant EP/L016508/01), the Scottish Funding Council, the University of Edinburgh, and Heriot-Watt University.}
\email{l.dyer@sms.ed.ac.uk}

\author{Sukjung Hwang}
\address{Department of Mathematics\\ Yonsei University, Republic of Korea}
\email{\url{sukjung_hwang@yonsei.ac.kr}}

\date{}

\begin{abstract}
	We prove the solvability of the parabolic $L^p$ Dirichlet boundary value problem for $1 < p \leq \infty$ for a PDE of the form $u_t = \di (A \nabla u) + B \cdot \nabla u$ on time-varying domains where the coefficients $A= [a_{ij}(X, t)]$ and $B=[b_i]$ satisfy a certain natural small Carleson condition.
	This result brings the state of affairs in the parabolic setting up to the elliptic standard.

	Furthermore, we establish that if the coefficients of the operator $A,\,B$ satisfy a vanishing Carleson condition and the time-varying domain is of VMO type then the parabolic $L^p$ Dirichlet boundary value problem is solvable for all $1 < p \leq \infty$.
	This result is related to results in papers by Maz’ya, Mitrea and Shaposhnikova, and Hofmann, Mitrea and Taylor where the fact that boundary of domain has normal in VMO or near VMO implies invertibility of certain boundary operators in $L^p$ for all $1 < p \leq \infty$ which then (using the method of layer potentials) implies solvability of the $L^p$ boundary value problem in the same range for certain elliptic PDEs.

	Our result does not use the method of layer potentials, since the coefficients we consider are too rough to use this technique but remarkably we recover $L^p$ solvability in the full range of $p$'s as the two papers mentioned above.
\end{abstract}
\maketitle

\section{Introduction}

Let us consider a parabolic differential equation on a time-varying domain $\Omega$ of the form
\begin{equation}
	\label{E:pde}
	\begin{cases}
		u_t = \di (A \nabla u) + B \cdot \nabla u & \text{in } \Omega \subset \R^{n+1}, \\
		u   = f                                   & \text{on } \partial \Omega,
	\end{cases}
\end{equation}
where $A= [a_{ij}(X, t)]$ is a $n\times n$ matrix satisfying the uniform ellipticity condition with $X \in \R^n$, $t\in \R$.
That is, there exists positive constants $\lambda$ and $\Lambda$ such that
\begin{equation}
	\label{E:elliptic}
	\lambda |\xi|^2 \leq \sum_{i,j} a_{ij}(X,t) \xi_i \xi_j \leq \Lambda |\xi|^2
\end{equation}
for almost every $(X,t) \in \Omega$ and all $\xi \in \R^n$.
In addition, we assume that the coefficients of $A$ and $B$ satisfy a natural, minimal smoothness condition,~\cref{E:T1:carl:osc}, and we do not assume any symmetry on $A$.

It has been observed via the method of layer potentials that when the domain on which we consider certain boundary value problems for elliptic or parabolic PDEs is sufficiently smooth the question of $L^p$ invertibility of certain boundary operator can be resolved using the Fredholm theory since this operator is just a compact perturbation of the identity.
This observation then implies invertibility of this boundary operator for all $1 < p \leq \infty$ and hence solvability of the corresponding $L^p$ boundary value problem in this range.

The notion of how smooth the domain has to be for the above observation to hold has evolved. Initial results for constant coefficient elliptic PDEs required domains of at least $C^{1,\alpha}$ type.
This was reduced to $C^1$ domains in an important paper of Fabes, Jodeit and Rivi\`ere~\cite{FJR78}.
Later the method of layer potentials was adapted to variable coefficient settings and the results were extended to elliptic PDEs with variable coefficients~\cite{Din08} on $C^1$ domains.

Further progress was made after advancements in singular integrals theory on sets that are not necessary of graph-type~\cite{Sem91,HMT10}.
It turns out that compactness of the mentioned boundary operator only requires that the normal (which must be well defined at almost every boundary point) belongs to VMO\@.

This observation for the Stokes system was made in~\cite{MMS09} where boundary value problems for domains whose normal belongs to VMO (or is near to VMO in the BMO norm) were considered.
In~\cite{HMT15} symbol calculus for operators of layer potential type on surfaces with VMO normals was developed and applied to various elliptic PDEs including elliptic systems.

So far we have only mentioned elliptic results.
One of the first results for the heat equation in Lipschitz cylinders is by Brown~\cite{Bro89}.
Here the domain considered is time independent and Fourier methods in the time variable are used.
Domains of time-varying type for the heat operator were first considered the papers~\cite{LM95,HL96} and again the method of layer potentials was used to establish $L^2$ solvability.
The question of solvability of various boundary value problems for parabolic PDEs on time-varying, domains has long history.
Recall, that in the elliptic setting~\cite{Dah77} has shown in a Lipschitz domain that the harmonic measure and surface measure are mutually absolutely continuous, and that the elliptic Dirichlet problem is solvable with data in $L^2$ with respect to surface measure.
R. Hunt then asked whether Dalhberg's result held for the heat equation in domains whose boundaries are given locally as functions $\phi(x,t)$, Lipschitz in the spatial variable.
It was conjectured (due to the natural parabolic scaling) that the correct regularity of $\phi(x,t)$ should be a H\"older condition of order $1/2$ in the time variable $t$ and Lipschitz in $x$.
It turns out that under this assumption the parabolic measure associated with the equation~\eqref{E:pde} is doubling~\cite{Nys97}.

However, in order to answer R. Hunt's question positively one has to consider more regular classes of domains than the one just described above.
This follows from the counterexample of~\cite{KW88} where it was shown that under just the $\Lip(1,1/2)$ condition on the domain $\Omega$ the associated caloric measure (that is the measure associated with the operator $\partial_t-\Delta$) might not be mutually absolutely continuous with the natural surface measure.
The issue was resolved in~\cite{LM95} where it was established that mutual absolute continuity of caloric measure and a certain parabolic analogue of the surface measure holds when $\phi$ has $1/2$ of a time derivative in the parabolic $\BMO(\R^n)$ space, which is a slightly stronger condition than $\Lip(1,1/2)$.
We shall call such domains to be of Lewis-Murray type.
\cite{HL96} subsequently showed that this condition is sharp. 
We thoroughly discuss these domains in \cref{S:parabolic-domains}.

Further work was done by~\cite{HL01,Riv03,Riv14} in graph domains and time-varying cylinders satisfying the Lewis-Murray condition where they proved the $L^p$ Dirichlet problem was solvable for all $p >p'$ for some potentially very large $p'$ (due to the technique used there is no control on the size of $p'$).
Finally~\cite{DH16} has established $L^p$ solvability $2 \leq p \leq \infty$ in domains that are locally of Lewis-Murray type under a small Carleson condition. 

While researching literature on domains of Lewis-Murray type and ways this concept can be localized (in the time variable the half-derivative is a nonlocal operator and hence any condition imposed on it is difficult to localize) we have realized that important results we have planned to rely on have issues (either in their proofs or even worse are simply false, see in particular \cref{r-rn} in the next section). This has prompted us to write \cref{S:parabolic-domains} on parabolic domains in substantially more detail we originally intended to. This sets the literature record straight and more importantly in detail explains the concept of localized domains of Lewis-Murray type. For readability of the paper and this section we have moved long proofs into an appendix.

In this paper we establish $L^p$ solvability results for parabolic PDEs on time-varying cylinders satisfying locally the Lewis-Murray condition in the full range $1<p\leq\infty$ improving the solvability range from~\cite{DH16} as well as older results such as~\cite{HL96}, where only $p=2$ was considered.
The coefficients we consider are very rough and in particular the method of layer potentials cannot be used. Despite this we recover (in the parabolic setting) an analogue of~\cite{MMS09} and~\cite{HMT15}.
When the domain $\Omega$, on which the parabolic PDE is considered, is of VMO type (that is certain derivatives both in temporal and spatial variables will be in VMO) and the coefficients of the operator satisfy a vanishing Carleson condition the $L^p$ solvability can be established for all $1<p\leq\infty$.
Remarkably this is the full range of solvability that holds for smooth coefficients (via the layer potential method).

Our proof is however completely different from the layer potential method, for example at no point is compactness used.
The proof is also substantially different than the case $2\le p\leq\infty$ of~\cite{DH16} in the following way.
We were inspired by~\cite{DPP07} and have used a so called $p$-adapted square function in order to prove $L^p$ solvability.
However, due to the presence of parabolic term a second square function type object will arise, namely
\begin{equation}
	\label{FakeArea}
	\int_{\Omega}|u_t(X,t)|^2|u(X,t)|^{p-2}\delta(X,t)^3\dX\dt,
\end{equation}
where $\delta(X,t)$ is the parabolic distance to the boundary.
When $p=2$ such object was called the ``area function'' and in~\cite{DH16} it was shown that it can be dominated by the usual square function.
It turns out however that the case $1<p<2$ is substantially more complicated and we were only able to establish required bounds for~\eqref{FakeArea} for non-negative $u$ after a substantial effort.

There is also an issue of whether the $p$-adapted square function is actually well-defined and locally finite (as the exponent on $|u|$ is negative).
We prove that when $u$ is a solution of a parabolic PDE the $p$-adapted square function is indeed well defined by adapting a recent regularity result~\cite{DP16}.
The paper~\cite{DP16} deals with complex coefficient elliptic PDEs but the method used there can be adapted to the parabolic setting;
see \cref{T:improved-reg} for details.
\vglue1mm

Many results in the parabolic setting, are motivated by previous results in the elliptic setting and ours is not different.
Let us therefore overview the major elliptic results related to our main theorem.

The papers~\cite{KKPT00} and~\cite{KP01} started the study of non-symmetric divergence elliptic operators with bounded and measurable coefficients.
\cite{KP01} used~\cite{KKPT00} to show that the elliptic measure of operators satisfying a type of Carleson measure condition is in $A_\infty$ and hence the $L^p$ Dirichlet problem is solvable for some, potentially large, $p$. 
In~\cite{DPP07}, the authors improved the result of~\cite{KP01} in the following way.
They showed that if
\begin{equation}
	\label{E:intro:elliptic:carl}
	\delta(X)^{-1}\left(\osc_{B_{\delta(X)/2}(X)}a_{ij}\right)^2
	\quad\text{ and }\quad
	\delta(X)\left(\sup_{B_{\delta(X)/2}(X)} b_{i}\right)^2
\end{equation}
are densities of Carleson measures with vanishing Carleson norms then the $L^p$ Dirichlet problem is solvable for all $1 < p \leq \infty$.
A similar result for the elliptic Neumann and regularity boundary value problem was established in~\cite{DPR16}.

The parabolic analogue of the elliptic Carleson condition~\cref{E:intro:elliptic:carl} is that
\begin{equation}
	\label{E:intro:carl}
	\delta(X,t)^{-1} \sup_{i, j}\left( \osc_{B_{\delta(X,t)/2}(X,t)} a_{ij} \right)^2
	+ \delta(X,t)\left( \sup_{B_{\delta(X,t)/2}(X,t)} b_i\right)^2
\end{equation}
is the density of a Carleson measure on $\Omega$ with a small Carleson norm and $\delta(X,t)$ is the parabolic distance of a point $(X,t)$ to the boundary $\partial\Omega$.

The condition~\eqref{E:intro:carl} arises naturally as follows.
Let $\Omega = \{ (x_0, x, t) : x_0 > \phi(x,t)\}$ for a function $\phi$ which satisfies the Lewis-Murray condition above.
Let $\rho : U \to \Omega$ be a mapping from the upper half space $U$ to $\Omega$. Consider $v = u \circ \rho$. It will follow that if $u$ solves~\cref{E:pde} in $\Omega$ then $v$ will be a solution to a parabolic PDE similar to~\cref{E:pde} in $U$.
In particular if $\rho$ is chosen to be the mapping in~\eqref{mapping} then the coefficients of the new PDE for $v$ will satisfy a Carleson condition like~\cref{E:intro:carl}, c.f.\ \cref{L:LemmaA}, provided the original coefficients (for $u$) were either smooth or constant.

Furthermore, if we do not insist on control over the size of the Carleson norm then we can still infer solvability of the $L^p$ Dirichlet problem for large $p$, as in~\cite{HL01,Riv03,Riv14}.

Finally, we ready to state our main result;
some notions used here are defined in detail in \cref{S:preliminaries}.

\begin{theorem}%
	\label{T:1}
	Let $\Omega$ be a domain as in \cref{D:domain} with character $(\ell,\eta,N,d)$ and let $A$ be bounded and elliptic~\cref{E:elliptic}, and $B$ be measurable.
	Consider any $1<p\leq\infty$ and
	assume that either:
	\begin{enumerate}
		\item
		      \begin{equation}
			      \label{E:T1:carl:osc}
			      \dmu= \left[\delta(X,t)^{-1} \sup_{i,j}\left( \osc_{B_{\delta(X,t)/2}(X,t)} a_{ij} \right)^2 + 	\delta(X,t)\sup_{B_{\delta(X,t)/2}(X,t)} \left| B \right|^{2} \right] \dX\dt
		      \end{equation}
		      is the density of a Carleson measure on $\Omega$ with Carleson norm $\|\mu\|_C$.

		\item Or
		      \begin{equation}
			      \label{E:1:carl}
			      \dmu = \left( \delta(X,t)|\nabla A|^2 + \delta(X,t)^3|\partial_t A|^2 + \delta(X,t)|B|^2 \right) \dX\dt
		      \end{equation}
		      is the density of a Carleson measure on $\Omega$ with Carleson norm $\|\mu\|_C$ and
		      \begin{equation}
			      \label{E:1:bound}
			      \delta(X,t)|\nabla A| + \delta(X,t)^2|\partial_t A| + \delta(X,t)|B| \leq \|\mu\|_C^{1/2}.
		      \end{equation}
	\end{enumerate}
	Then there exists $K=K(\lambda,\Lambda,\ell,n,p) > 0$ such that if for some $r_0 > 0$
	\[
		\max\{\eta, \|\mu\|_{C,r_0} \} < K
	\]
	the $L^p$ Dirichlet boundary value problem~\cref{E:pde} is solvable (c.f.\ \cref{D:Dp}).
	Moreover, the following estimate holds for all continuous boundary data $f \in C_0(\partial\Omega)$
	\[
		\|N(u)\|_{L^p(\partial\Omega, \dsig)} \lesssim \|f\|_{L^p(\partial\Omega, \dsig)},
	\]
	where the implied constant depends only on the operator, $n$, $p$ and character $(\ell, \eta, N,d)$, and $N(u)$ is the non-tangential maximal function of $u$.
\end{theorem}

\begin{corollary}
	In particular, if $\Omega$ is of VMO-type ($\eta$ in the character $(\ell,\eta,N,d)$ can be taken arbitrary small), and the Carleson measure $\mu$ from \cref{T:1} is a vanishing Carleson measure then the $L^p$ Dirichlet boundary value problem~\cref{E:pde} is solvable for all $1<p\leq\infty$.
\end{corollary}

\section{Preliminaries}%
\label{S:preliminaries}

Here and throughout we consistently use $\nabla u$ to denote the gradient in the spatial variables and $u_t$ or $\partial_t u$ the gradient in the time variable.

\subsection{Parabolic Domains}%
\label{S:parabolic-domains}

In this subsection we define a class of time-varying domains whose boundaries are given locally as functions $\phi(x,t)$, Lipschitz in the spatial variable and satisfying the Lewis-Murray condition in the time variable.
At each time $\tau\in\R$ the set of points in $\Omega$ with fixed time $t=\tau$, that is $\Omega_\tau = \Omega \cap \{t=\tau\}$, is a non-empty bounded Lipschitz domain in $\R^n$.
We start with a discussion of the Lewis-Murray condition, give a summary and clarification of the results in the literature, and introduce some new equivalent definitions.

We define a \textit{parabolic cube} in $\R^{n-1} \times \R$,  for a constant $r>0$, as
\begin{equation*}
	\label{D:Q}
	Q_{r} (x, t),
	= \{ (y, s) \in \R^{n-1}\times\R : |x_i - y_i| < r \ \text{for all } 1 \leq i \leq n-1, \ | t - s |^{1/2} < r \}.
\end{equation*}
Let $J_r \subset \R^{n-1}$ be a \textit{spatial cube} of radius $r$.
For a given $f: \R^n \to \R$ let
\[
	f_{Q_r} = \frac{1}{|Q_r|} \int_{Q_r} f(x,t) \dx\dt.
\]
When we write $f \in \BMO (\R^n) $ we mean that $f$ belongs to the parabolic version of the usual BMO space with the norm $\|f\|_{*}$ where
\begin{equation}
	\label{E:BMO}
	\|f\|_{*} = \sup_{Q_r} \frac{1}{|Q_r|} \int_{Q_r} |f - f_{Q_r} | \dx\dt  < \infty.
\end{equation}
Recall that the Lewis-Murray condition imposed that a half derivative in time of $\phi(x,t)$ belongs to parabolic $\BMO$.
There are a few different ways one can define half derivatives and $\BMO$-Sobolev spaces and there are also some erroneous results in the literature which we correct here.
To bring clarity, we start by discussing the various definitions in the global setting of a graph domain $\Omega = \{(x_0, x,t) : x_0 > \phi(x,t) \}$, where $\phi: \R^{n-1} \times \R \to \R$.
We follow the standard notation of~\cite{HL96}.

If $g \in C^\infty_0(\R)$ and $0<\alpha<2$ then the one-dimensional fractional differentiation operators $D_\alpha$ are defined on the Fourier side by
\[
	\widehat{D_\alpha g}(\tau) = |\tau|^\alpha \hat{g}(\tau).
\]
If $0 < \alpha < 1$ then by standard results
\[
	D_\alpha g(t) = c \int_\R \frac{g(t) - g(s)}{|t-s|^{1+\alpha}} \ds.
\]
Therefore, we define the \textit{pointwise half derivative in time} of $\phi : \R^{n-1} \times \R \to \R$ to be
\begin{equation}
	\label{E:Dt}
	D_{1/2}^{t} \phi(x,t) = c_n \int_{\R} \frac{\phi(x,s) - \phi(x,t)}{|s-t|^{3/2}} \ds,
\end{equation}
for a properly chosen constant $c_n$ (c.f.~\cite{HL96}).

However, this definition ignores the spatial coordinates.
Instead by following~\cite{FR67} we may define the \textit{parabolic half derivative in time} of $\phi : \R^{n-1} \times \R \to \R$ to be
\begin{equation}
	\label{E:Dn}
	\widehat{\Dn\phi}(\xi,\tau) = \frac{\tau}{\|(\xi,\tau)\|} \hat{\phi}(\xi,\tau),
\end{equation}
where $\xi$ and $\tau$ denote the spatial and temporal variables on the Fourier side respectively, and $\|(x,t)\| = |x| + |t|^{1/2}$ denotes the parabolic norm.
In addition we define the \textit{parabolic derivative} (in space and time) of $\phi : \R^{n-1} \times \R \to \R$ to be
\begin{equation}
	\label{E:D}
	\widehat{\D\phi}(\xi,\tau) = \|(\xi,\tau)\| \hat{\phi}(\xi,\tau).
\end{equation}
$\D^{-1}$ is the parabolic Riesz potential.
One can also represent $\D$ as
\begin{equation}
	\label{E:D:sum}
	\D = \sum_{j=1}^n R_j \D_j,
\end{equation}
where $\D_j = \partial_j$ for $1 \leq j \leq n-1$, $\Dn$ is defined above and $R_j$ are the parabolic Riesz transforms defined on the Fourier side as
\begin{equation}
	\label{E:Rj}
	\begin{split}
		\widehat{R_j}(\xi,\tau) &= \frac{i \xi_j}{\|(\xi,\tau)\|} \quad \text{for } 1 \leq j \leq n-1 \text{ and} \\
		\widehat{R_n}(\xi,\tau) &= \frac{\tau}{\|(\xi,\tau)\|^2}.
	\end{split}
\end{equation}
Furthermore the kernels of $R_j$ have average zero on (parabolically weighted) spheres around the origin, obey the standard Calder\`on-Zygmund kernel and therefore by standard Calder\`on-Zygmund theory each $R_j$ defines a bounded operator on $L^p(\R^n)$ for $1 < p < \infty$ and is bounded on $\BMO(\R^n)$~\cite{Pee66,FR66,FR67,HL96}.

We say that $\phi : \R^{n-1} \times \R \to \R$ is $\Lip(1,1/2)$ with Lipschitz constant $\ell$ if $\phi$ is Lipschitz in the spatial variables and H\"older continuous of order $1/2$ in the temporal variable.
That is
\begin{equation}
	\label{E:Lip}
	|\phi_j(x,t)- \phi_j(y,t)| \leq \ell \left( |x-y| +|t-s|^{1/2}\right).
\end{equation}

The \textit{Lewis-Murray condition} on the domain $\Omega$, for which they proved the mutual absolute continuity of the caloric measure and the natural surface measure, is $\phi \in \Lip(1,1/2)$ and $\|\Dt \phi\|_* \leq \eta$; note this $\BMO$ norm is taken over $\R^n$.

It is worth remarking that neither the operators $\Dt$, $\Dn$ or $\D$ easily lend themselves to being localised to a function $\phi : Q_d \to \R$ due to their non-local natures.
However, our goal is provide a theory where the domain is locally given by graphs which satisfy the Lewis-Murray condition.
The parabolic nature of the PDE (especially time irreversibility and exponential decay of solutions with vanishing boundary data) suggest we should expect to need only local conditions on the functions describing the boundary.

To this end we state the following theorems where we show some equivalent statements to the Lewis-Murray condition for a global function $\phi : \R^{n-1} \times \R \to \R$.
Furthermore, the final conditions admit themselves to both being localised easily as well as amiable to extension, see \cref{T:ext} later for details on a extension.

The equivalence of \cref{I:Dt,I:Dn} below is shown in~\cite{HL96} with an equivalence of norms in the small and large sense, see~\cite[(2.10) and Theorem 7.4]{HL96} for precise details, c.f.\ \cref{E:D:sum,E:Rj}.
\begin{theorem}%
	\label{T:equiv:known}
	Let $\phi : \R^{n-1} \times \R \to \R$ and $\phi \in \Lip(1,1/2)$ then the following conditions are equivalent:
	\begin{condition}[series=equiv]
		\item\label{I:Dt} $\Dt\phi \in \BMO(\R^n)$

		\item\label{I:Dn} $\Dn\phi\in\BMO(\R^n)$

		\item\label{I:D} $\D\phi \in\BMO(\R^n)$.
	\end{condition}
	Furthermore $\Dn\phi = R_n \D\phi$ and so $\|\Dn\phi\|_* \lesssim \|\D\phi\|_*$.
\end{theorem}

We now extend this theorem by adding three more equivalent statements.
To motivate \cref{I:av} of \cref{T:equiv} below we first recall a characterisation of $\BMO$ from~\cite[p.~546]{Str80}.
Let $M(f, Q) = \frac{1}{|Q|}\int_Q f$ denote the average of $f$ over a cube $Q$, and let $\tilde{Q}_\rho(x)$ be the cube of radius $\rho$ with $x$ in the upper right corner.
\begin{lemma}[{\cite{Str80}}]%
	\label{L:BMO:av}
	$f \in \BMO(\R^n)$ is equivalent to
	\begin{equation}
		\label{E:BMO:av}
		\sup_{Q_r} \sum_{k=1}^n \frac{1}{|Q_r|} \int_{Q_r}\int_0^r
		\left|M(f,\tilde{Q}_\rho(x)) - M(f,\tilde{Q}_\rho(x - \rho e_k))\right|^2 \frac{\drho}{\rho}\dx = B < \infty,
	\end{equation}
	where $e_k$ are the usual unit vectors in $\R^n$, and $\|f\|^2_* \sim B$.
\end{lemma}

The equivalence of \cref{I:D,I:Dt:riv} in theorem below is a generalisation of~\cite{Str80} to the parabolic setting that is stated in~\cite{Riv03}, c.f.\ \cite{FS72,CT75,CT77}.
We have some question-marks over the proof given in~\cite{Riv03}; however the argument we give for  \cref{I:diff} also works for \cref{I:Dt:riv} and hence the claim in~\cite{Riv03} is correct.
\begin{theorem}%
	\label{T:equiv}
	Let $\phi : \R^{n-1} \times \R \to \R$ and $\phi \in \Lip(1,1/2)$ then the following conditions are equivalent:
	\begin{condition}[start=3]
		\item $\D\phi \in\BMO(\R^n)$

		\item\label{I:Dt:riv}
		\begin{equation}
			\label{E:Dt:riv}
			\sup_{Q_r} \frac{1}{|Q_r|}\int\limits_{Q_r}\int\limits_{\|(y,s)\| \leq r}
			\frac{|\phi(x+y,t+s) - 2\phi(x,t) + \phi(x-y,t-s)|^2}{\|(y,s)\|^{n+3}} \dy\ds\dx\dt =  B_{\ref{I:Dt:riv}} < \infty,
		\end{equation}

		\item\label{I:diff}
		\begin{condition}
			\item\label{I:diff:grad}
			\begin{equation}
				\label{E:diff:grad}
				\sup_{Q_r} \frac{1}{|Q_r|}\int\limits_{Q_r}\int_{|y|<r}
				\frac{|\phi(x+y,t) - 2\phi(x,t) + \phi(x-y,t)|^2}{|y|^{n+1}} \dy\dx\dt = B_{\ref{I:diff:grad}} < \infty,
			\end{equation}

			\item\label{I:diff:Dt}
			\begin{equation}
				\label{E:Dt:av}
				\sup_{Q_r = J_r \times I_r} \frac{1}{|Q_r|}\int_{Q_r}\int_{I_r}
				\frac{|\phi(x,t) - \phi(x,s)|^2}{|t-s|^2} \ds\dt\dx = B_{\ref{I:diff:Dt}} < \infty.
			\end{equation}
		\end{condition}

		\item\label{I:av}
		Let $u = (u', u_n) \in \mathbb{S}^{n-1}$ and let $e_n$ be the unit vector in the time direction.
		For $k = 1, \dots, n-1$ let
		\begin{align*}
			A_k & = \int_0^1 \rho u' \cdot \left(M\left(\nabla \phi, \tilde{Q}_\rho (x + \lambda\rho u', t)\right) - M\left(\nabla \phi, \tilde{Q}_\rho (x + \lambda\rho u' - \rho e_k, t)\right) \right) \dlambda, \\
			A_n & = \int_0^1 \rho u' \cdot \left(M\left(\nabla \phi, \tilde{Q}_\rho (x + \lambda\rho u', t)\right) - M\left(\nabla \phi, \tilde{Q}_\rho (x + \lambda\rho u', t - \rho^2)\right) \right) \dlambda.
		\end{align*}
		\begin{condition}
			\item\label{I:av:grad}
			\begin{equation}
				\label{E:av:grad}
				\sup_{Q_r} \sum_{k=1}^n \frac{1}{|Q_r|} \int_{Q_r} \int_{u \in \mathbb{S}^{n-1}} \int_0^r \frac{|A_k|^2}{\rho^3} \drho\du\dx\dt = B_{\ref{I:av:grad}} < \infty,
			\end{equation}

			\item\label{I:av:Dt}
			\begin{equation*}
				\sup_{Q_r = J_r \times I_r} \frac{1}{|Q_r|}\int_{Q_r}\int_{I_r}
				\frac{|\phi(x,t) - \phi(x,s)|^2}{|t-s|^2} \ds\dt\dx = B_{\ref{I:av:Dt}} < \infty. \tag{\ref{E:Dt:av}}
			\end{equation*}
		\end{condition}
	\end{condition}
	Furthermore we have equivalence of the norms

	\begin{equation}
		\label{E:equiv:norms}
		\|\D\phi\|_*^2
		\sim B_{\ref{I:Dt:riv}}
		\sim B_{\ref{I:diff:grad}} + B_{\ref{I:diff:Dt}}
		\sim B_{\ref{I:av:grad}} + B_{\ref{I:av:Dt}}.
	\end{equation}

\end{theorem}
We give a proof of this result in the appendix at the end of the paper.

\begin{remark}%
	\label{R:av:grad}
	\Cref{I:av:grad} doesn't immediately look too similar to its supposed motivation, \cref{E:BMO:av} in \cref{L:BMO:av}.
	However, if we move back into Cartesian coordinates and undo the mean value theorem then we obtain something much similar to a combination of~\cref{E:BMO:av} and a endpoint version of~\cite[(3.1)]{Str80}.
	The reason why we can obtain the endpoint, whereas~\cite[(3.1)]{Str80} can only be used for a fractional derivative smaller than $1$, is due to extra integrability and cancellation coming from \cref{E:equiv:phi:k}.  Consider

	\begin{align*}
		A_k' & = M\left(\phi, \tilde{Q}_{\|(y,s)\|} (x + y, t)\right) - M\left(\phi, \tilde{Q}_{\|(y,s)\|} (x, t)\right) \\
		     & \quad- M\left(\phi, \tilde{Q}_{\|(y,s)\|} (x + y - {\|(y,s)\|} e_k, t)\right) + M\left(\phi, \tilde{Q}_{\|(y,s)\|} (x -{\|(y,s)\|} e_k, t)\right), \\
		A_n' & = M\left(\phi, \tilde{Q}_{\|(y,s)\|} (x + y, t)\right) - M\left(\phi, \tilde{Q}_{\|(y,s)\|} (x, t)\right) \\
		     & \quad- M\left(\phi, \tilde{Q}_{\|(y,s)\|} (x + y , t- {\|(y,s)\|}^2)\right) + M\left(\phi, \tilde{Q}_{\|(y,s)\|} (x , t-{\|(y,s)\|}^2)\right)
	\end{align*}
	then \cref{I:av:grad} is equivalent to
	\begin{equation}
		\label{E:av:grad:alt}
		\sup_{Q_r} \sum_{k=1}^n \frac{1}{|Q_r|} \int_{Q_r} \int_{\|(y,s)\| < r} \frac{|A'_k|^2}{\|(y,s)\|^{n+3}} \dy\ds\dx\dt = \tilde B_{\ref{I:av:grad}} < \infty.
	\end{equation}
\end{remark}

\begin{proposition}%
	\label{P:equiv:grad}
	$\nabla\phi \in \BMO(\R^n)$ implies \cref{I:av:grad} and $\nabla\phi(\cdot,t) \in \BMO(\R^{n-1})$ uniform a.e.\ in time implies \cref{I:diff:grad}, with constants $B_{\ref{I:diff:grad}}$ and $B_{\ref{I:av:grad}}$ controlled by the appropriate $\|\nabla\phi\|_*^2$ norm. Here $\BMO(\R^{n-1})$ denotes the $\BMO$ norm in the spatial variables only.
\end{proposition}
\begin{proof}
	The statement $\nabla\phi \in \BMO(\R^{n-1})$ implies \cref{I:diff:grad}  follows from~\cite[Theorem 3.3]{Str80}.
	In order to establish the second claim for the ease of notation let us fix $Q_r$ and $k$ in $1 \leq k \leq n-1$.
	Then since $|u'| \leq 1$ after changing the order of integration (and the substitution $y = x + \lambda\rho u' \in Q_{2r}$) we get that \cref{E:av:grad} is bounded by
	\begin{align*}
		\int\limits_0^1 \int\limits_{\mathbb{S}^{n-1}} \int\limits_0^r  \frac{1}{|Q_r|} \int\limits_{Q_{2r}}
		\left| \left(M\left(\nabla u, \tilde{Q}_\rho (y, t)\right) - M\left(\nabla u, \tilde{Q}_\rho (y - \rho e_k, t)\right) \right)\right|^2
		\!\dy\dt\frac{\drho}{\rho}\du\dlambda.
	\end{align*}
	Then by \Cref{L:BMO:av}  the two interior integrals are bounded by $C\|\nabla\phi\|_*^2$.
	Therefore~\cref{E:av:grad} is controlled by $C\|\nabla\phi\|_*^2$.
\end{proof}
The opposite implications are likely to be false due the highly singular nature of Riesz potentials, c.f.\ \cref{E:D:sum,E:Rj}.

\begin{corollary}%
	\label{C:equiv:grad}
	If $\|\nabla\phi\|_* \lesssim \eta$, and  $B_{\ref{I:diff:Dt}} \lesssim \eta^2$ then $\|\D\phi\|_* \lesssim \eta$.
\end{corollary}

Here we have replaced conditions~\ref{I:diff:grad} or~\ref{I:av:grad} by slightly stronger but easier to verify condition $\|\nabla\phi\|_* \lesssim \eta$.
We believe that, without too much extra work, one could formulate our main theorem and associated lemmas with a local version of \cref{I:diff:grad} in place of $\|\nabla\phi\|_*$.

\begin{remark}\label{r-rn}
	In~\cite[Lemma 2.1]{Riv03} it is stated that another condition is equivalent to those given in \cref{T:equiv:known,T:equiv}; however this claim is not correct and only one of the stated implications holds.

	A result of Strichartz~\cite[Theorem 3.3]{Str80} states that in the one dimensional setting $\Dt \phi(t) \in \BMO(\R)$ is equivalent to the one dimensional version of \cref{I:diff:Dt,I:av:Dt}
	\begin{equation}
		\label{E:Str:R}
		\sup_{I' \subset \R} \left( \frac{1}{|I'|}\int_{I'}\int_{I'} \frac{|\phi(t) - \phi(s)|^2}{|t-s|^2} \dt\ds\right)^{1/2} \leq B,
	\end{equation}
	with $B \sim \|\Dt \phi(\cdot)\|_{\BMO(\R)}$.

	In~\cite[Lemma 2.1]{Riv03} it is claimed that given $\phi : \R^{n-1} \times \R \to \R$ and $\phi \in \Lip(1,1/2)$ the pointwise $n$-dimensional analogue of~\cref{E:Str:R}
	\begin{equation}
		\label{E:riv}
		\sup_{x \in \R^{n-1}} \sup_{I' \subset \R} \left( \frac{1}{|I'|}\int_{I'}\int_{I'} \frac{|\phi(x,t) - \phi(x,s)|^2}{|t-s|^2} \dt\ds\right)^{1/2} \leq B
	\end{equation}
	is equivalent to $\D\phi \in\BMO(\R^n)$ with $B \sim \|\D \phi\|_{\BMO(\R^n)}$.
	This is incorrect.
	By~\cite{Str80}~\cref{E:riv} is equivalent to $\Dt\phi(x,\cdot) \in \BMO(\R)$ pointwise for a.e.\ $x$.
	After some tedious and technical calculations we were able to show $\sup_x \Dt\phi(x,\cdot) \in \BMO(\R)$ implies $\Dt\phi \in\BMO(\R^n)$ and hence $\D\phi \in\BMO(\R^n)$ via \cref{I:Dt:riv} of \cref{T:equiv}.
	However, the converse is not true even if we assume more structure for the function $\D\phi(x,t)$.
	This is due to the fact that there is ``no reasonable Fubini theorem relating $\BMO(\R^n)$ to $\BMO(\R)$''~\cite[p.~558]{Str80}.

	Fortunately the lack of a converse implication does not cast doubt over the subsequent results of~\cite{Riv03} since the author only uses the claimed eqivalence in the correct direction --- that~\cref{E:riv} implies $\D\phi \in\BMO(\R^n)$.
\end{remark}

\subsection*{Localisation}
After the comprehensive review of the Lewis-Murray condition for a graph domain $\Omega$ we continue in our aim to construct a time-varying domain which is locally described by local graphs $\phi_j$.

For a vector $x \in \R^{n-1}$ we denote consider the norm $|x|_\infty = \sup_i |x_i|$.

Consider $\phi : Q_{8d} \to \R^{n-1} \times \R$.
The localised version of \cref{E:Dt:av} from \cref{T:equiv} is simply
\begin{equation}
	\label{E:Dt:loc}
	\sup_{\substack{Q_r = J_r \times I_r \\ Q_r \subset Q_{8d}}} \frac{1}{|Q_r|}\int_{Q_r}\int_{I_r}
	\frac{|\phi(x,t) - \phi(x,s)|^2}{|t-s|^2} \ds\dt\dx < \infty.
\end{equation}

We write $\|f\|_{*, d}$ to be the $\BMO$ norm of $f$ where the supremum in the $\BMO$ norm, c.f.\ \cref{E:BMO}, is taken over all cubes $Q_r$ with $r \leq d$.
For a function $f:J\times I\to\R$, where $J\subset \R^{n-1}$ and $I\subset \R$ are closed bounded cubes we consider the norm $\|f\|_{*,J \times I}$ defined as above where the supremum is taken over all parabolic cubes $Q_r$ contained in $J\times I$.
The norm $\|f\|_{*,J \times I, d}$ is where the supremum is taken over all parabolic cubes $Q_r$ with $r \leq d$ contained in $J\times I$.
If the context is clear we suppress the $J \times I$ and just write $\|f\|_*$ or $\|f\|_{*, d}$.

Recall that $\VMO (\R^n)$ is defined as the closure of all continuous functions in the $\BMO$ norm or equivalently $\BMO$ functions $f$ such that $\|f\|_{*, d} \to 0$ as $d \to 0$.
Alternatively, if we define
\[
	d(f,\VMO):= \inf_{h\in C} \|f-h\|_{*}
\]
then $f\in \VMO$ if and only if $d(f,\VMO)=0$; for $f\in\BMO$ this measures the distance of $f$ to $\VMO$.
In our case, the boundary of the parabolic domains we consider can be locally described as a graph of a continuous function.
However, as our domain is unbounded in time we may potentially require an infinite family of local graphs $\{ \phi_j \}$.
Therefore we need to measure the distance to $\VMO$ uniformly across this infinite family.

Let $\delta: \R_+ \to \R_+$, $\delta(0) = 0$ and $\delta$ be continuous at $0$ then we define $C_\delta$ to be the set of continuous functions with the same modulus of continuity $\delta$.
That is
\begin{equation}
	\label{E:Cdelta}
	C_\delta = \{g \in C : |g(x) - g(y)| \leq \delta(|x-y|) \text{ for all } x,y \}.
\end{equation}
Note that every family of equicontinuous functions can be represented as $C_\delta$ for some function $\delta$ and $C = \cup_\delta C_\delta$.
For $f: Q_{8d} \to \R$ we define $d(f,C_\delta)$ as
\begin{equation*}
	d(f,C_\delta) = \inf_{h\in C_\delta} \|f-h\|_{*, Q_{8d}}.
\end{equation*}

We are now ready to state and prove result on extensibility of $\phi : Q_{8d} \to \R$ to a  global function.
\begin{theorem}%
	\label{T:ext}
	Let $\phi : Q_{8d} \subset \R^{n-1} \times \R \to \R$ be $\Lip(1,1/2)$ with Lipschitz constant $\ell$.
	If there exist a scale $r_1$, a constant $\eta>0$ and a modulus of continuity $\delta$ such that
	\begin{equation}
		\sup_{\substack{Q_s = J_s \times I_s \\ Q_s \subset Q_{8d}, \, s \leq r_1}} \frac{1}{|Q_s|}\int_{Q_s}\int_{I_s}
		\frac{|\phi(x,t) - \phi(x,\tau)|^2}{|t-\tau|^2} \dtau\dt\dx \leq \eta^2
	\end{equation}
	and
	\begin{equation}
		\label{E:ext:1}
		d(\nabla\phi, C_\delta) \leq \eta
	\end{equation}
	then there exists a scale $d'\leq d$, that only depends on $d$, $\delta$, $\eta$, and $r_1$ and not $\phi$, such that for all $Q_r \subset Q_{4d}$ with $r \leq d'$ there exists a global $\Lip(1,1/2)$ function $\Phi : \R^{n-1} \times \R \to \R$ with the following properties for all $0 < \epsilon < 1$:
	\begin{property}
		\item\label{I:ext:ext} $\Phi|_{Q_r} = \phi|_{Q_r}$,

		\item\label{I:ext:lip} the $\Lip(1,1/2)$ constant of $\Phi$ is $\ell$,

		\item\label{I:ext:nabla} $\|\nabla\Phi\|_{*} \lesssim_\epsilon \eta^{1-\epsilon} + \eta\ell $, and

		\item\label{I:ext:Dt} $\displaystyle \sup_{Q_s = J_s \times I_s} \frac{1}{|Q_s|}\int_{Q_s}\int_{I_s}
			\frac{|\Phi(x,t) - \Phi(x,\tau)|^2}{|t-\tau|^2} \dtau\dt\dx \lesssim \eta^2$.
	\end{property}
	Therefore by \cref{C:equiv:grad}, $\|\D\Phi\|_{*} \lesssim_\epsilon \eta^{1-\epsilon} + \eta\ell$.
\end{theorem}

We again give the proof of this result in the appendix.
We are now ready to define the class of parabolic domains on which we will work. Motivated by the usual definition of a Lipschitz domain we have:

\begin{definition}
	$\Z \subset \R^n\times \R$ is an \textit{$\ell$-cylinder} of diameter $d$ if there exists a coordinate system $(x_0,x,t)\in \R\times\R^{n-1}\times \R$ obtained from the original coordinate system by translation in spatial and time variables, and rotation only in the spatial variables such that
	\[
		\Z = \{ (x_0,x,t) : |x|\leq d, |t|^{1/2} \leq d, |x_0| \leq (\ell + 1)d \}
	\]
	and for $s>0$
	\[
		s\Z:=\{(x_0,x,t) : |x|<sd, |t|^{1/2}\leq sd, |x_0| \leq (\ell +1)sd \}.
	\]
\end{definition}

\begin{definition}%
	\label{D:domain}
	$\Omega\subset \R^n\times \R$ is an \textit{admissible parabolic domain} with character $(\ell,\eta, N,d)$ if there exists a positive scale $r_1$, and a modulus of continuity $\delta$ such that for any time $\tau\in\R$ there are at most $N$ $\ell$-cylinders $\{{\Z}_j\}_{j=1}^N$ of diameter $d$ satisfying the following conditions:
	\begin{condition}
		\item $\displaystyle \partial\Omega\cap\{|t-\tau|\le d^2\}=\bigcup_j ({\Z}_j \cap \partial\Omega)$.

		\item In the coordinate system $(x_0,x,t)$ of the $\ell$-cylinder $\Z_j$
		\[
			\Z_j \cap \Omega \supset \left\{(x_0,x,t)\in\Omega : |x|<d, |t|<d^2, \delta(x_0,x,t) \leq d/2 \right \}.
		\]
		\item\label{I:domain:lip} $8{\Z}_j \cap \partial\Omega$ is the graph $\{x_0=\phi_j(x,t)\}$ of a
		function $\phi_j: Q_{8d} \to \R$, with $Q_{8d} \subset \R^{n-1} \times \R$, such that
		\begin{equation}
			\label{E:L1}
			|\phi_j(x,t)- \phi_j(y,s)| \leq \ell \left( |x-y| +|t-s|^{1/2}\right) \quad \text{and} \quad
			\phi_j(0,0)=0.
		\end{equation}
		\item\label{I:domain:eta}
		\begin{equation}
			\label{E:L2}
			d(\nabla\phi_j,C_\delta)\le\eta
		\end{equation}
		and
		\begin{equation}
			\label{E:domain:eta}
			\sup_{\substack{Q_s = J_s \times I_s \\ Q_s \subset Q_{8d}, \, s \leq r_1}} \frac{1}{|Q_s|}\int_{Q_s}\int_{I_s}
			\frac{|\phi_j(x,t) - \phi_j(x,\tau)|^2}{|t-\tau|^2} \dtau\dt\dx \leq \eta^2.
		\end{equation}
	\end{condition}
	Here and throughout $\delta(x_0,x,t) := \dist \left( (x_0,x,t),\partial\Omega\right)$ and $\dist$ is the parabolic distance $\dist[(X,t),(Y,s)]=|X-Y|+|t-s|^{1/2}$.

	We say that $\Omega$ is of $\VMO$ type if $\eta$ in the character $(\ell,\eta, N,d)$ can be taken arbitrarily small (at the expense of a potentially smaller $d$ and $r_1$, and larger $N$).
\end{definition}

\begin{remark}
	When~\eqref{E:L2} holds for small or vanishing $\eta$ it follows that for a fixed time $\tau$ the normal $\nu$ to the fixed-time spatial domain $\Omega_{\tau} = \Omega \cap \{ t=\tau \}$ can be written in local coordinates as
	\[
		\nu=\frac1{|(-1,\nabla\phi_j)|}(-1,\nabla\phi_j)
	\]
	and hence $d(\nu,\VMO)\lesssim \eta$.
	Therefore $\Omega_\tau$ is similar to the domains considered in the papers~\cite{MMS09} and~\cite{HMT15} which have dealt with the elliptic problems on domains with normal in or near $\VMO$.
\end{remark}

\begin{remark}
	It follows from this definition that for each $\tau\in\R$ the time-slice $\Omega_\tau$ of an admissible parabolic domain $\Omega\subset \R^n\times \R$ is a bounded Lipschitz domain in $\R^n$ and they all have a uniformly bounded diameter.
	That is
	\[
		\inf_{\tau\in\R}\diam(\Omega_\tau) \sim d \sim \sup_{\tau\in\R}\diam(\Omega_\tau),
	\]
	where $d$ is the scale from \cref{D:domain}  and the implied constants only depends on $N$.
	In particular, if ${\mathcal O}\subset \R^n$ is a bounded Lipschitz domain then the parabolic cylinder $\Omega={\mathcal O}\times \R$ is an example of a domain satisfying \cref{D:domain}.
\end{remark}

\begin{definition}%
	\label{D:measure}
	Let $\Omega\subset \R^n\times \R$ be an admissible parabolic domain with character $(\ell, \eta,N,d)$.
	The \textit{measure} $\sigma$ defined on sets $A\subset \partial\Omega$ is
	\begin{equation}
		\label{E:sigma}
		\sigma(A)=\int_{-\infty}^\infty {\H}^{n-1}\left(A \cap \{(X,t) \in \partial\Omega \}\right) \dt,
	\end{equation}
	where ${\H}^{n-1}$ is the $n-1$ dimensional Hausdorff measure on the Lipschitz boundary $\partial\Omega_\tau$.
\end{definition}

We consider solvability of the $L^p$ Dirichlet boundary value problem with respect to this measure $\sigma$.
The measure $\sigma$ may not be comparable to the usual surface measure on $\partial\Omega$: in the $t$-direction the functions $\phi_j$ from \cref{D:domain} are only $1/2$ Lipschitz and hence the standard surface measure might not be locally finite.
Our definition assures that for any $A\subset 8\Z_j$, where $\Z_j$ is an $\ell$-cylinder, we have
\begin{equation}
	\label{E:comp}
	\H^n(A) \sim \sigma\left(\{(\phi_j(x,t),x,t):\,(x,t)\in A\}\right),
\end{equation}
where the constants in~\cref{E:comp}, by which these measures are comparable, only depend on $\ell$ of the character $(\ell,\eta,N,d)$ of the domain $\Omega$.
If $\Omega$ has a smoother boundary, such as Lipschitz (in all variables) or better, then the measure $\sigma$ is comparable to the usual $n$-dimensional Hausdorff measure $\H^n$.
In particular, this holds for a parabolic cylinder $\Omega={\mathcal O}\times \R$.

\begin{corollary}%
	\label{P:ext}
	Let $\Omega$ be defined as in \cref{D:domain} by a family of functions $\{\phi_j\}$, $\phi_j : Q_{8d} \to \R$. Then there exists an extended family $\{\Phi_j\}$, $\Phi_j : \R^{n-1} \times \R \to \R$, such that
	\begin{property}
		\item $\left\{\left.\Phi_j\right|_{Q_{8r}} \right\}$ still describes $\Omega$, as in \cref{D:domain}, but with character $(\ell,\eta,\tilde N,r)$ instead of $(\ell,\eta, N,d)$, where $\tilde N \geq N$ and $r \leq r_1 \leq d$ is from \cref{T:ext};

		\item $\|\nabla\Phi_j\|_{*} \lesssim_\epsilon \eta^{1-\epsilon} + \eta\ell$, and

		\item $\|\D\Phi_j\|_{*} \lesssim_\epsilon \eta^{1-\epsilon} + \eta\ell$.
	\end{property}
\end{corollary}
\begin{proof}
	This follows from \cref{T:ext} and by tiling the support of each $\phi_j$ into parabolic cubes of size $8r$ with enough overlap.
\end{proof}

\begin{corollary}%
	\label{C:extension:VMO}
	If $\Omega$ is a $\VMO$ type domain then we may take $\eta$ arbitrarily small in \cref{P:ext}, or in~\cref{E:L2,E:domain:eta} of \cref{D:domain}, by reducing $r$.
\end{corollary}

\subsection{Pullback Transformation and Carleson Condition}%
\label{S:pullback}

We now briefly recall the pullback mapping of Dahlberg-Kenig-Ne\v{c}as-Stein on the upper half-space $U$ $\rho : U \to \Omega$ (c.f.~\cite{HL96,HL01}) in the setting of parabolic equations defined by
\begin{equation}
	\label{mapping}
	\rho (x_0, x, t) = (x_0 + P_{\gamma x_0}\phi(x,t), x, t).
\end{equation}
For simplicity assume
\begin{equation}
	\label{E:Omega}
	\Omega = \{ (x_0, x, t) \in \R\times \R^{n-1} \times \R : x_0 > \phi(x,t) \}
\end{equation}
where $\phi(x,t): \R^{n-1} \times \R \to \R$ and satisfies \cref{I:domain:eta,I:domain:lip} of \cref{D:domain}.
This transformation maps the upper half-space
\begin{equation}
	\label{E:U}
	U = \{ (x_0, x, t) : x_0 > 0, \ x \in \R^{n-1}, \ t \in \R \}
\end{equation}
into $\Omega$ and allows us to consider the $L^p$ solvability of the PDE~\cref{E:pde} in the upper half-space instead of in the original domain $\Omega$.

To complete the definition of the mapping $\rho$ we define a parabolic approximation to the identity $P$ to be
an even non-negative function $P(x,t) \in C_{0}^{\infty}(Q_{1}(0,0))$, for $(x,t) \in \R^{n-1}\times\R$, with $\int P(x,t)\dx\dt = 1$ and set
\[
	P_{\lambda} (x,t) := \lambda^{-(n+1)} P\left( \frac{x}{\lambda}, \frac{t}{\lambda^2} \right).
\]
Let $P_\lambda \phi$ be the convolution operator
\[
	P_{\lambda} \phi(x,t) := \int_{\R^{n-1}\times \R} P_{\lambda} (x-y, t-s) \phi(y,s) \dy\ds
\]
then $P$ satisfies for constants $\gamma$
\[
	\lim_{(y_0, y, s) \to (0, x, t)} P_{\gamma y_0} \phi(y,s) = \phi(x, t)
\]
and $\rho$ defined in~\eqref{mapping} extends continuously to $\rho:\overline{U}\to \overline{\Omega}$.
The usual surface measure on $\partial U$ is comparable with the measure $\sigma$ defined by~\cref{E:sigma} on $\partial\Omega$.

Suppose that $v = u \circ \rho$ and $f^v = f \circ \rho$ then~\cref{E:pde} transforms to a new PDE for the variable $v$
\begin{equation}
	\label{E:pullback}
	\begin{cases}
		v_t = \di (A^v \nabla v) + B^v \cdot \nabla v & \text{in } U, \\
		v   = f^v                                     & \text{on } \partial U,
	\end{cases}
\end{equation}
where $A^v = [a^v_{ij}(X, t)]$, $B^v = [b^v_i (X, t)]$ are $(n\times n)$ and $(1\times n)$ matrices.

The precise relations between the original coefficients $A$ and $B$ and the new coefficients $A^v$ and $B^v$ are detailed in~\cite[p.~448]{Riv14}.
We note that if the constant $\gamma > 0$ is chosen small enough then the coefficients $a^v_{ij}, b^v_i : U \to \R$ are Lebesgue measurable and $A^v$ satisfies the standard uniform ellipticity condition with constants $\lambda^v$ and $\Lambda^v$, since the original matrix $A$ did.

\begin{definition}
	Let $\Omega$ be a parabolic domain from \cref{D:domain}.
	For $(Y,s)\in\partial\Omega$, $(X,t), (Z, \tau) \in \Omega$ and $r>0$ we write:
	\begin{align*}
		B_r(X,t)      & =\{(Z,\tau)\in{\R}^{n}\times \R : \dist[(X,t),(Z,\tau)]<r \}, \\
		Q_{r} (X, t)  & = \{ (Z,\tau) \in \R^{n}\times\R : |x_i - z_i| < r \text{ for all } 0 \leq i \leq n-1, \, | t - \tau |^{1/2} < r \}, \\
		\Delta_r(Y,s) & = \partial \Omega\cap B_r(Y,s),	\quad T(\Delta_r) = \Omega\cap B_r(Y,s), \\
		\delta(X,t)   & =\inf_{(Y,s)\in \partial\Omega} \dist[(X,t),(Y,s)].
	\end{align*}
\end{definition}

\begin{definition}[Carleson measure]%
	\label{D:carl}
	A measure $\mu:\Omega\to\R^+$  is a \textit{Carleson measure} if there exists a constant $C=C(d)$ such that for all $r\leq d$ and all surface balls $\Delta_r$
	\begin{equation}
		\label{E:carl}
		\mu(T(\Delta_r)) \leq C \sigma (\Delta_r).
	\end{equation}
	The best possible constant $C$ is called the \textit{Carleson norm} and is denoted by $\|\mu\|_{C,d}$.
	Occasionally, for brevity, we drop the $d$ and just write $\|\mu\|_{C}$ if the context is clear. We say that $\mu$ is a vanishing Carleson measure if $\|\mu\|_{C,d}\to 0$ as $d\to 0+$.
\end{definition}

When $\partial\Omega$ is locally given as a graph of a function $x_0=\phi(x,t)$ in the coordinate system $(x_0,x,t)$ and $\mu$ is a measure supported on $\{x_0>\phi(x,t)\}$ we can reformulate the Carleson condition locally using the parabolic boundary cubes $Q_r$ and corresponding Carleson regions $T(Q_r)$.
The Carleson condition~\cref{E:carl} then becomes
\begin{equation}
	\label{E:carl:cube}
	\mu(T(Q_r)) \leq C|Q_r|=Cr^{n+1}.
\end{equation}
Note that the Carleson norms induced from~\cref{E:carl,E:carl:cube} are not equal but are comparable.

We now return back to the pullback transformation and investigate the Carleson condition on the coefficients of $A$ and $B$.
The following result comes directly from a careful reading of the proofs of Lemma 2.8 and Theorem 7.4 in~\cite{HL96} combined with \cref{T:equiv:known,T:equiv}.

\begin{lemma}%
	\label{L:LemmaA}

	Let $\sigma$ and $\theta$ be non-negative integers, $\alpha = (\alpha_1, \ldots, \alpha_{n-1})$ a multi-index with $l=\sigma + |\alpha| + \theta$, $d$ a scale and fix $\gamma$.
	If $\phi : \R^{n-1} \times \R \to \R$ satisfies for all $x, y \in \R^{n-1}$, $t,s \in \R$ and for some positive constants $\ell$ and $\eta$
	\[
		|\phi(x,t) - \phi(y, s)| \leq \ell\left(|x-y| + |t-s|^{1/2}\right),
	\]
	\begin{equation}
		\label{E:LemmaA}
		\|\D\phi\|_* \leq \eta
	\end{equation}
	then the measure $\nu$ defined at $(x_0, x, t)$ by
	\[
		\dnu = \left( \frac{\partial^{l} P_{\gamma x_0} \phi}{\partial x_{0}^{\sigma} \partial x^{\alpha} \partial t^{\theta}} \right)^2
		x_{0}^{2l + 2\theta - 3} \dx\dt\dx_0
	\]
	is a Carleson measure on cubes of diameter $\leq d/4$ whenever either $\sigma + \theta \geq 1$ or $|\alpha| \geq 2$, with
	\[
		\nu \left[ (0, r) \times Q_{r}(x,t) \right] \lesssim \eta \left| Q_{r} (x,t) \right|,
	\]
	where $r \le d/4$.
	Moreover, if $l \geq 1$ then at $(x_0, x, t)$, with $x_0 \leq d/4$,
	\begin{equation}
		\label{E:pullback-derivatives2}
		\left| \frac{\partial^{l} P_{\gamma x_0} \phi}{\partial x_{0}^{\sigma} \partial x^{\alpha} \partial t^{\theta}} \right|
		\lesssim \eta(1+ \ell)  x_{0}^{1-l-\theta},
	\end{equation}
	where the implicit constants depend on $d,l,n$.
\end{lemma}

The drift term $B^v$ from the pullback transformation in~\cref{E:pullback} includes the term
\[
	\frac{\partial}{\partial t} P_{\gamma x_0} \phi u_{x_0}.
\]
From \cref{L:LemmaA} with $\sigma = |\alpha| = 0$, $\theta = 1$, we see that
\[
	x_0 \left[ \frac{\partial}{\partial t} P_{\gamma x_0} \phi (x,t) \right]^2 \dX\dt
\]
is a Carleson measure in $U$.
Thus it is natural to expect that
\begin{equation}
	\label{E:pullback:B:carl}
	\dmu_1 (X,t) = x_0 |B^v|^2(X,t) \dX\dt
\end{equation}
is a Carleson measure in $U$ and $B^v$ satisfies
\begin{equation}
	\label{E:pullback:B:bound}
	x_0 |B^v| (X,t) \leq \Lambda_B < \|\mu_1\|_{C}^{1/2}.
\end{equation}
Indeed, this is the case provided the original vector $B$ satisfies the assumption that
\begin{equation}
	\label{E:pde:B:carl}
	\dmu(X,t) = \delta(X,t) \left[\sup_{B_{\delta(X,t)/2}(X,t)} |B| \right]^2 \dX\dt
\end{equation}
is a Carleson measure in $\Omega$.
Here $\|\mu_1\|_{C}$ depends on $\eta$ and the Carleson norm of~\cref{E:pde:B:carl}.

Similarly, for the matrix $A^v$ if we apply \cref{L:LemmaA} and use the calculations in~\cite[\S 6]{Riv14} then
\begin{equation}
	\label{E:pullback:A:carl}
	\dmu_2 (X,t) = (x_0 |\nabla A^v|^2 + x_0^3 |A^v_t|^2)(X,t) \dX\dt
\end{equation}
is a Carleson measure in $U$ and $A^v$ satisfies
\begin{equation}
	\label{E:pullback:A:bound}
	(x_0 |\nabla A^v| + x_0^2 |A^v_t|) (X,t) \leq \|\mu_2\|_{C}^{1/2}
\end{equation}
for almost everywhere $(X,t) \in U$ provided the original matrix $A$ satisfies that
\begin{equation}
	\label{E:pde:A:carl}
	\begin{split}
		&\dmu (X,t) = \\
		&\left(\delta(X,t) \left[\sup_{B_{\delta(X,t)/2}(X,t)}|\nabla A|\right]^2+\delta(X,t)^3 \left[\sup_{B_{\delta(X,t)/2}(X,t)}|\partial_t A|\right]^2\right)\dX\dt
	\end{split}
\end{equation}
is a Carleson measure in $\Omega$.

We note that if both $\|\mu\|_{C,r}$ and $\eta$ are small then so too are the Carleson norms $\|\mu_1\|_{C,r}$ and  $\|\mu_2\|_{C,r}$ of the matrix $A^v$ and vector $B^v$, at least if we restrict ourselves to small Carleson regions $r\le d$;
this comes from \cref{T:ext,P:ext,C:extension:VMO}.
Then by \cref{L:LemmaA} we see that $\|\mu_1\|_{C,r}$ and  $\|\mu_2\|_{C,r}$ only depend on $\eta$ and $\|\mu\|_{C,r}$ on Carleson regions of size $r \le d$.
In particular they are small if both $\eta$ and $\|\mu\|_{C,r}$ are small.
It further follows by \cref{C:extension:VMO} that we can make $\|\mu_1\|_{C,r}$ and $\|\mu_2\|_{C,r}$ as small as we like if $\mu$ is a vanishing Carleson norm and the domain $\Omega$ is of VMO type.

Observe that condition~\cref{E:pde:A:carl} is slightly stronger than~\cref{E:T1:carl:osc}, which we claimed to assume in \cref{T:1}.
We replace condition~\cref{E:pde:A:carl} by the weaker condition~\cref{E:T1:carl:osc} later via perturbation results of~\cite{Swe98}.


\begin{definition}%
	\label{D:rho}
	We define $\rho_j: U \to 8\Z_j$ to be the local pullback mapping in $8\Z_j$ associated to the function $\Phi_j$ in \cref{T:ext}, the extension of $\phi_j$ from \cref{D:domain}.
\end{definition}

\begin{remark}%
	\label{R:domain:smooth}
	By~\cite{BZ17} and its adaptation to the setting of admissible domains in~\cite[\S 2.3]{DH16}, one may construct a `proper generalised distance' globally when $\eta$ in the character of the domain is small.
	The smallness of $\eta$ in the character of the domain is used to guarantee that overlapping coordinate charts, generated by a local construction, are almost parallel.
	We may then use the result of~\cite[Theorem 5.1]{BZ17} to show there exists a domain $\Omega^\epsilon$ of class $C^\infty$, a homeomorphism $f^\epsilon : \overline{\Omega} \to \overline{\Omega^\epsilon}$ such that $f^\epsilon(\partial\Omega) = \partial\Omega^\epsilon$ and $f^\epsilon : \Omega \to \Omega^\epsilon$ is a $C^\infty$ diffeomorphism.
\end{remark}

\subsection{Parabolic Non-tangential Cones, Maximal Functions and \texorpdfstring{$p$}{p}-adapted Square and Area Functions}

We proceed with the definition of parabolic non-tangential cones and define the cones in a (local) coordinate system where $\Omega=\{(x_0,x,t) : x_0>\phi(x,t)\}$, which also applies to the upper half-space $U$.

\begin{definition}%
	\label{D:cone}
	For a constant $a>0$, we define the \textit{parabolic non-tangential cone} at a point $(x_0,x,t)\in\partial\Omega$ as follows
	\begin{equation*}
		\label{E:cone}
		\Gamma_{a}(x_0, x, t) = \left\{(y_0, y, s)\in \Omega : |y - x| + |s-t|^{1/2} < a(y_0 - x_0), \  x_0 < y_0 \right\}.
	\end{equation*}
	We occasionally truncate the cone $\Gamma$ at the height $r$
	\begin{equation*}
		\label{E:cone:r}
		\Gamma_{a}^{r}(x_0, x, t) = \left\{(y_0, y, s)\in \Omega : |y - x| + |s-t|^{1/2} < a(y_0 - x_0), \  x_0 < y_0 < x_0 + r  \right\}.
	\end{equation*}
\end{definition}

\begin{definition}[Non-tangential maximal function]
	For a function $u : \Omega \rightarrow \R$, the \textit{non-tangential maximal function} $N_a(u): \partial\Omega \to \R$ and its truncated version at a height $r$ are defined as
	\begin{equation}
		\label{D:NTan}
		\begin{split}
			N_{a}(u)(x_0,x,t) &= \sup_{(y_0,y,s)\in \Gamma_{a}(x_0,x,t)} \left|u(y_0 , y, s)\right|, \\
			N_{a}^{r}(u)(x_0,x,t) &= \sup_{(y_0,y,s)\in \Gamma_{a}^{r}(x_0,x,t)} \left|u(y_0 , y, s)\right|\quad\mbox{for }(x_0,x,t)\in\partial\Omega.
		\end{split}
	\end{equation}
\end{definition}

The following $p$-adapted square function was introduced in~\cite{DPP07} and has been modified appropriately for the parabolic setting.
It is used to control the spatial derivatives of the solution.
When $p = 2$ it is equivalent to the usual square function and when $p < 2$ we use the convention that the expression $|\nabla u|^2 |u|^{p-2}$ is zero whenever $\nabla u$ vanishes.

\begin{definition}[$p$-adapted square function]%
	\label{D:Sp}
	For a function $u : \Omega \rightarrow \R$, the \textit{$p$-adapted square function} $S_{p,a}(u): \partial\Omega \to \R$ and its truncated version at a height $r$ are defined as
	\begin{equation}
		\label{E:Sp}
		\begin{split}
			S_{p,a} (u)(Y,s) &= \left(\int_{\Gamma_a(Y,s)}|\nabla u(X,t)|^2 |u(X,t)|^{p-2} \delta(X,t)^{-n} \dX \dt\right)^{1/p} \hspace{-1em}, \\
			S_{p,a}^r (u)(Y,s) &= \left(\int_{\Gamma^r _a(Y,s)}|\nabla u(X,t)|^2 |u(X,t)|^{p-2} \delta(X,t)^{-n} \dX \dt \right)^{1/p} \hspace{-1em}.
		\end{split}
	\end{equation}
	By applying Fubini we also have
	\begin{equation}
		\label{E:Sp:norm}
		\|S_{p,a} (u)\|^p_{L^p(\partial U)} \sim \int_{U} |\nabla u|^2 |u|^{p-2} x_0 \dx_0 \dx \dt.
	\end{equation}

	It is not know a priori if these integrals are locally integrable even for $p > 2$.
	However, \cref{T:improved-reg} shows that these expressions makes sense and are finite for solutions to~\cref{E:pde}.
\end{definition}

We also need a $p$-adapted version of an object called the area function which was introduced in~\cite{DH16} and is used to control the solution in the time variable.
Again when $p=2$ this is just the area function of~\cite{DH16}.
\begin{definition}[$p$-adapted area function]%
	\label{D:Ap}
	For a function $u : \Omega \rightarrow \R$, the \textit{$p$-adapted area function} $A_{p,a}(u): \partial\Omega \to \R$ and its truncated version at a height $r$ are defined as
	\begin{equation}
		\label{E:Ap}
		\begin{split}
			A_{p,a} (u) (Y,s) &= \left(\int_{\Gamma_a(Y,s)}|u_t|^2 |u(X,t)|^{p-2} \delta(X,t)^{2-n} \dX \dt \right)^{1/p} \hspace{-1em}, \\
			A_{p,a}^r (u) (Y,s) &= \left(\int_{\Gamma^r_a(Y,s)}|u_t|^2 |u(X,t)|^{p-2} \delta(X,t)^{2-n} \dX \dt \right)^{1/p} \hspace{-1em}.
		\end{split}
	\end{equation}
	Also by Fubini
	\begin{equation}
		\label{E:Ap:norm}
		\|A_{p,a} (u)\|^p_{L^p(\partial U)} \sim \int_{U} |u_t|^2 |u|^{p-2} x_0^3 \dx_0 \dx \dt.
	\end{equation}
	As before, it is not known a priori if these expressions are finite for solutions to~\cref{E:pde} but in \cref{L:A<S} we establish control of $A_{p,a}$ by $S_{p,2a}$ and use the finiteness of $S_{p,a}$ from \cref{T:improved-reg}.
\end{definition}

\subsection{\texorpdfstring{$L^p$}{Lp} Dirichlet Boundary Value Problem}

We are now in the position to define the $L^p$ Dirichlet boundary value problem.

\begin{definition}[\cite{Aro68}]%
	\label{D:weak}
	We say that $u$ is a \textit{weak solution} to a parabolic operator of the form~\cref{E:pde} in $\Omega$ if $u, \nabla u \in L^2_{\loc}(\Omega)$, $\sup_t \|u(\cdot, t)\|_{L^2_{\loc} (\Omega_t)} < \infty$ and
	\begin{equation*}
		\int_\Omega (-u \phi_t + A\nabla u \cdot \nabla \phi - \phi B \cdot \nabla u ) \dX\dt = 0
	\end{equation*}
	for all $\phi \in C^\infty_0(\Omega)$.
\end{definition}

\begin{definition}%
	\label{D:Dp}
	We say that the $L^p$ \textit{Dirichlet problem} with boundary data in $L^p(\partial\Omega, \dsig)$ is solvable if the unique solution $u$ to~\cref{E:pde} for any continuous boundary data $f$ decaying to $0$ as $t \to \pm\infty$ satisfies the following non-tangential maximum function estimate
	\begin{equation}
		\label{E:Dp}
		\|N(u)\|_{L^p(\partial\Omega, \dsig)} \lesssim \|f\|_{L^p(\partial\Omega, \dsig)},
	\end{equation}
	with the implied constant depending only on the operator, $n$, $p$ and $\Omega$.
\end{definition}

\section{Basic Results and Interior Estimates}

We now recall some foundational estimates that will be used.
The following result is from~\cite{DH16}, which was adapted from the elliptic result in~\cite{Din02}.
\begin{lemma}%
	\label{L:N:apertures}
	Let $r>0$ and $0< a < b$.
	Consider the non-tangential maximal functions defined using two set of cones cones $\Gamma_{a}^{r}$ and $\Gamma_{b}^{r}$.
	Then for any $p > 0$ there exists a constant $C_p > 0$ such that for all $u : U \to \R$
	\[
		N_{a}^{r} (u) \leq N_{b}^{r} (u)
		\quad\text{and}\quad
		\|N_{b}^{r}(u)\|_{L^p(\partial U)}\leq C_p\|N_{a}^{r}(u)\|_{L^p(\partial U)}.
	\]
\end{lemma}

\begin{lemma}[{A Cacciopoli inequality, see~\cite{Aro68}}]%
	\label{L:Caccio}
	Let $A$ and $B$ satisfy~\cref{E:elliptic,E:pullback:B:bound} and suppose that $u$ is a weak solution of~\cref{E:pde} in $Q_{4r}(X,t)$ with $0 < r < \delta(X,t)/8$.
	Then there exists a constant $C=C(\lambda,\Lambda, n)$ such that
	\begin{equation*}
		\begin{split}
			r^{n} \left(\sup_{Q_{r/2}(X,t)} u \right)^{2}
			&\leq C \sup_{t-r^2 \leq s \leq t+r^2} \int_{Q_r(X,t) \cap \{t = s\}} u^{2}(Y,s) \dY
			+ C\int_{Q_{r}(X, t)} |\nabla u|^{2} \dY\ds \\
			&\leq \frac{C^2}{r^2} \int_{Q_{2r}(X, t)} u^{2}(Y, s) \dY\ds.
		\end{split}
	\end{equation*}
\end{lemma}

Lemmas 3.4 and 3.5 in~\cite{HL01} give the following estimates for weak solutions of~\cref{E:pde}.

\begin{lemma}[Interior H\"{o}lder continuity]%
	\label{L:int_Holder}
	Let $A$ and $B$ satisfy~\cref{E:elliptic,E:pullback:B:bound} and suppose that $u$ is a weak solution of~\cref{E:pde} in $Q_{4r}(X,t)$ with $0 < r < \delta(X,t)/8$.
	Then for any $(Y,s), (Z,\tau) \in Q_{2r}(X,t)$
	\[
		\left|u(Y, s) - u(Z, \tau)\right| \leq C \left( \frac{|Y-Z| + |s - \tau|^{1/2}}{r}\right)^{\alpha} \sup_{Q_{4r}(X,t)} |u|,
	\]
	where $C=C(\lambda, \Lambda, n)$, $\alpha=\alpha(\lambda,\Lambda,n)$, and $0 < \alpha < 1$.
\end{lemma}

\begin{lemma}[Harnack inequality]%
	\label{L:Harnack}
	Let $A$ and $B$ satisfy~\cref{E:elliptic,E:pullback:B:bound} and suppose that $u$ is a weak non-negative solution of~\cref{E:pde} in $Q_{4r}(X,t)$, with $0 < r < \delta(X,t)/8$.
	Suppose that $(Y,s), (Z,\tau) \in Q_{2r}(X,t)$ then there exists $C=C(\lambda, \Lambda, n)$ such that, for $\tau < s$,
	\[
		u(Z, \tau) \leq u(Y, s) \exp \left[ C\left( \frac{|Y-Z|^2}{|s-\tau|} + 1 \right) \right].
	\]
\end{lemma}

We state a version of the maximum principle from~\cite{DH16} that is a modification of~\cite[Lemma 3.38]{HL01}.
\begin{lemma}[Maximum Principle]%
	\label{L:MP}
	Let $A$ and $B$ satisfy~\cref{E:elliptic,E:pullback:B:bound}, and let $u$ and $v$ be bounded continuous weak solutions to~\cref{E:pde} in $\Omega$.
	If $|u|,|v|\to 0$ uniformly as $t \to -\infty$ and
	\[
		\limsup_{(Y,s)\to (X,t)} (u-v)(Y,s) \leq 0
	\]
	for all  $(X,t)\in\partial\Omega$, then $u\leq v$ in $\Omega$.
\end{lemma}

\begin{remark}[\cite{DH16}]%
	\label{R:max}
	The proof of Lemma 3.38 from~\cite{HL01} works given the assumption that $|u|,|v|\to 0$ uniformly as $t\to-\infty$.
	Even with this additional assumption,  the lemma as stated is sufficient for our purposes.
	We shall mostly use it when $u \leq v$ on the boundary of $\Omega\cap\{t\ge \tau\}$ for a given time $\tau$.
	Obviously then the assumption that $|u|,|v|\to 0$ uniformly as $t\to-\infty$ is not necessary.
	Another case when the lemma as stated here applies is when $u|_{\partial\Omega},v|_{\partial\Omega}\in C_0(\partial\Omega)$, where $C_0(\partial\Omega)$ denotes the class of continuous functions decaying to zero as $t\to\pm\infty$.
	This class is dense in any $L^p(\partial\Omega,d\sigma)$, $1<p<\infty$ allowing us to consider an extension of the solution operator from $C_0(\partial\Omega)$ to $L^p$.
\end{remark}

\section{Improved Regularity for \texorpdfstring{$p$}{p}-adapted square function}
Here we extend recent work of~\cite{DP16} for complex coefficient elliptic equations to the real parabolic setting.
The goal is to obtain a improved regularity result for weak solutions of~\eqref{E:pde} implying that $|\nabla u|^2|u|^{p-2}$ belongs to $L^1_{loc}(\Omega)$ when $1<p<2$.
Having this it follows that the $p$-adapted square function $S_{p,a}$ is well defined at almost every boundary point.

\begin{theorem}[{c.f.~\cite[Theorem 1.1]{DP16}}]%
	\label{T:improved-reg}
	Suppose $u \in W^{1,2}_{\loc}(\Omega)$ is a weak solution to $\L u = u_t$, where $\L u = \di(A\nabla u) + B\nabla u$, $A$ is bounded and elliptic, $B$ is locally bounded and satisfies
	\begin{equation}
		\label{E:B}
		\delta(X,t) |B(X,t)| \leq K
	\end{equation}
	for some uniform constant $K>0$.
	Then for any parabolic ball $B_{4r}(X,t) \subset \Omega$ and $p,q \in (1,\infty)$  we have the following improvement in regularity
	\begin{equation}
		\label{E:improved-reg}
		\left(\fint_{B_r(X,t)} |u|^p \right)^{1/p}
		\le C_\epsilon \left(\fint_{B_{2r}(X,t)} |u|^q \right)^{1/q} +  \epsilon\left(\fint_{B_{2r}(X,t)} |u|^2 \right)^{1/2} \hspace{-1em}.
	\end{equation}
	Here the constant $C_\epsilon$ only depends on $p$, $q$, $\epsilon$, $n$, $\lambda$, $\Lambda$, and $K$ but not on $u$, $(X,t)$ or $r$.
	In addition, for all $1 < p < \infty$
	\begin{equation}
		\label{E:p-adapted-bound}
		r^2 \fint_{B_r(X,t)} |\nabla u|^2 |u|^{p-2}
		\leq C_\epsilon \fint_{B_{2r}(X,t)} |u|^p + \epsilon\left(\fint_{B_{2r}(X,t)} |u|^2 \right)^{p/2} \hspace{-1em},
	\end{equation}
	where again the constant only depends on $\epsilon$, $p$, $n$, the ellipticity constants of $A$, and $K$.
	This also shows that $|u|^{(p-2)/2}\nabla u \in L^2_{\loc}(\Omega)$.
\end{theorem}

\begin{remark}
	If $q \geq 2$ in~\cref{E:improved-reg} or if $p \geq 2$ in~\cref{E:p-adapted-bound} then one can take $\epsilon = 0$ because the $L^2$ averages can be controlled by the first term on the right hand side of these inequalities.
\end{remark}

We focus only on the case $1<p<2$ as the $p\ge 2$ result above follows from the Cacciopoli inequality, \cref{L:Caccio}.
We shall establish the following lemma for the $1<p<2$ case which concludes the proof of \cref{T:improved-reg}.

\begin{lemma}[{c.f.~\cite[Lemma 2.7]{DP16}}]%
	\label{L:2.7}
	Let $u$ be a weak solution to $\L u = u_t$ in $\Omega$ for $A$ elliptic and bounded, and $B$ bounded satisfying~\cref{E:B}.
	Then for any $p <2$, any ball $B_r(X,t)$ with $r < \delta(X,t)/4$, and any $\epsilon > 0$
	\begin{equation}
		\label{E:2.7:square}
		r^2 \int_{B_r(X,t)} |\nabla u|^2 |u|^{p-2}
		\leq C_\epsilon \fint_{B_{2r}(X,t)} |u|^p + \epsilon\left(\fint_{B_{2r}(X,t)} |u|^2 \right)^{p/2}
	\end{equation}
	and
	\begin{equation}
		\label{E:2.7}
		\left(\fint_{B_{r}(X,t)} |u|^2 \right)^{1/2}
		\leq C_\epsilon \left( \fint_{B_{2r}(X,t)} |u|^p\right)^{1/p}  + \epsilon\left(\fint_{B_{2r}(X,t)} |u|^2 \right)^{1/2} \hspace{-1em},
	\end{equation}
	where the constants only depend on $n$, $\epsilon$, $\lambda$, $\Lambda$ and $K$.
	In particular, $|u|^{(p-2)/2}\nabla u \in L^2_{\loc}(\Omega)$.
\end{lemma}

\begin{proof}
	We start by assuming that $A$ and $B$ are smooth then the solution $u$ to $\L u = u_t$ is smooth.
	We prove the above inequalities with constants that do not depend on the smoothness of $A$ or $B$ and then remove the smoothness assumption at the end of the proof via the method of~\cite{HL01}.
	To simplify notation we suppress the argument of the ball $B_r(X,t)$.

	Let
	\begin{equation}
		\rho_\delta(s) = \begin{cases}
			\delta^{(p-2)/2} & 0 \leq s \leq \delta \\
			s^{(p-2)/2}      & s > \delta.
		\end{cases}
	\end{equation}
	The choice of cut off function $\rho_\delta$ in this proof is inspired by~\cite[p.~311]{Lan99}, \cite[p.~1088]{CM05}.
	We multiply $\L u = u_t$ by $\rho_\delta^2(|u|)u$ and integrate by parts to obtain
	\begin{equation}{}
		\label{E:2.7:Lu}
		\begin{split}
			\int_{B_r} \nabla \left(\rho_\delta^2(|u|)u \right) A \nabla u
			&= \int_{B_r} \rho_\delta^2(|u|)u u_t
			+ \int_{B_r} \rho_\delta^2(|u|) B \cdot \nabla u \\
			&\quad + \int_{\partial B_r} (\rho_\delta^2(|u|))\nu \cdot A \nabla u \dsig(y,s),
		\end{split}
	\end{equation}
	where $\nu$ is the outer unit normal to $B_r$.
	Consider $E_\delta = \{u > \delta \} $ then the left hand side of~\cref{E:2.7:Lu} is
	\begin{equation}
		\label{E:2.7:E_delta}
		\int_{B_r} \nabla \left(\rho_\delta^2(|u|)u \right) A \nabla u
		= \delta^{p-2} \int_{B_r \setminus E_\delta} \nabla u \cdot A \nabla u
		+ \int_{B_r \cap E_\delta} A \nabla u \cdot \nabla \left(|u|^{p-2}u\right)
	\end{equation}
	and by ellipticity of $A$ on the open set $B_r \cap E_\delta$ we have for some $\lambda'>0$
	\begin{equation}
		\label{E:2.7:E_delta:bound}
		\lambda' \int_{B_r \cap E_\delta} |u|^{p-2} |\nabla u|^2
		\leq \int_{B_r \cap E_\delta} A \nabla u \cdot \nabla \left(|u|^{p-2}u\right).
	\end{equation}
	Our strategy is to let $\delta \to 0$ and show all the integrals involving $B_r \setminus E_\delta$ tend to 0.

	First, we use the following result from~\cite{Lan99}.
	They proved if $u \in C^2\left(\overline{B_r}\right)$ and $u = 0 $ on $\partial B_r$ then for $q > -1$
	\begin{equation}
		\label{E:Lan99}
		\lim\limits_{\delta \to 0} \delta^q \int_{B_r \setminus E_\delta} |\nabla u|^2 = 0.
	\end{equation}

	To deal with the boundary integral in~\cref{E:2.7:Lu} we note that equations~\cref{E:2.7:Lu,E:2.7:E_delta,E:2.7:E_delta:bound} remain valid for any enlarged ball $B_{\alpha r}$ for $1 \leq \alpha \leq 5/4$.
	We write~\cref{E:2.7:Lu} for every $B_{\alpha r}$ and then average in $\alpha$ over the interval $[1,5/4]$.
	The last term in~\cref{E:2.7:Lu} then turns into a solid integral over $B_{5r/4} \setminus B_r$.
	Therefore,
	\begin{equation*}
		\begin{split}
			\lambda' \int_{B_r \cap E_\delta} |u|^{p-2} |\nabla u|^2
			&\leq \sup\limits_{\alpha \in [1,5/4]} \left|\int_{B_{\alpha r}} \rho^2_\delta \left(|u|\right)u u_t \right|
			+ \sup\limits_{\alpha \in [1,5/4]} \left|\int_{B_{\alpha r}} \rho^2_\delta \left(|u|\right)u B \cdot \nabla u \right| \\
			&\quad + \left| r^{-1} \int_{B_{5\alpha r/4}\setminus B_r} \rho^2_\delta\left(|u|\right)u \nu \cdot A \nabla u \right|
			+ o(1) \\
			&= I + II + III + o(1),
		\end{split}
	\end{equation*}
	where $o(1)$ contains the integral over $B_{\alpha r} \setminus E_\delta$, which tends to $0$ as $\delta \to 0$.
	We bound $II$ and $III$ as~\cite{DP16}
	\begin{equation*}
		II + III \leq C_\epsilon r^{-2}\int_{B_{5r/4}} |u|^p
		+ \epsilon r^{p-2} \int_{B_{5r/4}} |\nabla u|^p
		+o(1).
	\end{equation*}

	Now we turn to $I$ and use the same idea as the proof of~\cref{E:Lan99} in~\cite[(3.3)]{Lan99} to show $I$ converges as expected.
	By splitting the integral with the set $E_\delta$, using the fact $\delta^{p-2} \leq |u|^{p-2}$ on $B_{\alpha r} \setminus E_\delta$ (since $p < 2$), and the smoothness of $u$, which implies $|u|^{p-2}uu_t \in L^1(B_{\alpha r})$, we obtain
	\begin{equation*}
		\begin{split}
			\int_{B_{\alpha r}} \rho^2_\delta \left(|u|\right)u u_t
			&= \int_{B_{\alpha r} \cap E_\delta} |u|^{p-2} u u_t + \delta^{p-2} \int_{B_{\alpha r} \setminus E_\delta} u u_t \\
			&\leq \int_{B_{\alpha r} \cap E_\delta} |u|^{p-2} u u_t + \int_{B_{\alpha r} \setminus E_\delta} |u|^{p-2} u u_t \\
			&\leq \int_{B_{\alpha r}} |u|^{p-1} |u_t|
			< \infty.
		\end{split}
	\end{equation*}
	Therefore by the dominated convergence theorem
	\begin{equation}
		\int_{B_{\alpha r}} \rho^2_\delta \left(|u|\right)u u_t
		\to \int_{B_{\alpha r}} |u|^{p-2} u u_t.
	\end{equation}

	We change from working with balls to integrating over parabolic cubes $Q_{\alpha r}$ and denote by $Q_{\alpha r}|_{s}$ the cube $Q_{\alpha r}$ restrict to the hypersurface $\{t=s\}$.
	Using the fundamental theorem of calculus we obtain in the limit that
	\begin{equation}
		\label{E:2.7:Lp-restricted}
		\begin{split}
			\int_{B_{\alpha r}} |u|^{p-2} u u_t
			&\sim \int_{B_{\alpha r}} \frac{\partial}{\partial t} \left(|u|^p \right) \dt \dX \\
			&\leq \int_{Q_{\alpha r}} \frac{\partial}{\partial t} \left(|u|^p \right) \dt \dX
			= \int_{t_0 - (\alpha r)^2}^{t_0 + (\alpha r)^2} \frac{\mathrm{d}}{\mathrm{d} t} \int_{Q_{\alpha r}|_{s}} |u|^p \dX \ds \\
			&\leq \|u\|^p_{L^p_X \left( Q_{\alpha r} \left|_{t_0 + (\alpha r)^2} \right.\right)} + \|u\|^p_{L^p_X \left( Q_{\alpha r} \left|_{t_0 - (\alpha r)^2}\right.\right)}.
		\end{split}
	\end{equation}
	Observe that~\cref{E:2.7:Lp-restricted} holds for all time restricted cubes $Q_{\alpha r}| _{t_0 \pm (\alpha r)^2}$ with $\alpha \in [1,1.1]$.
	Once again we average over these cubes to show
	\begin{equation*}
		\begin{split}
			\int_{B_{\alpha r}} |u|^{p-2} u u_t
			&\lesssim \frac{1}{r^2} \int_{Q_{1.1\alpha r}} |u|^p \dX \dt.
		\end{split}
	\end{equation*}
	Since $Q_{1.1\alpha r}\subset B_{2r}$, in the limit as $\delta \to 0$
	\begin{equation*}
		I \lesssim \frac{1}{r^2} \int_{B_{2r}} |u|^p \dX \dt.
	\end{equation*}

	Therefore grouping the estimates together we have the following bound
	\begin{equation}
		\lambda' \int_{B_r \cap E_\delta} |u|^{p-2} |\nabla u|^2
		\lesssim C_\epsilon r^{-2}\int_{B_{2r}} |u|^p
		+ \epsilon r^{p-2} \int_{B_{5r/4}} |\nabla u|^p + o(1).
	\end{equation}
	We let $\delta \to 0$ and proceed as~\cite{DP16} to obtain~\cref{E:2.7:square,E:2.7} for smooth $A$ and $B$.
	Finally, since no constants depend on the smoothness of $A$ or $B$, we can remove the smoothness assumption by the same argument as in~\cite{HL01}.
	We suppose $A$ is just elliptic and bounded, and $B$ satisfies~\cref{E:B} then we approximate $A$ and $B$ by smooth matrices and vectors respectively.
	For each smooth approximation we have~\cref{E:2.7:square,E:2.7} and then passing to the limit we obtain analogous estimates for $W^{1,2}_{\loc}$ solutions $u$ of $\L u = u_t$ with the constants having the same dependence as before.
\end{proof}

It follows that the $p$-adapted square function $S_{p,a}$ is well defined.
\cite{DH16} also considered an area function and established~\cite[Lemma 5.2]{DH16} that this area function can be controlled by the usual square function. 
The case $1<p<2$ is significantly more complicated so for this reason we focus only on non-negative solutions $u$.

We fix a boundary point $(Y,s) \in \partial\Omega$ and consider $A_{p,a}(Y,s)$.
Clearly, the non-tangential cone $\Gamma_a(Y,s)$ can be covered by non-overlapping collection of Whitney cubes $\{Q_i\}$ with the following properties:
\begin{equation}%
	\label{Wh}
	\Gamma_a(Y,s)\subset\bigcup_i Q_i \subset \Gamma_{2a}(Y,s),
	\quad r_i:=\diam(Q_i) \sim \dist(Q_i,\partial\Omega),
	\quad 4Q_i\subset\Omega,
\end{equation}
and the cubes $\{2Q_i\}$ having only finite overlap.
It follows that
\begin{align}%
	\label{eq-AAAA}
	\left[A_{p,a}(Y,s)\right]^p
	 & \lesssim \sum_i (r_i)^{2-n}\int_{Q_i}|u_t|^2u^{p-2} \dX\dt \\
	 & \lesssim  \sum_i (r_i)^{2-n}\int_{Q_i}|\nabla^2 u|^2u^{p-2} + \left(|\nabla A|^2 + |B|^2\right)|\nabla u|^2u^{p-2}\dX\dt. \notag
\end{align}
We need the following estimate on each $Q_i$.

\begin{lemma}
	Assume the ellipticity condition~\eqref{E:elliptic} and that the coefficients $A$ and $B$ of~\eqref{E:pde} satisfy the conditions
	\[
		|\nabla A(X,t)| \leq K/\delta(X,t) \quad\text{and}\quad |B(X,t)| \leq K/\delta(X,t),
	\]
	for some uniform constant $K > 0$.
	Then for all non-negative solutions $u$ of~\eqref{E:pde} and any parabolic cube $Q$ such that $4Q \subset \Omega$ we have the following estimate
	\begin{equation}
		\label{est-repl}
		\int_{Q}|\nabla^2 u|^2u^{p-2}\dX\dt \lesssim r^{-2}\int_{2Q}|\nabla u|^2u^{p-2}\dX\dt,
	\end{equation}
	where $r=\diam(Q)$.
\end{lemma}

\begin{proof}
	Since we assume differentiability of the matrix $A$ in the spatial variables we may also assume that $A$ is symmetric.
	Let us denote by $W=(w_k)$, where $w_k=\partial_k u$ for $k=0,1,\dots,n-1$.
	Differentiating~\eqref{E:pde} we obtain the following PDE for each $w_k$
	\begin{equation}
		\label{eq-v}
		(w_k)_t-\di (A \nabla w_k) = \di ((\partial_k A)W)+\partial_k(B \cdot W).
	\end{equation}

	We multiply~\eqref{eq-v} by $w_k u^{p-2}\zeta^2$, integrate over $2Q$ and integrate by parts.
	Here $0\le\zeta\le 1$ is a smooth cut-off function equal to 1 on $Q$, vanishing outside $2Q$ and satisfying $r|\nabla\zeta|+r^2|\zeta_t| \le C$ for some $C>0$ independent of $Q$.
	This gives
	\begin{align}
		\label{eq:PDEwk}
		 & \int_{2Q}(w_k)_tw_ku^{p-2}\zeta^2\dX\dt+\int_{2Q}a_{ij}(\partial_j w_k)\partial_i(w_ku^{p-2}\zeta^2)\dX\dt \notag \\
		 & = -\int_{2Q}(\partial_ka_{ij})w_j\partial_i(w_ku^{p-2}\zeta^2) \dX\dt
		- \int_{2Q}b_i w_i\partial_k(w_ku^{p-2}\zeta^2)\dX\dt.
	\end{align}
	We rearrange and group similar terms together
	\begin{equation}
		\label{eq:rearg}
		\begin{split}
			& \frac12\int_{2Q} \left[(w_ku^{p/2-1}\zeta)^2\right]_t \dX\dt - \frac{p-2}2\int_{2Q} w_k^2u^{p-3}u_t\zeta^2 \dX\dt \\
			& \quad +\int_{2Q}A\left(\nabla(w_k\zeta)u^{p/2-1}\right) \cdot \left(\nabla(w_k\zeta)u^{p/2-1}\right) \dX\dt  \\
			& \quad +(p-2)\int_{2Q} A\left(\nabla(w_k\zeta)u^{p/2-1}\right) \cdot \left((\nabla u)w_ku^{p/2-2}\zeta \right) \dX\dt  \\
			& =\int_{2Q} |w_k|^2u^{p-2}\zeta\zeta_t\dX\dt + \int_{2Q} |w_k|^2u^{p-2} A\nabla\zeta \cdot \nabla\zeta \dX\dt  \\
			& \quad - \int_{2Q} b_iw_i\partial_k(w_k\zeta)u^{p-2}\zeta\dX\dt \\
			& \quad - (p-2)\int_{2Q} b_iw_i \left((\partial_k u)w_ku^{p/2-2}\zeta\right) u^{p/2-1}\zeta \dX\dt  \\
			& \quad - \int_{2Q} b_iw_iw_ku^{p-2}\zeta\zeta_k\dX\dt-\int_{2Q}(\partial_k a_{ij})w_jw_ku^{p-2}\zeta\zeta_i\dX\dt  \\
			& \quad - \int_{2Q} (\partial_k a_{ij}) w_j(\partial_i w_k\zeta)u^{p-2}\zeta\dX\dt  \\
			& \quad - (p-2)\int_{2Q}(\partial_k a_{ij}) w_j\left((\partial_iu) w_ku^{p/2-2}\zeta\right) u^{p/2-1}\zeta\dX\dt.
		\end{split}
	\end{equation}
	All the terms after the equal sign are ``error'' terms since they either contain a derivative of $\zeta$, or coefficients $\nabla A$ or $B$.
	These will be handled using the Cauchy-Schwartz inequality and the estimates for $|\nabla A|, |B|\le K/r$.
	The four main terms are on the left hand side of~\eqref{eq:rearg}.
	The term that needs further work is the second term and we use the PDE~\eqref{E:pde} for $u_t$.
	This gives
	\begin{equation}%
		\label{eq:rearg2}
		\begin{split}
			-\frac{p-2}2\int_{2Q}w_k^2u^{p-3}u_t\zeta^2\dX\dt
			& = -\frac{p-2}2\int_{2Q}w_k^2u^{p-3}\di(A\nabla u)\zeta^2\dX\dt \\
			&\quad- \frac{p-2}2\int_{2Q}w_k^2u^{p-3}B\cdot W\zeta^2\dX\dt.
		\end{split}
	\end{equation}
	Again the second term will be an ``error'' term.
	For the first term we observe the equality
	\[
		u^{p-3}\di(A\nabla u) = \di(A(\nabla u)u^{p-3}) - (p-3)A((\nabla u)u^{p/2-2})\cdot((\nabla u)u^{p/2-2}).
	\]
	It follows (by integrating by parts)
	\begin{equation}
		\label{eq:rearg3}
		\begin{split}
			& - \frac{p-2}2\int_{2Q}w_k^2u^{p-3}\di(A\nabla u)\zeta^2\dX\dt  \\
			& = (p-2)\int_{2Q} A(\nabla(w_k\zeta)u^{p/2-1}) \cdot ((\nabla u)w_ku^{p/2-2}\zeta)\dX\dt \\
			& \quad +\frac{(2-p)(3-p)}2 \int_{2Q}A((\nabla u)w_ku^{p/2-2}\zeta)\cdot ((\nabla u)w_ku^{p/2-2}\zeta)\dX\dt .
		\end{split}
	\end{equation}
	We now group all main terms together; these are the first, second and fourth terms on the left-hand side of~\eqref{eq:rearg} and the terms of~\eqref{eq:rearg3}.
	This gives
	\begin{equation}%
		\label{eq:rearg4}
		\begin{split}
			& \text{LHS of~\eqref{eq:rearg}}
			= \frac12 \int_{2Q} \left[\left(w_ku^{p/2-1}\zeta\right)^2\right]_t \dX\dt \\
			& \quad + \int_{2Q} A\left(\nabla(w_k\zeta)u^{p/2-1}\right) \cdot \left(\nabla(w_k\zeta)u^{p/2-1}\right) \dX\dt \\
			& \quad + 2(p-2)\int_{2Q} A\left(\nabla(w_k\zeta)u^{p/2-1}\right) \cdot \left((\nabla u)w_ku^{p/2-2}\zeta\right) \dX\dt \\
			& \quad + \frac{(2-p)(3-p)}2 \int_{2Q} A\left((\nabla u)w_ku^{p/2-2}\zeta\right) \cdot \left((\nabla u)w_ku^{p/2-2}\zeta\right) \dX\dt \\
			& = \frac12\int_{2Q} \left[\left(w_ku^{p/2-1}\zeta\right)^2\right]_t \dX\dt \\
			& \quad + \left(1-\frac{2(2-p)}{3-p}\right) \int_{2Q} A\left(\nabla(w_k\zeta)u^{p/2-1}\right) \cdot \left(\nabla(w_k\zeta)u^{p/2-1}\right) \dX\dt \\
			& \quad + \int_{2Q} A\left(\textstyle\sqrt{\frac{2(2-p)}{3-p}} \left[\nabla(w_k\zeta)u^{p/2-1}\right]
			- \sqrt{\frac{(2-p)(3-p)}2} \left[(\nabla u)w_ku^{p/2-2}\zeta\right]\right)  \\
			& \qquad \cdot \left(\textstyle\sqrt{\frac{2(2-p)}{3-p}} \left[\nabla(w_k\zeta)u^{p/2-1}\right]
			- \sqrt{\frac{(2-p)(3-p)}2} \left[(\nabla u)w_ku^{p/2-2}\zeta\right]\right) \dX\dt \\
			& \geq \frac12 \int_{2Q} \left[(w_ku^{p/2-1}\zeta)^2\right]_t \dX\dt
			+ \frac{(p-1)\lambda}{3-p} \int_{2Q} \left| \nabla(w_k \zeta) u^{p/2-1} \right|^2 \dX\dt.
		\end{split}
	\end{equation}
	Here we have first completed the square (using symmetry of $A$), and then used the ellipticity of the matrix $A$.
	The important point is that for all $1<p<2$ the coefficient $\frac{(p-1)\lambda}{3-p}$ is positive.

	We also note that we could have completed the square differently and obtained instead of~\eqref{eq:rearg4} the estimate
	\begin{equation}
		\label{eq:rearg5}
		\begin{split}
			\text{LHS of~\eqref{eq:rearg}} &\ge
			\frac12 \int_{2Q} \left[(w_ku^{p/2-1}\zeta)^2\right]_t \dX\dt \\
			&\quad + \frac{(p-1)(2-p)\lambda}{2}\int_{2Q} \left|(\nabla u)w_ku^{p/2-2}\zeta\right|^2 \dX\dt.
		\end{split}
	\end{equation}
	It follows that we could average~\eqref{eq:rearg4} and~\eqref{eq:rearg5} and have both
	\[
		\int_{2Q} \left|\nabla(w_k\zeta)u^{p/2-1}\right|^2 \dX\dt
		\quad\text{and}\quad
		\int_{2Q} \left|(\nabla u)w_ku^{p/2-2}\zeta\right|^2 \dX\dt
	\]
	in the estimate with small positive constants.

	Now we briefly mention how all the error terms of~\eqref{eq:rearg},~\eqref{eq:rearg2} and~\eqref{eq:rearg4} can be handled.
	Some can be immediately estimated from above by
	\[
		r^{-2}\int_{2Q}|W|^2u^{p-2}\dX\dt,
	\]
	where the scaling factor $r^{-2}$ comes from the estimates on $\nabla\zeta$, $\zeta_t$, $|\nabla A|$ and $|B|$.
	For other terms (for example the first term of fifth line of~\eqref{eq:rearg} or the second term of the same line) we use Cauchy-Schwarz.
	One of the terms in the product will be
	\[
		\left(r^{-2}\int_{2Q}|W|^2u^{p-2}\dX\dt\right)^{1/2} \hspace{-1em},
	\]
	while the other term is one of
	\[
		\left(\int_{2Q}\left|\nabla(w_k\zeta)u^{p/2-1}\right|^2\dX\dt\right)^{1/2}
		\text{ or}\quad
		\left(\int_{2Q}\left|(\nabla u)w_ku^{p/2-2}\zeta\right|^2\dX\dt\right)^{1/2} \hspace{-1em}.
	\]
	It follows using the $\epsilon$-Cauchy-Schwarz inequality that we can hide these on the left-hand side of~\eqref{eq:rearg}.
	Finally, we put everything together by summing over all $k$ and recalling that $W=\nabla u$.
	This gives for some constant $\epsilon = \epsilon(p,\lambda,n)>0$ with $\epsilon \to 0$ as $p \to 1$.
	\begin{align}
		 & \sup_{\tau}\int_{Q\cap\{t=\tau\}}|\nabla u|^2u^{p-2}\dX
		+ \epsilon\int_{Q}|\nabla^2 u|^2u^{p-2}\dX\dt
		+ \epsilon\int_{Q}|\nabla u|^4u^{p-4}\dX\dt \notag         \\
		 & \le Cr^{-2}\int_{2Q}|\nabla u|^2u^{p-2}\dX\dt.
	\end{align}
	In particular~\eqref{est-repl} holds.
\end{proof}
After using~\eqref{est-repl} in~\eqref{eq-AAAA} we can conclude the following.

\begin{lemma}%
	\label{L:A<S}
	Let $u$ be a non-negative solution of~\eqref{E:pde} with matrix $A$ satisfying the ellipticity hypothesis and the coefficients satisfying the bound $|\nabla A|,|B|\le K/\delta$.
	Then given $a > 0$ there exists a constant $C = (\Lambda,\lambda, a, K,p, n)$ such that
	\begin{equation}
		\label{E:A<S:local}
		A_{p,a}(u)(X,t) \leq C S_{p,2a}(u)(X,t).
	\end{equation}
	From this we have the global estimate
	\begin{equation}
		\label{E:A<S:global}
		\|A_{p,a}(u)\|_{L^p(\partial\Omega)}^p \leq C_2 \|S_{p,a}(u)\|_{L^p(\partial\Omega)}^p.
	\end{equation}
\end{lemma}

As far as the proof goes, the calculations above clearly work for solutions $u$ with uniform bound $u\ge\epsilon>0$.
Hence considering $v_\epsilon=u+\epsilon$ and then taking the limit $\epsilon\to 0+$ using Fatou's lemma yields~\eqref{E:A<S:local} for all non-negative $u$, where we have used the convention that $|\nabla u|^2u^{p-2}=0$ whenever $u=0$ and $\nabla u=0$ with a similar convention for the second gradient in $A_{p,a}$.

\section{Bounding the \texorpdfstring{$p$}{p}-adapted square function by the non-tangential maximum function}

We slightly abuse notation and only work on a Carleson region $T(\Delta_r)$ in the upper half space $U$ even though we formulate the following lemmas on any admissible domain $\Omega$.
The equivalence of these formulations via the pullback map $\rho$ is discussed in~\cref{S:pullback} and~\cite{DH16}, and hence we omit the details.
We start with a local bound of the $p$-adapted square function by the non-tangential maximal function.

\begin{lemma}%
	\label{L:S<N:local}
	Let $\Omega$ be an admissible domain from \cref{D:domain} with character $(\ell,\eta, N, d)$.
	Let $1< p < 2$ and $u$ be a non-negative solution of~\cref{E:pde}, with the Carleson conditions~\cref{E:1:carl,E:1:bound} on the coefficients $A$ and $B$.
	Then there exists a constant $C=C(\lambda, \Lambda, N, C_0)$ such that for any solution $u$ with boundary data $f$ on any ball $\Delta_r \subset \partial\Omega$ with $r \leq \min\{d/4, d/(4C_0)\}$ we have
	\begin{equation}
		\label{E:S<N:local}
		\int_{T(\Delta_r)} |\nabla u|^2 |u|^{p-2} x_0 \dx_0 \dx \dt
		\leq C(1 + \|\mu\|)(1 + \ell^2) \int_{\Delta_{2r}} \left(N^{2r}\right)^2(u) \dx\dt.
	\end{equation}
\end{lemma}

In addition we have the following global result.
\begin{lemma}%
	\label{L:S<N:global}
	Let $\Omega$ be an admissible domain with smooth boundary $\partial\Omega$.
	Let $1< p < 2$ and $u$ be a weak non-negative solution of~\cref{E:pde} satisfying~\cref{E:pullback:B:carl,E:pullback:B:bound,E:pullback:A:carl,E:pullback:A:bound} with Dirichlet boundary data $f \in L^p(\partial\Omega)$.
	Then there exists positive constants $C_1$ and $C_2$ independent of $u$ such that for small $r_0 > 0$ we have
	\begin{equation}
		\label{E:S<N:global}
		\begin{split}
			&\frac{C_1}{2} \int_0^{r_0/2}\int_{\partial\Omega} |\nabla u|^2 |u|^{p-2} x_0 \dx\dt\dx_0
			+ \frac{2}{r_0} \int_0^{r_0}\int_{\partial\Omega} u^p(x_0,x,t) \dx\dt\dx_0 \\
			&\leq \int_{\partial\Omega} u^p(r_0,x,t) \dx\dt + \int_{\partial\Omega} u^p(0,x,t) \dx\dt \\
			&\quad+ C_2\left(\|\mu_1\|_{C, 2r} + \|\mu_2\|_{C, 2r} + \|\mu_2\|^{1/2}_{C, 2r}\right) \int_{\partial\Omega} \left(N^{2r}(u)\right)^p \dx\dt.
		\end{split}
	\end{equation}
\end{lemma}

\begin{proof}[Proof of \cref{L:S<N:local,L:S<N:global}]
	Let $Q_r(y,s)$ be a parabolic cube on the boundary with $r < d$ and let $\zeta$ be a smooth cut off function independent of the $x_0$ variable.
	As long as there is no ambiguity we suppress the argument of $Q_r$ and extensively use the Einstein summation convention.
	Let $\zeta$ be supported in $Q_{2r}$, equal 1 in $Q_r$ and satisfy the estimate $r|\nabla \zeta| + r^2 |\zeta_t| \leq C$ for some constant $C$ and.

	We start by estimating
	\begin{equation}
		\int_0^r\int_{Q_{2r}} |u|^{p-2} \frac{a_{ij}}{a_{00}} (\partial_i u) (\partial_j u) \zeta^2 x_0 \dx\dt\dx_0,
	\end{equation}
	where by ellipticity we have
	\begin{equation*}
		\frac{\lambda}{\Lambda} \int_0^r\int_{Q_r} |\nabla u|^2 |u|^{p-2} x_0 \dx\dt\dx_0
		\leq \int_0^r\int_{Q_{2r}} |u|^{p-2} \frac{a_{ij}}{a_{00}} (\partial_i u) (\partial_j u) \zeta^2 x_0 \dx\dt\dx_0.
	\end{equation*}
	Now we integrate by parts whilst noting that $\nu = (1,0,0, \dots ,0)$ since the domain is $\{x_0 > 0\}$.
	\begin{equation}
		\label{E:S<N:IBP}
		\begin{split}
			&\int_0^r\int_{Q_{2r}} |u|^{p-2} \frac{a_{ij}}{a_{00}} (\partial_i u) (\partial_j u) \zeta^2 x_0 \dx\dt\dx_0  \\
			&= \frac{1}{p} \int_{Q_{2r}} \frac{a_{0j}}{a_{00}} \partial_j \left(|u(r,x,t)|^p\right) r \zeta^2 \dx\dt \\
			&\quad - \int_0^r\int_{Q_{2r}}\frac{1}{a_{00}} |u|^{p-2}  u \partial_i(a_{ij}\partial_j u) \zeta^2 x_0 \dx\dt\dx_0  \\
			&\quad - \int_0^r\int_{Q_{2r}}\partial_i\left(\frac{1}{a_{00}}\right)  |u|^{p-2}  u a_{ij}\partial_j u \zeta^2 x_0 \dx\dt\dx_0 \\
			&\quad - 2 \int_0^r\int_{Q_{2r}} \frac{a_{ij}}{a_{00}} |u|^{p-2}u (\partial_j u) \zeta \partial_i\zeta x_0 \dx\dt\dx_0 \\
			&\quad -\int_0^r\int_{Q_{2r}} \frac{a_{0j}}{a_{00}} |u|^{p-2}u (\partial_j u) \zeta^2 \dx\dt\dx_0 \\
			&\quad -\int_0^r\int_{Q_{2r}} \frac{a_{ij}}{a_{00}} \partial_i\left(|u|^{p-2}\right)  u (\partial_j u) \zeta^2 \dx\dt\dx_0 \\
			&= I + II + III + IV + V + VI
		\end{split}
	\end{equation}
	Our strategy is to further estimate all these terms and then group similar terms together.
	First consider $II$, we use that $u$ is a solution to~\cref{E:pde}
	\begin{equation*}
		\begin{split}
			II &= - \int_0^r\int_{Q_{2r}}\frac{1}{a_{00}} |u|^{p-2}  u u_t \zeta^2 x_0 \dx\dt\dx_0
			+ \int_0^r\int_{Q_{2r}}\frac{1}{a_{00}} |u|^{p-2}  u b_i \partial_i u \zeta^2 x_0 \dx\dt\dx_0 \\
			&= II_1 + II_2.
		\end{split}
	\end{equation*}
	Using the identity $2x_0 = \partial_0 x_0^2$ we integrate by parts in $x_0$ to obtain
	\begin{equation*}
		\begin{split}
			II_1 &= - \frac{1}{2}\int_0^r\int_{Q_{2r}} \frac{1}{a_{00}} |u|^{p-2}  u u_t \zeta^2 \partial_0 x_0^2 \dx\dt\dx_0 \\
			&= - \frac{1}{2} \int_{Q_{2r}} \frac{1}{a_{00}} |u(r,x,t)|^{p-2}  u(r,x,t) u_t(r,x,t) \zeta^2 r^2 \dx\dt \\
			&\quad+ \frac{1}{2}\int_0^r\int_{Q_{2r}} \partial_0 \left(\frac{1}{a_{00}}\right) |u|^{p-2}  u u_t \zeta^2 x_0^2 \dx\dt\dx_0 \\
			&\quad+ \frac{p-1}{2}\int_0^r\int_{Q_{2r}} \frac{1}{a_{00}} |u|^{p-2}  \partial_0 u u_t \zeta^2 x_0^2 \dx\dt\dx_0\\
			&\quad+ \frac{1}{2}\int_0^r\int_{Q_{2r}} \frac{1}{a_{00}} |u|^{p-2}  u \partial_0\partial_tu \zeta^2 x_0^2 \dx\dt\dx_0 \\
			&= II_{11} + II_{12} + II_{13} + II_{14}.
		\end{split}
	\end{equation*}
	Consider the boundary term $II_{11}$ and we integrate by parts in $t$
	\begin{equation*}
		\begin{split}
			II_{11} &= - \frac{1}{4} \int_{Q_{2r}} \frac{1}{a_{00}} |u(r,x,t)|^{p-2}  \partial_t \left(u^2(r,x,t)\right) \zeta^2 r^2 \dx\dt \\
			&= \frac{1}{4} \int_{Q_{2r}} \partial_t \left(\frac{1}{a_{00}}\right) |u(r,x,t)|^{p-2}  u^2(r,x,t) \zeta^2 r^2 \dx\dt \\
			&\quad+ \frac{1}{2} \int_{Q_{2r}} \frac{1}{a_{00}} |u(r,x,t)|^{p-2}  u^2(r,x,t) \zeta \zeta_t r^2 \dx\dt \\
			&\quad+ \frac{p-2}{4} \int_{Q_{2r}} \frac{1}{a_{00}} |u(r,x,t)|^{p-2}  u(r,x,t) u_t(r,x,t) \zeta^2 r^2 \dx\dt \\
			&= II_{111} + II_{112} + II_{113}.
		\end{split}
	\end{equation*}
	Since $p < 2$, so $p-2 < 0$, we can absorb $II_{113}$ into $II_{11}$ and save $II_{12}$ to bound later on.

	Considering $II_{14}$, we swap the order of differentiation on $\partial_0\partial_t u$ and integrate by parts in $t$ to show
	\begin{equation*}
		\begin{split}
			II_{14} &= \frac{1}{2}\int_0^r\int_{Q_{2r}} \frac{1}{a_{00}} |u|^{p-2}  u \partial_t\partial_0u \zeta^2 x_0^2 \dx\dt\dx_0 \\
			&= -\frac{1}{2}\int_0^r\int_{Q_{2r}} \partial_t\left(\frac{1}{a_{00}}\right)  |u|^{p-2}  u \partial_0u \zeta^2 x_0^2 \dx\dt\dx_0 \\
			&\quad - \frac{p-1}{2}\int_0^r\int_{Q_{2r}} \frac{1}{a_{00}} |u|^{p-2}  u_t \partial_0u \zeta^2 x_0^2 \dx\dt\dx_0\\
			&\quad- \int_0^r\int_{Q_{2r}} \frac{1}{a_{00}} |u|^{p-2}  u \partial_0u \zeta \zeta_t x_0^2 \dx\dt\dx_0 \\
			&= II_{141} + II_{142} + II_{143}.
		\end{split}
	\end{equation*}
	Observe that $II_{142} = -II_{13}$ so these terms cancel.
	We bound $II_{141}$ by
	\begin{equation*}
		\begin{split}
			II_{141} &= \frac{1}{2}\int_0^r\int_{Q_{2r}} \frac{\partial_t a_{00}}{a^2_{00}}  |u|^{p-2}  u \partial_0u \zeta^2 x_0^2 \dx\dt\dx_0 \\
			&\lesssim \left(\int_0^r\int_{Q_{2r}} |A_t|^2 |u|^p x^3_0 \zeta^2 \dx\dt\dx_0\right)^{1/2} \left(\int_0^r\int_{Q_{2r}} |\nabla u|^2|u|^{p-2} x_0 \zeta^2 \dx\dt\dx_0\right)^{1/2} \hspace{-1em}.
		\end{split}
	\end{equation*}
	Two parts of $II_1$ we have left to bound are $II_{112}$ and $II_{143}$.
	Both of these integrals involve $\zeta \zeta_t$ and therefore if $\zeta$ is a partition of unity when we sum over that partition these terms sum to 0.

	The terms $II_2$ and $III$ are simply dealt with by
	\begin{equation*}
		\begin{split}
			II_2 \lesssim \left(\int_0^r\int_{Q_{2r}} |B|^2 |u|^p x_0 \zeta^2 \dx\dt\dx_0\right)^{1/2} \left(\int_0^r\int_{Q_{2r}} |\nabla u|^2|u|^{p-2} x_0 \zeta^2 \dx\dt\dx_0\right)^{1/2}
		\end{split}
	\end{equation*}
	and
	\begin{equation*}
		\begin{split}
			III \lesssim \left(\int_0^r\int_{Q_{2r}} |\nabla A|^2 |u|^p x_0 \zeta^2 \dx\dt\dx_0\right)^{1/2} \left(\int_0^r\int_{Q_{2r}} |\nabla u|^2|u|^{p-2} x_0 \zeta^2 \dx\dt\dx_0\right)^{1/2} \hspace{-1em}.
		\end{split}
	\end{equation*}

	The integral in the term $IV$ contains the terms $\zeta \partial_i\zeta$ and as before if $\zeta$ is a partition of unity then after summing this term cancels out.
	Therefore the terms that we have yet to estimate are $I$, $V$ and $VI$.

	We consider $V$ in the two cases $j=0$ and $j\neq 0$ separately.
	Since $\zeta$ is independent of $x_0$ by the fundamental theorem of calculus
	\begin{equation*}
		\begin{split}
			V_{\{j=0\}} &= -\int_0^r\int_{Q_{2r}} |u|^{p-2}u (\partial_0 u) \zeta^2 \dx\dt\dx_0
			= - \frac{1}{p}\int_0^r\int_{Q_{2r}} \partial_0\left(|u|^p \zeta^2\right) \dx\dt\dx_0  \\
			&= \frac{1}{p} \int_{Q_{2r}} |u(0,x,t)|^p \zeta^2 \dx\dt - \frac{1}{p} \int_{Q_{2r}} |u(r,x,t)|^p \zeta^2 \dx\dt.
		\end{split}
	\end{equation*}
	For the $j \neq 0$ case we use that $\partial_0 x_0 = 1$ and integrate this case by parts in $x_0$
	\begin{equation*}
		\begin{split}
			V_{\{j \neq 0\}} &= - \frac{1}{p}\int_0^r\int_{Q_{2r}} \frac{a_{0j}}{a_{00}} \partial_j\left(|u|^p\right) \zeta^2 \dx\dt\dx_0 \\
			&= - \frac{1}{p}\int_0^r\int_{Q_{2r}} \frac{a_{0j}}{a_{00}} \partial_j\left(|u|^p\right) \zeta^2 \partial_0 x_0 \dx\dt\dx_0 \\
			&= - \frac{1}{p} \int_{Q_{2r}} \frac{a_{0j}}{a_{00}} \partial_j\left(|u(r,x,t)|^p\right) \zeta^2 r \dx\dt
			+ \frac{1}{p}\int_0^r\int_{Q_{2r}} \frac{a_{0j}}{a_{00}} \partial_j\partial_0\left(|u|^p\right) \zeta^2 x_0 \dx\dt\dx_0 \\
			&\quad + \frac{1}{p}\int_0^r\int_{Q_{2r}} \partial_0\left(\frac{a_{0j}}{a_{00}}\right) \partial_j\left(|u|^p\right) \zeta^2 x_0 \dx\dt\dx_0 \\
			&= V_1 + V_2 + V_3.
		\end{split}
	\end{equation*}
	The term $V_1 = -I_{\{j \neq 0\}}$ so they cancel out.
	For $V_2$ we integrate by parts in $x_j$
	\begin{equation*}
		\begin{split}
			V_2 &= -\sum_{j \neq 0}\frac{1}{p} \int_{Q_{2r}} \frac{a_{0j}}{a_{00}} \partial_0\left(|u(r,x,t)|^p\right) \zeta^2 r \dx\dt \\
			&\quad - \frac{1}{p}\int_0^r\int_{Q_{2r}} \partial_j\left(\frac{a_{0j}}{a_{00}}\right) \partial_0\left(|u|^p\right) \zeta^2 x_0 \dx\dt\dx_0 \\
			&\quad - \frac{2}{p}\int_0^r\int_{Q_{2r}} \frac{a_{0j}}{a_{00}} \partial_0\left(|u|^p\right) \zeta\partial_j\zeta x_0 \dx\dt\dx_0 \\
			&= V_{21} + V_{22} + V_{23}.
		\end{split}
	\end{equation*}
	$V_{22}$ and $V_3$ are of the same type and can be estimated as $III$ by
	\begin{equation*}
		\begin{split}
			&\left|\int_0^r\int_{Q_{2r}} \nabla\left(\frac{a_{0j}}{a_{00}}\right) \nabla\left(|u|^p\right) \zeta^2 x_0 \dx\dt\dx_0 \right|
			\lesssim \int_0^r\int_{Q_{2r}} |u|^{p-1}|\nabla u| |\nabla A| \zeta^2 x_0 \dx\dt\dx_0 \\
			&\lesssim  \left(\int_0^r\int_{Q_{2r}} |\nabla A|^2 |u|^p  \zeta^2 x_0 \dx\dt\dx_0\right)^{1/2} \left(\int_0^r\int_{Q_{2r}} |\nabla u|^2|u|^{p-2} \zeta^2 x_0 \dx\dt\dx_0\right)^{1/2}\hspace{-1em}.
		\end{split}
	\end{equation*}

	The final term from~\cref{E:S<N:IBP} to estimate is $VI$
	\begin{equation*}
		\begin{split}
			VI &= -\int_0^r\int_{Q_{2r}} \frac{a_{ij}}{a_{00}} \partial_i\left(|u|^{p-2}\right)  u (\partial_j u) \zeta^2 \dx\dt\dx_0 \\
			&= (2-p) \int_0^r\int_{Q_{2r}} \frac{a_{ij}}{a_{00}} |u|^{p-2} (\partial_i u) (\partial_j u) \zeta^2 \dx\dt\dx_0
		\end{split}
	\end{equation*}
	and since $2-p < 1$ we can hide $VI$ in the left hand side of~\cref{E:S<N:IBP}.

	We are now at the stage where we can group together all the similar terms and estimate them.
	There are 4 different types of terms:
	\begin{align*}
		J_1 & = I_{\{ j = 0 \}} + II_{111} + V_{\{ j = 0 \}} + V_{21} \\
		J_2 & = II_{12} \\
		J_3 & = II_{141} + II_2 + III + \sum_{j \neq 0} V_{22} + \sum_{j \neq 0} V_3 \\
		J_4 & = II_{112} + II_{143} + IV + \sum_{j \neq 0} V_{23}.
	\end{align*}

	We shall use the following standard result multiple times to deal with terms containing $|\nabla A|^2$, $|A_t|$ or $|B|$; a reference for this is~\cite[p.~59]{Ste93}.
	Let $\mu$ be a Carleson measure and $U$ the upper half space then for any function $u$ we have
	\begin{equation}
		\label{E:carl:estimate}
		\int_U |u|^p \dmu \leq \|\mu\|_C \|N(u)\|^p_{L^p(\R^n)},
	\end{equation}
	with a local version holding on Carleson boxes as well.

	First we consider $J_1$, which consists of boundary terms at $(0,x,t)$ and $(r,x,t)$.
	\begin{equation*}
		\label{E:J1}
		\begin{split}
			J_1 &= \frac{1}{p} \int_{Q_{2r}} \partial_0 \left(|u(r,x,t)|^p \right) \zeta^2 r \dx\dt
			-\frac{1}{4} \int_{Q_{2r}}  \frac{\partial_t a_{00}}{a_{00}^2} |u(r,x,t)|^{p-2}  u^2(r,x,t) \zeta^2 r^2 \dx\dt \\
			&\quad + \frac{1}{p} \int_{Q_{2r}} |u(0,x,t)|^p \zeta^2 \dx\dt
			- \frac{1}{p} \int_{Q_{2r}} |u(r,x,t)|^p \zeta^2 \dx\dt \\
			&\quad -\sum_{j \neq 0}\frac{1}{p} \int_{Q_{2r}} \frac{a_{0j}}{a_{00}} \partial_0\left(|u(r,x,t)|^p\right)\zeta^2 r \dx\dt.
		\end{split}
	\end{equation*}
	The second term in $J_1$, originating from $II_{111}$, has the bound
	\begin{equation*}
		\begin{split}
			II_{111} &= -\frac{1}{4} \int_{Q_{2r}}  \frac{\partial_t a_{00}}{a_{00}^2} |u(r,x,t)|^{p-2}  u^2(r,x,t) \zeta^2 r^2 \dx\dt \\
			&\leq \frac{1}{4\lambda^2} \int_{Q_{2r}} |A_t| |u(r,x,t)|^{p} \zeta^2 r^2 \dx\dt
			\leq \frac{\|\mu_2\|^{1/2}_{C, 2r}}{\lambda^2} \|N^r(u)\|^{p}_{L^p(Q_{2r})}.
		\end{split}
	\end{equation*}

	\begin{equation*}
		\label{E:J2}
		\begin{split}
			&J_2 = \frac{1}{2}\int_0^r\int_{Q_{2r}} \partial_0 \left(\frac{1}{a_{00}}\right) |u|^{p-2}  u u_t \zeta^2 x_0^2 \dx\dt\dx_0 \\
			&\leq \frac{1}{2 \lambda^2} \left(\int_0^r\int_{Q_{2r}} |\nabla A|^2 |u|^p x_0 \zeta^2 \dx\dt\dx_0 \right)^{1/2}
			\left(\int_0^r\int_{Q_{2r}} |u_t|^2 |u|^{p-2} x_0^3 \zeta^2 \dx\dt\dx_0 \right)^{1/2} \\
			&\leq \frac{1}{\lambda^2} \left( \|\mu_2\|_{C, 2r} \|N^r(u)\|^{p}_{L^p(Q_{2r})} \right)^{1/2}
			\left(\int_0^r\int_{Q_{2r}} |u_t|^2 |u|^{p-2} x_0^3 \zeta^2 \dx\dt\dx_0 \right)^{1/2}\hspace{-1em}.
		\end{split}
	\end{equation*}
	With a constant $C_3 = C_3(\lambda, \Lambda, n)$ we can bound $J_3$ by
	\begin{equation*}
		\label{E:J3}
		\begin{split}
			& J_3
			\leq C_3 \left(\int_0^r\int_{Q_{2r}} \left( x_0|\nabla A|^2 + x_0 |B|^2 + x_0^3 |A_t|^2 \right)  |u|^p \zeta^2 \dx\dt\dx_0 \right)^{1/2}\\
			&\qquad\times \left(\int_0^r\int_{Q_{2r}} |\nabla u|^2 |u|^{p-2} x_0 \zeta^2 \dx\dt\dx_0 \right)^{1/2} \\
			&\leq C_3 \left( \left(\|\mu_1\|_{C, 2r} + \|\mu_2\|_{C, 2r} \right) \|N^r(u)\|^{p}_{L^p(Q_{2r})} \right)^{1/2}
			\left(\int_0^r\int_{Q_{2r}} |\nabla u|^2 |u|^{p-2} x_0 \zeta^2 \dx\dt\dx_0 \right)^{1/2}\hspace{-1em}.
		\end{split}
	\end{equation*}
	Finally, $J_4$ consists of terms of the type $\zeta \partial_t \zeta$ or $\zeta \partial_i \zeta$.
	Later we take $\zeta$ to be a partition of unity and so when we sum up over the partition all the terms in $J_4$ sum to 0.

	Therefore after all these calculations
	\begin{equation}
		\label{E:J:bound}
		\begin{split}
			&\int_0^r\int_{Q_{2r}} |u|^{p-2} \frac{a_{ij}}{a_{00}} (\partial_i u) (\partial_j u) \zeta^2 x_0 \dx\dt\dx_0
			= J_1 + J_2 + J_3 + J_4 \\
			&\leq \frac{n\Lambda}{\lambda} \int_{Q_{2r}} \partial_0 \left(|u(r,x,t)|^p \right) \zeta^2 r \dx\dt
			+ \int_{Q_{2r}} |u(0,x,t)|^p \zeta^2 \dx\dt \\
			& - \int_{Q_{2r}} |u(r,x,t)|^p \zeta^2 \dx\dt
			+ \frac{\|\mu_2\|^{1/2}_{C, 2r}}{\lambda^2} \|N^r(u)\|^{p}_{L^p(Q_{2r})} \\
			& + \frac{1}{\lambda^2} \left( \|\mu_2\|_{C, 2r} \|N^r(u)\|^{p}_{L^p(Q_{2r})} \right)^{1/2}
			\left(\int_0^r\int_{Q_{2r}} |u_t|^2 |u|^{p-2} x_0^3 \zeta^2 \dx\dt\dx_0 \right)^{1/2} \\
			& + C_3 \left( \left(\|\mu_1\|_{C, 2r} + \|\mu_2\|_{C, 2r} \right) \|N^r(u)\|^{p}_{L^p(Q_{2r})} \right)^{1/2} \\
			&\qquad\times \left(\int_0^r\int_{Q_{2r}} |\nabla u|^2 |u|^{p-2} x_0 \zeta^2 \dx\dt\dx_0 \right)^{1/2}
			+ J_4.
		\end{split}
	\end{equation}

	By assuming that $\Omega$ is smooth as well as an admissible domain (\cref{D:domain}) there exists a collar neighbourhood $V$ of $\partial\Omega$ in $\R^{n+1}$ such that $\Omega \cap V$ can be globally parametrised by $(0,r) \times \partial\Omega$ for some small $r > 0$, see \cref{R:domain:smooth} and~\cite{DH16} for details.
	Using \cref{D:domain}, there is a collection of charts covering $\partial\Omega$ with bounded overlap, say by $M$.
	We consider a partition of unity of these charts $\zeta_j$, with $\zeta_j$ having the same definition, support and estimates as $\zeta$ before, and $\sum_j \zeta_j = 1$ everywhere.
	Therefore, when we sum~\cref{E:J:bound} over this partition of unity the term on the left hand side is bounded below by
	\[
		\frac{1}{\Lambda} \int_0^r\int_{\partial \Omega} |u|^{p-2} (A \nabla u \cdot \nabla u) x_0 \dx\dt\dx_0,
	\]
	which is comparable to the truncated $p$-adapted square function $\|S^r_p(u)\|_{L^p(\partial\Omega)}^p$.
	Therefore, remembering that after summing $J_4 = 0$, for any $\epsilon > 0$ we have
	\begin{equation}
		\label{E:S<N:longBound}
		\begin{split}
			&\frac{\lambda}{\Lambda} \|S^r_p(u)\|_{L^p(\partial\Omega)}^p
			\sim \frac{\lambda}{\Lambda} \int_0^r\int_{\partial \Omega} |u|^{p-2} |\nabla u|^2 x_0 \dx\dt\dx_0 \\
			&\leq \frac{n\Lambda}{\lambda} \int_{\partial\Omega} \partial_0 \left(|u(r,x,t)|^p\right) r \dx\dt
			+ \int_{\partial\Omega} |u(0,x,t)|^p \dx\dt - \int_{\partial\Omega} |u(r,x,t)|^p  \dx\dt  \\
			&\quad + \frac{M \|\mu_2\|^{1/2}_{C, 2r}}{\lambda^2} \|N^r(u)\|^{p}_{L^p(\partial\Omega)}
			+ \frac{\|\mu_2\|_{C, 2r} }{4 \epsilon \lambda^2} \|N^r(u)\|^{p}_{L^p(\partial\Omega)} \\
			&\quad + \epsilon \int_0^r\int_{\partial\Omega} |u_t|^2 |u|^{p-2} x_0^3 \zeta^2 \dx\dt\dx_0
			+ C_3 \frac{\|\mu_1\|_{C, 2r} + \|\mu_2\|_{C, 2r}}{4\epsilon} \|N^r(u)\|^{p}_{L^p(\partial\Omega)} \\
			&\quad + \epsilon \int_0^r\int_{\partial\Omega} |\nabla u|^2 |u|^{p-2} x_0 \zeta^2 \dx\dt\dx_0.
		\end{split}
	\end{equation}

	By applying \cref{L:A<S} to the $p$-adapted area function in~\cref{E:S<N:longBound} we see that the $p$-adapted square function on the right hand side of~\cref{E:S<N:longBound} is always multiplied by $\epsilon$.
	By choosing $\epsilon$ small enough we can absorb this $p$-adapted square function into the left hand side yielding
	\begin{equation}
		\label{E:S<N:shortBound}
		\begin{split}
			C_1\|S^r_p(u)\|_{L^p(\partial\Omega)}^p
			&\leq \int_{\partial\Omega} \partial_0 \left(|u(r,x,t)|^p\right) r \dx\dt \\
			&+ \int_{\partial\Omega} |u(0,x,t)|^p \dx\dt
			- \int_{\partial\Omega} |u(r,x,t)|^p \dx\dt \\
			&+ C_2\left(\|\mu_1\|_{C, 2r} + \|\mu_2\|_{C, 2r} + \|\mu_2\|^{1/2}_{C, 2r}\right) \|N^r(u)\|^{p}_{L^p(\partial\Omega)}.
		\end{split}
	\end{equation}
	We integrate~\cref{E:S<N:shortBound} in the $r$ variable, average over $[0,r_0]$ and use the identity $(\partial_0 |u|^p)x_0 = \partial_0(|u|^p x_0) - |u|^p$ to give
	\begin{equation}
		\begin{split}
			&C_1\int_0^{r_0} \int_{\partial\Omega} \left(x_0 - \frac{x_0^2}{r_0}\right) |\nabla u|^2 |u|^{p-2} \dx\dt\dx_0 \\
			&\quad+ \frac{2}{r_0} \int_0^{r_0} \int_{\partial\Omega} |u(x_0,x,t)|^p \dx\dt\dx_0 \\
			&\leq \int_{\partial\Omega} |u(r_0,x,t)|^p \dx\dt
			+ \int_{\partial\Omega} |u(0,x,t)|^p \dx\dt \\
			&\quad + C_2\left(\|\mu_1\|_{C, 2r} + \|\mu_2\|_{C, 2r} + \|\mu_2\|^{1/2}_{C, 2r}\right) \|N^r(u)\|^{p}_{L^p(\partial\Omega)}.
		\end{split}
	\end{equation}
	Finally truncating the first integral on the left hand side to $[0, r_0/2]$ gives
	\begin{equation}
		\begin{split}
			&\frac{C_1}{2}\int_0^{r_0/2} \int_{\partial\Omega} |\nabla u|^2 |u|^{p-2}x_0 \dx\dt\dx_0
			+ \frac{2}{r_0} \int_0^{r_0} \int_{\partial\Omega} |u(x_0,x,t)|^p \dx\dt\dx_0 \\
			&\leq \int_{\partial\Omega} |u(r_0,x,t)|^p \dx\dt
			+ \int_{\partial\Omega} |u(0,x,t)|^p \dx\dt \\
			&\quad + C_2\left(\|\mu_1\|_{C, 2r} + \|\mu_2\|_{C, 2r} + \|\mu_2\|^{1/2}_{C, 2r}\right) \|N^r(u)\|^{p}_{L^p(\partial\Omega)}.
		\end{split}
	\end{equation}

	The local estimate for~\cref{L:S<N:local} is obtained (exactly as in~\cite{DH16}) if we do not sum over all the coordinate patches but instead use the estimates obtained for a single boundary cube $Q_r$ in~\cref{E:J:bound}.
\end{proof}

We just need to control the first integral on the right hand side of~\cref{E:S<N:global} to achieve our goal of controlling the $p$-adapted square function.
Thankfully this has already been done for us in the proof of~\cite[Cor.~5.3]{DH16} which we encapsulate below.
\begin{lemma}%
	\label{L:S<N:helper}
	Let $\Omega$ be as in \cref{L:S<N:global} and $u$ be a non-negative solution to~\cref{E:pde}.
	For a small $r_0 > 0$ depending on the geometry the domain $\Omega$ there exists a constant $C$ such that for $\epsilon = \|\mu_1\|_{C, 2r} + \|\mu_2\|_{C, 2r} + \|\mu_2\|^{1/2}_{C, 2r}$
	\begin{equation*}
		\label{E:S<N:helper}
		\int_{\partial\Omega} u(r_0,x,t)^p \dx\dt
		\leq \frac{2}{r_0} \int_0^{r_0} \int_{\partial\Omega} u(x_0,x,t)^p \dx\dt\dx_0
		+ C \epsilon \|N^{r_0}(u)\|^{p}_{L^p(\partial\Omega)}.
	\end{equation*}
\end{lemma}

Combining \cref{L:S<N:global,L:S<N:helper} gives us the desired result.
\begin{corollary}%
	\label{C:S<N}
	Let $\Omega$ be as in \cref{L:S<N:global} and $u$ be a non-negative solution to~\cref{E:pde}.
	For a small $r_0 > 0$ depending on the geometry the domain $\Omega$ there exists constants $C_1, C_2 > 0$ such that for $\epsilon = \|\mu_1\|_{C, 2r} + \|\mu_2\|_{C, 2r} + \|\mu_2\|^{1/2}_{C, 2r}$
	\begin{equation}
		\label{E:S<N}
		\begin{split}
			\|S_p^{r_0/2}(u)\|^p_{L^p(\partial\Omega)}
			&\sim \int_0^{r_0/2} \int_{\partial\Omega} |\nabla u|^2 |u|^{p-2}x_0 \dx\dt\dx_0 \\
			&\leq C_1 \int_{\partial\Omega} |u(0,x,t)|^p \dx\dt
			+ C_2 \epsilon \|N^{r_0}(u)\|^{p}_{L^p(\partial\Omega)}.
		\end{split}
	\end{equation}
\end{corollary}

\section{Bounding the non-tangential maximum function by the \texorpdfstring{$p$}{p}-adapted square function}

Our goal in this section has been vastly simplified due to~\cite{Riv03} proving a local good-$\lambda$ inequality.
We use this to bound the non-tangential maximum function by the $p$-adapted square function.
We first bound the non-tangential maximum function by the usual $L^2$ based square function $S_2(u)$ but a simple argument from~\cite[(3.41)]{DPP07} shows that for $1 < p < 2$ and any $\epsilon >0$ we have
\begin{equation}
	\label{E:S2<Sp}
	\|S_2^r(u)\|_{L^p(\partial\Omega)} \leq C_\epsilon \|S_p^r(u)\|_{L^p(\partial\Omega)} + \epsilon \|N^r(u)\|_{L^p(\partial\Omega)},
\end{equation}
with a local version of this statement holding as well.

The good--$\lambda$ inequality from~\cite[p.~508]{Riv03} is expressed in the following lemma.
\begin{lemma}%
	\label{L:N<S:goodlambda}
	Let $v$ be a solution to~\cref{E:pullback} and $v(X,t) = 0$ for some point $(X,t) \in Q_r$.
	Let $E = \{ (0, x,t) \in Q_r : S_{2,a}(v) \leq \lambda \}$ and $q > 2$ then
	\begin{equation}
		\begin{split}
			\label{E:N<S:goodlambda}
			\left| \{ (0, x,t) \in Q_r : N_{a}(v) > \lambda \}\right|
			&\lesssim  |\{ (0, x,t) \in Q_r : S_{2,a}(v) > \lambda \}| \\
			&\quad+ \frac{1}{\lambda^q} \int_E S_{2,a}(v)^q \dx\dt.
		\end{split}
	\end{equation}
\end{lemma}

If $p \geq 2$ then the following lemma is immediate from~\cite[Lemma 6.1]{DH16}, which is an adaptation of~\cite[Theorem 1.3 and Proposition 5.3]{Riv03}.
\begin{lemma}%
	\label{L:N<S:local}
	Let $v$ be a solution to~\cref{E:pullback} in $U$ and the coefficients of~\cref{E:pullback} satisfy the Carleson estimates~\cref{E:pullback:B:carl,E:pullback:B:bound,E:pullback:A:carl,E:pullback:A:bound} on all parabolic balls of size $\leq r_0$.
	Then there exists a constant $C$ such that for any $r \in (0, r_0/8)$
	\begin{equation}
		\begin{split}
			\label{E:N<S:local}
			\int_{Q_r} N_{a/12}(v)^p \dx\dt
			&\leq C \left( \int_{Q_{2r}} A_{2,a}(v)^p \dx\dt + \int_{Q_{2r}} S_{2,a}(v)^p \dx\dt \right) \\
			&\quad+ r^{n+1} |v(A_{\Delta_r})|^p,
		\end{split}
	\end{equation}
	where $A_{\Delta_r}$ is a corkscrew point of the boundary ball $\Delta_r$.
	That is a point $2r^2$ later in time than the centre of $\Delta_r$ and at a distance comparable to $r$ from the boundary and $r$ from the centre of the ball $\Delta_r$.
\end{lemma}

\begin{proof}
	We first assume that $v(X,t) = 0$ for some $(X,t) \in Q_r$ and then we have the good-$\lambda$ inequality~\cref{E:N<S:goodlambda}.
	The passage from this good-$\lambda$ inequality to a local $L^p$ estimate is standard in the spirit of~\cite{FS72}.
	We remove the assumption $v(X,t) = 0$ for the cost of adding the $r^{n+1} |v(A_{\Delta_r})|^p$ term in the same way as~\cite{Riv03,DH16}.
\end{proof}

From this local estimate we can obtain the following global $L^p$ estimate by the same proof as the global $L^2$ estimate from~\cite[Theorem 6.3]{DH16}.
\begin{theorem}%
	\label{T:N<S:global}
	Let $u$ be a solution to~\cref{E:pde} and the coefficients of~\cref{E:pde} satisfy the Carleson estimates~\cref{E:pde:B:carl,E:pde:A:carl} then
	\begin{equation}
		\label{E:N<S2:global}
		\|N^r(u)\|_{L^p(\partial\Omega)} \lesssim \|S_2^r(u)\|_{L^p(\partial\Omega)} + \|u\|_{L^p(\partial\Omega)}
	\end{equation}
	and by~\cref{E:S2<Sp}
	\begin{equation}
		\label{E:N<S:global}
		\|N^r(u)\|_{L^p(\partial\Omega)} \lesssim \|S_p^r(u)\|_{L^p(\partial\Omega)} + \|u\|_{L^p(\partial\Omega)}.
	\end{equation}
\end{theorem}

\section{Proof of \Cref{T:1}}

We only consider the case $1 < p < 2$ and use interpolation to obtain solvability for $p \geq 2$.
First assume either stronger Carleson condition of~\cref{E:pde:A:carl}, or~\cref{E:1:carl,E:1:bound} holds.
Therefore the Carleson conditions on the pullback coefficients~\cref{E:pullback:B:carl,E:pullback:B:bound,E:pullback:A:carl,E:pullback:A:bound} hold.

Without loss of generality, by \cref{R:domain:smooth}, we may assume that our domain is smooth.
Consider $f^+=\max\{0,f\}$ and $f^-=\max\{0,-f\}$ where $f\in C_0(\partial\Omega)$ and denote the corresponding solutions with these boundary data $u^+$ and $u^-$ respectively.
Hence we may apply the \cref{C:S<N} separately to $u^+$ and $u^-$. By the maximum principle these two solutions are non-negative.
It follows that for any such non-negative $u$ we have
\[
	\|S_p^{r}(u)\|^p_{L^p (\partial\Omega)} \leq C \|f\|^p_{L^p (\partial\Omega)} + C (\|\mu\|^{1/2}_{C}+\|\mu\|_{C}) \|N^{2r}(u)\|^p_{L^p(\partial\Omega)}
\]
and \cref{T:N<S:global} gives
\[
	\|N^{r}(u)\|^{p}_{L^{p}(\partial\Omega)} \leq C\|f\|^{p}_{L^p (\partial\Omega)} + C \|S_p^{2r}(u)\|^{p}_{L^p(\partial\Omega)},
\]
here $\|\mu\|_{C}$ is the Carleson norm of~\cref{E:1:carl} on Carleson regions of size $\le r_0$.
As noted earlier, if for example $\Omega$ is of VMO type then size of $\mu$ appearing in this estimate will only depend on the Carleson norm of coefficients on $\Omega$, provided we only consider small Carleson regions.
Hence we can choose $r_0$ small enough (depending on the domain $\Omega$) such that the Carleson norm after the pullback is say only twice the original Carleson norm of the coefficients over all balls of size $\le r_0$.

Since we are assuming $\|\mu\|_{C}$ is small, clearly we also have $\|\mu\|_{C}\le C\|\mu\|^{1/2}_{C}$.
By rearranging these two inequalities and combining estimates for $u^+$ and $u^-$, we obtain, for $0<r\leq r_0/8$,
\[
	\|N^{r}(u)\|^{p}_{L^{p}(\partial\Omega)} \leq C\|f\|^{p}_{L^p (\partial\Omega)} + C \|\mu\|^{1/2}_{C} \|N^{4r}(u)\|^{p}_{L^{p}(\partial\Omega)}.
\]
By a simple geometric argument in~\cite{DH16} involving cones of different apertures, \cref{L:Harnack,L:N:apertures} show there exists a constant $M$ such that
\begin{equation}
	\label{E:1:N4r<Nr}
	\|N^{4r}(u)\|^{p}_{L^{p}(\partial\Omega)} \leq M \|N^{r}(u)\|^{p}_{L^{p}(\partial\Omega)}.
\end{equation}
It follows that if $CM \|\mu\|^{1/2}_{C} <1/2$ by combining the last two inequalities we obtain
\[
	\|N^{r}(u)\|^{p}_{L^{p}(\partial\Omega)} \leq 2C\|f\|^{p}_{L^p (\partial\Omega)},
\]
which is the desired estimate (for truncated version of non-tangential maximum function).
The result with the non-truncated version of the non-tangential maximum function $N(u)$ follows as our domain is bounded in space and hence~\cref{E:1:N4r<Nr} can be iterated finitely many times until the non-tangential cones have sufficient height to cover the whole domain.\vglue1mm

Finally, we comment on how the Carleson condition~\cref{E:pde:A:carl} can be relaxed to the weaker condition~\cref{E:T1:carl:osc}.
The idea is the same as~\cite[Theorem 3.1]{DH16}.
As shown there, if the operator $\mathcal L$ satisfies the weaker condition~\cref{E:T1:carl:osc}, then it is possible (via mollification of coefficients) to find another operator ${\mathcal L}_1$ which is a small perturbation of the operator $\mathcal L$ and ${\mathcal L}_1$ satisfies~\cref{E:pde:A:carl}.
The solvability of the $L^p$ Dirichlet problem in the range $1<p<2$ for ${\mathcal L}_1$ follows by our previous arguments.
However, as $\mathcal L$ is a small perturbation of the operator ${\mathcal L}_1$ we have by the perturbation argument of~\cite{Swe98} $L^p$ solvability of $\mathcal L$  as well.

Finally, for larger values of $p$ we use the maximum principle and interpolation to obtain solvability results in the full range $1 < p < \infty$. \qed\

\section{Appendix --- proofs of results from section 2}

\begin{proof}[Proof of \cref{T:equiv}]
	We begin by proving the equivalence of \cref{I:D,I:av} using ideas from~\cite{Str80} and write $F=\D\phi$ where $F$ is a tempered distribution.
	Let
	\begin{equation*}
		\varphi^k = \chi_{\tilde{Q}_1(0,0)} - \chi_{\tilde{Q}_1(e_k)}
	\end{equation*}
	then for $1 \leq k \leq n-1$
	\begin{equation}
		\label{E:equiv:phi:k}
		\begin{split}
			\widehat{\varphi^k}(\xi, \tau) &= \frac{2\sin^2(\xi_k / 2)}{\xi_k} \frac{1 - e^{-i\tau}}{i\tau} \prod^{n-1}_{j \neq k} \frac{1 - e^{-i\xi_j}}{i\xi_j}, \\
			\widehat{\varphi^n}(\xi, \tau) &= \frac{2\sin^2(\tau / 2)}{\tau} \prod^{n-1}_{j = 1} \frac{1 - e^{-i\xi_j}}{i\xi_j},
		\end{split}
	\end{equation}
	with~$\widehat{\varphi^k}(\xi, \tau) \sim \xi_k$ for small $\xi_k$ and $1 \leq k \leq n-1$.
	We let
	\begin{equation*}
		\widehat{\psi^u} = \frac{e^{i(\xi,0)\cdot u} - 1}{\|(\xi,\tau)\|}
	\end{equation*}
	and denote by $\psi^u_\rho(x,t)$  the usual parabolic dilation by $\rho$, that is
	\[
		\psi^u_\rho(x,t) = \rho^{-(n+1)}\psi^u(x/\rho, t/\rho^2).
	\]
	It is worth noting that $(\varphi^k * \psi^u)_\rho = \varphi^k_\rho * \psi^u_\rho$.
	Therefore we may rewrite \cref{I:av:grad}, by \cref{R:av:grad}, as
	\begin{equation}
		\sup_{Q_r} \sum_{k=1}^{n-1} \frac{1}{|Q_r|} \int_{Q_r} \int_{u \in \mathbb{S}^{n-1}} \int_0^r
		\left(\psi^u_\rho * \varphi^k_\rho * F \right)^2 \frac{\drho}{\rho}\du\dx\dt
		\sim B_{\ref{I:av:grad}}.
	\end{equation}
	Similarly if we let
	\begin{equation}
		\label{E:equiv:phi:n}
		\widehat{\psi_n^u} = \frac{e^{i(0,\tau)\cdot u} - 1}{\|(\xi,\tau)\|}
	\end{equation}
	then we may rewrite \cref{I:av:Dt} as
	\begin{equation}{}
		\sup_{Q_r} \frac{1}{|Q_r|} \int_{Q_r} \int_{u \in \mathbb{S}^{n-1}} \int_0^r
		\left(\psi^u_{n,\rho} * F \right)^2 \frac{\drho}{\rho}\du\dx\dt
		\sim B_{\ref{I:av:Dt}}.
	\end{equation}
	The functions $\varphi^k * \psi^u$ and $\psi^u_n$ all satisfy the following conditions for some $\epsilon_i > 0$
	\begin{equation}
		\label{E:equiv:estimates}
		\begin{split}
			\int \psi \dx\dt &= 0, \\
			|\psi(x,t)| &\lesssim \|(x,t)\|^{-n-1-\epsilon_1} \text{ for } \|(x,t)\| \geq a > 0, \\
			|\widehat{\psi}(\xi,\tau)| &\lesssim \|(\xi,\tau)\|^{\epsilon_2} \text{ for } \|(\xi,\tau)\| \leq 1, \\
			|\widehat{\psi}(\xi,\tau)| &\lesssim \|(\xi,\tau)\|^{-\epsilon_3} \text{ for } \|(\xi,\tau)\| \geq 1.
		\end{split}
	\end{equation}
	Therefore if $\D\phi = F \in \BMO(\R^n)$ then $B_{\ref{I:av:grad}} \lesssim \|\D\phi\|_*^2$ and $B_{\ref{I:av:Dt}} \lesssim \|\D\phi\|_*^2$ by~\cite[Theorem 2.1]{Str80}; this shows \cref{I:D} implies \cref{I:av}.

	For the converse we proceed via an analogue of the proof of~\cite[Theorem 2.6]{Str80}.
	Consider
	\begin{equation*}
		\hat{\theta}(\xi,\tau) = \|(\xi,\tau)\| \hat{\zeta}(\xi,\tau),
	\end{equation*}
	where $\zeta \in C^\infty_0(\R)$.
	Let $H^1_{00}$ be the dense subclass of continuous $H^1$ functions $g$ such that $g$ and all its derivatives decay rapidly, see~\cite[p.~225]{Ste70}.
	Via an analogue of~\cite[Theorem 3]{FS72}, \cite[Lemma 2.7]{Str80} by assuming \cref{I:av:grad,I:av:Dt} if $g \in H^1_{00}(\R^n)$ then for each $1 \leq k \leq n-1$
	\begin{equation}
		\label{E:equiv:duality:k}
		\left|\int_{\mathbb{S}^{n-1}} \int_0^\infty \iint\limits_{\R^{n-1} \times \R}
		\psi^u_\rho * \varphi^k_\rho * F(x,t) \theta_\rho * g(x,t) \dx\dt\frac{\drho}{\rho}\du \right| \lesssim B^{1/2}_{\ref{I:av:grad}} \|g\|_{H^1},
	\end{equation}
	and
	\begin{equation}
		\label{E:equiv:duality:t}
		\left|\int_{\mathbb{S}^{n-1}} \int_0^\infty \iint\limits_{\R^{n-1} \times \R}
		\psi^u_{n,\rho} * F(x,t) \theta_\rho * g(x,t) \dx\dt\frac{\drho}{\rho}\du \right| \lesssim B^{1/2}_{\ref{I:av:Dt}} \|g\|_{H^1}.
	\end{equation}

	For $1 \leq k \leq n-1$ let
	\begin{equation}
		\label{E:equiv:m}
		\begin{split}
			m_k(\xi,\tau) &= \int_{\mathbb{S}^{n-1}} \int_0^\infty
			\overline{\hat{\psi}^u \left(-\rho\xi,-\rho^2 \tau\right) \hat{\varphi}^k \left(-\rho\xi,-\rho^2 \tau\right)}
			\|(\xi,\tau)\| \zeta(\rho\|(\xi,\tau)\|)
			\drho\du, \\
			m_n(\xi,\tau) &= \int_{\mathbb{S}^{n-1}} \int_0^\infty
			\overline{\hat{\psi}_n^u \left(-\rho\xi,-\rho^2 \tau\right)}
			\|(\xi,\tau)\| \zeta(\rho\|(\xi,\tau)\|)
			\drho\du.
		\end{split}
	\end{equation}
	All of these functions $m_i$ are homogeneous of degree zero, smooth away from the origin and the associated Fourier multipliers $M_k$, for $1 \leq k \leq n$, are Caldor\'on-Zygmund operators that preserve the class $H^1_{00}$ and are bounded on $H^1$.

	The non-degeneracy condition from~\cite{CT75} on the family of functions $\{m_k\}_{k=1}^n$ holds --- that is the property that $\sum_k |m_k(r \xi, r^2\tau)|^2$ does not vanish identically in $r$ for $(\xi,\tau) \neq (0,0)$.
	Therefore by~\cite{CT75,CT77} we can find smooth homogeneous of degree zero functions $u_{k,j}(\xi,\tau)$ and positive numbers $r_j$ such that for all $(\xi,\tau) \neq (0,0)$
	\begin{equation}
		\label{E:equiv:mu=1}
		\sum_{k=1}^{n}\sum_{j=1}^{j_0} m_{k,r_j}(\xi,\tau) u_{k,j}(\xi,\tau) = 1,
	\end{equation}
	where $m_{k,r_j}$ are as $m_k$ but with $r_j\rho$ replacing $\rho$ in the arguments of $\hat{\psi}^u$, $\hat{\varphi}^k$ and $\hat{\psi}_2^u$ in \cref{E:equiv:m} (but not $\zeta$).

	Let $M_{k,j}$ and $U_{k,j}$ be the associated Fourier multiplier operators to their respective multipliers $m_{k,r_j}$ and $u_{k,j}$ then $\sum\sum M_{k,j} U_{k,j} g  = g$ for all $g \in H^1_{00}$.
	By~\cite[Theorem 3]{FS72}, \cite[Lemma 2.7]{Str80} there exists $h_{k,j} \in \BMO(\R^n)$ such that $\|h_{k,j}\|_*^2 \lesssim B_{\ref{I:av:grad}}$ or $B_{\ref{I:av:Dt}}$, and $(h_{k,j},g) = (F, M_{k,j}g)$ for all $g \in H^1_{00}$.
	If we replace $g$ by $U_{j,k}g \in H^1_{00}$ in the previous identity and sum over $j$ and $k$ we obtain $(h,g) = (F,g)$ for all $g \in H^1_{00}$ where $h = \sum_{k,j} U^*_{k,j}h_{k,j}$; furthermore by the $\BMO$ condition on $h_{k,j}$, $\|h\|_*^2 \lesssim B_{\ref{I:av:grad}} + B_{\ref{I:av:Dt}}$.
	The identity~\cref{E:equiv:mu=1} does not need to hold at the origin therefore $\hat{h} - \hat{F}$ may be supported at the origin and hence $F = h + p$ where $p$ is a polynomial.
	Due to the assumption $\phi\in \Lip(1,1/2)$ clearly $F$ must be a tempered distribution. Hence as in~\cite{Str80} we may conclude $F = h \in \BMO(\R^n)$.  This implies equivalence of \cref{I:D,I:av}.

	Similarly we may prove the equivalence of \cref{I:Dt:riv,I:diff} to \cref{I:D}. The changes needed are outlined below.
	\paragraph{We first look at \cref{I:diff} $\iff$ \cref{I:D}}
	In this instance we replace the convolutions $\varphi^k * \psi^u$ by
	\begin{equation*}
		\hat{\psi}^u_1(\xi,\tau) = \frac{e^{i(\xi,0)\cdot u} - 2 - e^{-i(\xi,0)\cdot u}}{\|(\xi,\tau)\|},
	\end{equation*}
	which corresponds to \cref{I:diff:grad}, and we keep the convolution $\psi^u_n$ as it is in~\cref{E:equiv:phi:n}.
	The same proof then goes through to give that \cref{I:diff} holds if and only if \cref{I:D} holds with equivalent norms, as in~\cref{E:equiv:norms}.

	\paragraph{\Cref{I:Dt:riv} $\iff$ \cref{I:D}}
	This case is stated in~\cite[Proposition 3.2]{Riv03}. Again the proof proceeds as above with  one convolution
	\begin{equation*}
		\hat{\psi}^u(\xi,\tau) = \frac{e^{i(\xi,\tau)\cdot u} - 2 - e^{-i(\xi,\tau)\cdot u}}{\|(\xi,\tau)\|}. \qedhere
	\end{equation*}
\end{proof}

\begin{proof}[Proof of \cref{T:ext}] Without loss of generality we only consider the case $\eta<1$. When $\eta\ge 1$ the existence of a extension with $\|\D\Phi\|_{*}\lesssim \eta + \ell$ requires a much simpler argument.

	By~\eqref{E:ext:1} there exists $f \in C_\delta$ such that $\|\nabla\phi - f\|_{*, Q_{8d}} \leq 2 \eta$ and a scale $0<r_0 = r_0(\delta) \leq d$ such that
	\[
		\|f\|_{*,Q_{8d},r_0} \leq 2 \eta.
	\]

	Let $d'=\eta\min(r_0,r_1)/2$ and consider some $r\le d'$ and $Q_r\subset Q_{4d}$. Find a natural number $k$ such that $R=2^kr$ and $R\eta/2 < r \leq R\eta$. By our choice of $d'$ the cube $Q_{2R}$ which is an enlargement of $Q_r$ by a factor $2^{k+1}$ is still contained in the original cube $Q_{8d}$.

	It follows that
	\begin{equation*}
		\centering
		\begin{gathered}
			\|\nabla\phi \|_{*,Q_{2R}} \lesssim \eta
			\quad \text{and} \quad \\
			\sup_{\substack{Q_s = J_s \times I_s \\ Q_s \subset Q_{2R}}} \frac{1}{|Q_s|}\int_{Q_s}\int_{I_s}
			\frac{|\phi(x,t) - \phi(x,\tau)|^2}{|t-\tau|^2} \dtau\dt\dx \leq \eta^2.
		\end{gathered}
	\end{equation*}

	Without loss of generality we may now assume that the cube $Q_{2R}$ is centred at the origin $(0,0)$ and that $\phi(0,0) = 0$, since the BMO norm is invariant under translation and ignores constants.
	We first define $\tilde\phi$ as	an extension in time via reflection and tiling of the cube $Q_r$:
	\begin{equation}%
		\label{E:BMO:ext:time}
		\tphi(x, t) =
		\begin{cases}
			\phi(x, t)        & t \in [-r^2,r^2] + 4kr^2, \\
			\phi(x, 2r^2 - t) & t \in [r^2, 3r^2]  + 4kr^2, \,k \in \Z.
		\end{cases}
	\end{equation}
	See \cref{F:ext} on \cpageref{F:ext} for an illustration of this. Clearly $\tilde\phi$ coincides with $\phi$ on $Q_r$.

	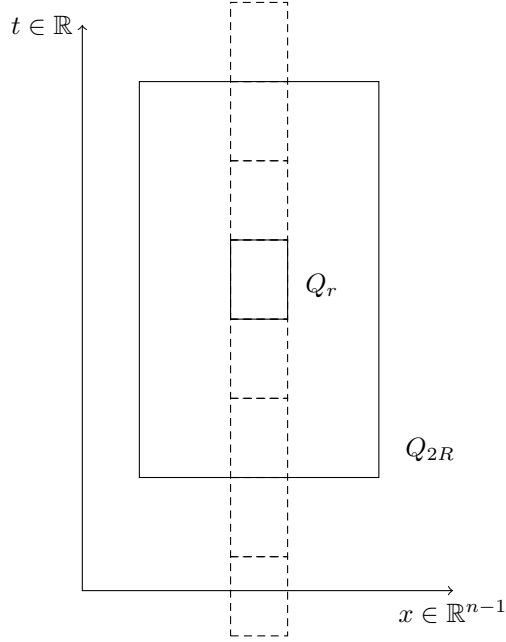
\begin{figure}
		\centering
		\begin{tikzpicture}[scale=0.75]
			\draw  (1.1,2.8) rectangle ++(1,1.4);
			\node at (2.7,3.4) {$Q_r$};

			\foreach \x in {-2.8,-1.4,...,8}
			\draw [densely dashed] (1.1, \x) rectangle ++(1, 1.4);

			\draw (-0.5,0) rectangle (3.7,7);
			\node at (4.6,0.5) {$Q_{2R}$};

			\draw[<->] (5,-2) node[below] {$x \in \R^{n-1}$} -- (-1.5,-2) -- (-1.5,8) node[left] {$t \in \R$};
		\end{tikzpicture}
		\caption{The reflection and tiling of the cube $Q_r \subset Q_{2R}$ defined in~\cref{E:BMO:ext:time}.}%
		\label{F:ext}
	\end{figure}

	It follows that  $\tilde\phi$ is a function $\tphi : \{|x|_\infty < 2R\} \times \R \to \R$ and $(\nabla\tphi)_{Q_{r}} = (\nabla\phi)_{Q_{r}}$. Consider a cut off function $\rho$ such that
	\[
		\rho(x) =
		\begin{cases}
			1 & \text{ if } |x|_\infty < r \\
			0 & \text{ if } |x|_\infty > 2R,
		\end{cases}
	\]
	and $|\nabla\rho| \lesssim 1/R \lesssim \eta/r$.
	Finally define
	\begin{equation}
		\label{E:BMO:ext}
		\Phi = \tphi\rho + (1-\rho)(x \cdot (\nabla\tphi)_{Q_{r}}).
	\end{equation}
	Clearly $\Phi$	 is well defined on ${\mathbb R}^{n-1}\times\mathbb R$ as $\rho=0$ outside the support of $\tilde\phi$. We claim that $\Phi$ satisfies $(i)$-$(iv)$ of \cref{T:ext} which we establish in a sequence of lemmas below. Observe also that from our definition of $\Phi$ we have
	\begin{equation}
		\label{E:BMO:ext:nabla}
		\nabla\Phi = \left(\nabla\tphi - (\nabla\tphi)_{Q_{r}}\right) \rho + \nabla\rho (\tphi - x \cdot (\nabla\tphi)_{Q_{r}}) + (\nabla\tphi)_{Q_{r}}.
	\end{equation}
\end{proof}

We start with  couple of lemmas that allows us to reduce our claim to the dyadic case;
this is to make the geometry easier to handle.

\begin{lemma}[{\cite[Lemma 2.3]{Jon80}, c.f.~\cite[Theorem 2.8]{Str80}}]%
	\label{L:BMO:dyadic}
	Let $f$ be defined on $\R^n$ and
	\begin{equation}
		\label{E:BMO:dyadic}
		\sup_Q \frac{1}{|Q|} \int_Q |f - f_Q| \leq c(\eta),
	\end{equation}
	where the supremum is taken over all dyadic cubes $Q \subset \R^n$.
	Further, assume that
	\begin{equation}
		\label{E:BMO:adj}
		\sup_{Q_1, Q_2} |f_{Q_1} - f_{Q_2}| \leq c(\eta),
	\end{equation}
	where the supremum is taken over all dyadic cubes $Q_1$, $Q_2$ of equal edge length with a touching edge.
	Then
	\[
		\|f\|_* \lesssim c(\eta).
	\]
\end{lemma}

Below $l(Q_s) = s$ denotes the radius of a parabolic cube.

\begin{lemma}[{\cite[Lemma 2.1 and pp.\ 44-45]{Jon80}}]%
	\label{L:BMO:dist}
	Let $f \in BMO(Q)$ and $Q_0 \subset Q_1 \subset Q$ then
	\begin{equation}
		\label{E:BMO:dist}
		|f_{Q_0} - f_{Q_1}| \lesssim \log\left(2 + \frac{l(Q_1)}{l(Q_0)}\right) \|f\|_{*, Q}.
	\end{equation}
	Furthermore, the same proof in~\cite{Jon80} gives the following slightly stronger result
	\begin{equation}
		\label{E:BMO:dist2}
		\frac{1}{|Q_0|}\int_{Q_0} |f - f_{Q_1}|
		\lesssim \log\left( 2 + \frac{l(Q_1)}{l(Q_0)} \right) \|f\|_{*, Q}.
	\end{equation}
	If $Q_0, Q_1 \subset Q$ and $l(Q_0) \leq l(Q_1)$ but they are not necessarily nested then
	\begin{equation}
		\label{E:BMO:dist:general}
		|f_{Q_0} - f_{Q_1}| \lesssim \left(\log\left(2 + \frac{l(Q_1)}{l(Q_0)}\right) + \log\left[2 + \frac{\dist(Q_0, Q_1)}{l(Q_1)}\right]\right)\|f\|_{*, Q}.
	\end{equation}
	If the cubes $Q_0, Q_1$ and $Q$ are dyadic then we may replace $\BMO$ by dyadic~$\BMO$.
\end{lemma}
There is a typo at the top of~\cite[p.~45]{Jon80}.
It should read $l(Q_k) \leq l(Q_j)$ (it currently reads the converse).

\begin{claim}%
	\label{C:BMO:ext}
	Let $\tphi$ be defined as in \cref{E:BMO:ext:time}, $\|\nabla\phi\|_{*,Q_{2R}} \lesssim \eta$, and let $Q$ be dyadic with $r \leq l(Q) \leq 2R$ then
	\begin{equation}
		\label{E:C:BMO:ext}
		\frac{1}{|Q|} \int_{Q} |\nabla\tphi - \nabla\tphi_{Q_r}| \lesssim_\epsilon \eta^{1-\epsilon}.
	\end{equation}
\end{claim}

\begin{proof}[Proof of claim.]
	Let $N \in \N$ be such that $l(Q) = 2^N l(Q_r)$.
	Let $\{Q^i\}$ be the $2^{N(n-1)}$ dyadic cubes that are translations of $Q_r$ and partition $Q \cap \{|t| \leq r^2 \}$. Then by \cref{L:BMO:dist}
	\begin{align*}
		\frac{1}{|Q|} \int_{Q} |\nabla\tphi - \nabla\tphi_{Q_r}|
		 & = \sum_i \frac{2^{2N} |Q^i|}{|Q|} \frac{1}{|Q^i|}\int_{Q^i} |\nabla\tphi - \nabla\tphi_{Q_r}| \\
		 & \leq \sum_i \frac{2^{2N} |Q^i|}{|Q|} \left(\frac{1}{|Q^i|}\int_{Q^i} |\nabla\phi - \nabla\phi_{Q^i}| + |\nabla\phi_{Q^i} - \nabla\phi_{Q_r}| \right) \\
		 & \lesssim \left(\eta + \eta \log(2 + R/r) \right)
		\lesssim \eta + \eta\log(1 + 1/\eta)
		\lesssim_\epsilon \eta^{1-\epsilon}. \qedhere
	\end{align*}
\end{proof}

\begin{lemma}[{\cite{Ste76}}]%
	\label{L:St76}
	Let $g, h \in L^1_{\loc}$ then
	\begin{equation}
		\frac{1}{|Q|} \int_Q |gh - (gh)_Q|
		\leq \frac{2}{|Q|} \int_Q |g(h-h_Q)| + \frac{|h_Q|}{|Q|} \int_Q |g-g_Q|.
	\end{equation}
\end{lemma}
\begin{proof}
	This small reduction is from~\cite[p.~582]{Ste76}.
	First observe
	\[
		gh - (gh)_Q = g(h - h_Q) + h_Q(g- g_Q) + g_Q h_Q - (gh)_Q
	\]
	and
	\[
		|g_Q h_Q - (gh)_Q| = \left|\frac{1}{|Q|}\int_Q g h_Q - \frac{1}{|Q|}\int_Q gh\right|
		\leq \frac{1}{|Q|}\int_Q |g(h-h_Q)|
	\]
	hence
	\begin{align}
		\label{E:BMO:reduction1}
		\left| \frac{1}{|Q|} \int_Q |gh - (gh)_Q| - \frac{|h_Q|}{|Q|} \int_Q |g-g_Q| \right|
		 & \leq 2\frac{1}{|Q|} \int_Q |g(h-h_Q)|.
	\end{align}
\end{proof}

We can now prove \cref{I:ext:nabla} of \cref{T:ext}.
\begin{lemma}%
	\label{L:BMO:ext:nabla}
	Let $\Phi : \R^n \to \R$ be defined as in \cref{E:BMO:ext} with $\|\nabla\phi\|_{*,Q_{2R}} \lesssim \eta$ then $\nabla\Phi \in \BMO(\R^n)$ and for all $0 < \epsilon < 1$
	\begin{equation}
		\label{E:BMO:ext:nabla:bound}
		\|\nabla\Phi\|_*
		\lesssim_\epsilon \eta^{1-\epsilon} + \eta\ell.
	\end{equation}
\end{lemma}

\begin{proof}
	Recall
	$\nabla\Phi = \left(\nabla\tphi - (\nabla\tphi)_{Q_{r}}\right) \rho + \nabla\rho (\tphi - x \cdot (\nabla\tphi)_{Q_{r}}) + (\nabla\tphi)_{Q_{r}}$;
	we can ignore the constant  term as  the BMO norm doesn't see it.
	Let $\psi = \nabla\tphi - (\nabla\tphi)_{Q_{r}}$ and $\theta = \tphi - x \cdot (\nabla\tphi)_{Q_{r}}$. We want to bound $\| \rho\psi \|_*$ and $\| \nabla\rho \theta \|_*$.
	We first tackle the term $\| \rho\psi \|_*$.

	\begin{steps}
		\item \cref{E:BMO:adj} holds: $\sup_{Q_1, Q_2} |(\rho\psi)_{Q_1} - (\rho\psi)_{Q_2}| \leq c(\eta)$ for $Q_1, Q_2$ dyadic cubes of equal side length and with a touching edge.

		Since $\tphi$ is the extension in the time direction by reflection and tiling (c.f. \cref{E:BMO:ext:time}), and $Q_1$, $Q_2$ and $Q_{r}$ are all dyadic cubes we may assume that if $l(Q_1) \leq r$ then $Q_1, Q_2 \subset \{|t| < r^2\}$, and if $l(Q_1) > r$ then $\{|t| < r^2\} \subset Q_1$.

		If $Q_1, Q_2 \subset Q_{2R}$ then $|(\rho\psi)_{Q_1} - (\rho\psi)_{Q_2}| \lesssim \|\rho\psi\|_{\dbmo,\, Q_{2R}}$.
		Therefore, if we show~\cref{E:BMO:dyadic} for $f=\rho\psi$ then by \cref{L:BMO:dist,L:St76} clearly
		\[
			|(\rho\psi)_{Q_1} - (\rho\psi)_{Q_2}| \lesssim \|\rho\psi\|_{\dbmo, \,Q_{2R}} \leq \|\psi\|_{\dbmo, \,Q_{2R}} \leq \|\nabla\tphi\|_{\dbmo, \,Q_{2R}} \lesssim \eta.
		\]

		Now look at the other cases: $Q_1 \subset Q_{2R}$ and $Q_2 \cap Q_{2R} = \emptyset$, or $Q_{2R} \subset Q_1$ and $Q_2 \cap Q_{2R} = \emptyset$.
		In both cases we wish to control $|(\rho\psi)_{Q_1}|$.

		\begin{steps}
			\item Case $Q_1 \subset Q_{2R}$, $Q_2 \cap Q_{2R} = \emptyset$ and $l(Q_1) \lesssim \frac{R\eta}{\ell}$.%
			\label{S:BMO:ext:diff:small}

			$Q_1$ is small here and touches the boundary of $Q_{2R}$.
			This means that $\|\rho\|_{L^\infty(Q_1)} \lesssim \frac{l(Q_1)}{R}$ since $\rho$ is 0 outside $Q_{2R}$.
			Therefore we just apply the trivial bound
			\[
				|(\rho\psi)_{Q_1}| \leq \|\rho\|_{L^\infty(Q_1)} \|\psi\|_{L^\infty(Q_1)} \lesssim \frac{l(Q_1)}{R} \ell
				\lesssim \eta.
			\]

			\item Case $Q_1 \subset Q_{2R}$, $Q_2 \cap Q_{2R} = \emptyset$ and $\frac{R\eta}{\ell} \lesssim l(Q_1) \leq 2R$.%
			\label{S:BMO:ext:diff:mid}

			Since $Q_1 \subset Q_{2R}$ we have $\frac{R\eta}{\ell} \lesssim l(Q_1) \leq 2R$.
			$Q_1$ is dyadic so there exists $N \in \Z$ such that $l(Q_1) = 2^N l(Q_r)$.
			\begin{steps}
				\item $N \leq 0$:

				This means that $l(Q_1) \leq l(Q_r)$ and so by the reflection and tiling in time, \cref{E:BMO:ext:time},  we may assume $Q_1 \subset \{|t| \leq r^2\}$ and by \cref{L:BMO:dist}
				\begin{align*}
					|(\rho\psi)_{Q_1}|
					 & \leq |\psi|_{Q_1}
					= \frac{1}{|Q_1|}\int_{Q_1} |\nabla\phi - \nabla\phi_{Q_{r}}|
					\leq \frac{1}{|Q_1|}\int_{Q_1} |\nabla\phi - \nabla\phi_{Q_{1}}| + |\nabla\phi_{Q_1} - \nabla\phi_{Q_{r}}| \\
					 & \lesssim \eta + \eta \log \left( 1 + \ell \right) + \eta\log(1+1/\eta)
					\lesssim_\epsilon \eta^{1-\epsilon} + \eta \log \left( 1 + \ell \right) .
				\end{align*}

				\item $N > 0$:

				By \Cref{C:BMO:ext} we obtain
				\begin{align*}
					|(\rho\psi)_{Q_1}|
					\leq |\psi|_{Q_1}
					= \frac{1}{|Q_1|}\int_{Q_1} |\nabla\tphi - \nabla\tphi_{Q_{r}}|
					\lesssim_\epsilon \eta^{1-\epsilon}.
				\end{align*}

			\end{steps}
			\item  Case $Q_{2R} \subset Q_1$, $Q_2 \cap Q_{2R} = \emptyset$ so $l(Q_1) \geq 2R$.%
			\label{S:BMO:ext:diff:large}

			Let $N$ satisfy $l(Q_1) = 2^N l(Q_{2R})$, the number of dyadic generations separating $Q_1$ and $Q_{2R}$.
			Then $Q_1$ overlaps $Q_{2R}$ (and its dyadic translates in time) exactly $2^{2N}$ times.
			Therefore by \cref{C:BMO:ext},
			\[
				|(\rho\psi)_{Q_1}|
				\leq |\psi|_{Q_1}
				\leq \frac{2^{2N}}{|Q_1|}\int_{Q_{2R}} |\nabla\tphi - \nabla\tphi_{Q_{r}}|
				\leq \frac{2^{2N}}{2^{N(n+1)}} \frac{1}{|Q_{2R}|}\int_{Q_{2R}} |\nabla\tphi - \nabla\tphi_{Q_{r}}|
				\lesssim_\epsilon \eta^{1-\epsilon}.
			\]
			Hence, modulo the unproved statement $\|\rho\psi\|_{\dbmo, \, Q_{2R}} \lesssim \eta$ we have shown
			\[
				|(\rho\psi)_{Q_1} - (\rho\psi)_{Q_2}| \lesssim_\epsilon \eta^{1-\epsilon} + \eta\log(1+\ell).
			\]
		\end{steps}

		\item \cref{E:BMO:dyadic} holds, that is: $\|\rho\psi\|_{\dbmo} \lesssim c(\eta)$.

		To apply \cref{L:St76} we need to control two terms
		\[
			\sup_{Q \text{ dyadic}} \|\rho\|_{L^\infty(Q)} \frac{1}{|Q|}\int_Q |\psi - \psi_Q|
		\]
		and
		\[
			\sup_{Q \text{ dyadic}} \frac{|\psi_Q|}{|Q|} \int_Q |\rho-\rho_Q|.
		\]

		\begin{steps}
			\item Estimating $\displaystyle \sup\limits_{Q \text{ dyadic}} \|\rho\|_{L^\infty(Q)} \frac{1}{|Q|}\int_Q |\psi - \psi_Q|$.

			In all the following cases we bound $ \|\rho\|_{L^\infty(Q)}  \leq 1$.

			\begin{steps}
				\item Case $l(Q) \leq r$.

				As before, by the reflection and tiling in time, we may assume $Q \subset \{|t| \leq r^2\}$ and so $\nabla\tphi = \nabla\phi$ on $Q$.
				Hence
				\begin{align*}
					\frac{1}{|Q|}\int_Q |\psi - \psi_Q|
					= \frac{1}{|Q|}\int_Q |\nabla\tphi - (\nabla\tphi)_{Q}|
					= \frac{1}{|Q|}\int_Q |\nabla\phi - (\nabla\phi)_{Q}|
					\lesssim \eta.
				\end{align*}

				\item Case $r < l(Q) \leq 2R$.

				Applying \cref{C:BMO:ext} gives
				\[
					\frac{1}{|Q|}\int_Q |\psi - \psi_Q| \leq |\psi|_Q \lesssim_\epsilon \eta^{1-\epsilon}.
				\]

				\item Case $2R < l(Q)$.

				From \cref{S:BMO:ext:diff:large} it follows that
				\[
					\frac{1}{|Q|}\int_Q |\psi - \psi_Q| \leq |\psi|_Q \lesssim_\epsilon \eta^{1-\epsilon}.
				\]
			\end{steps}

			\item Estimating $\displaystyle \sup\limits_{Q \text{ dyadic}} \frac{|\psi_Q|}{|Q|} \int_Q |\rho-\rho_Q|$.

			We have the following three cases to consider.
			\begin{steps}
				\item Case $Q \subset Q_{2R}$, $l(Q) \leq r$ and $Q \subset \{|t| \leq r^2\}$.%
				\label{S:BMO:ext:IVT}

				Because the cube $Q$ might not be touching the boundary we can't follow \cref{S:BMO:ext:diff:small} and bound $\frac{1}{|Q|} \int_Q |\rho-\rho_Q|$ by $\|\rho\|_{L^\infty(Q)}$, which here is likely be 1.
				However, we can use the mean value theorem and get a better bound.
				By the intermediate value theorem there exists $(z,\tau) \in Q$ such that $\rho(z) = \rho_Q$ and using that $\rho$ is independent of time and $|\nabla\rho| \lesssim 1/R$ we have
				\[
					|\rho(x) - \rho_Q| = |\rho(x) - \rho(z)| \leq |\nabla\rho| l(Q)
					\lesssim \frac{l(Q)}{R} \leq \frac{l(Q)}{r}.
				\]

				Then applying \cref{L:BMO:dist} gives
				\begin{align*}
					\frac{|\psi_Q|}{|Q|} \int_Q |\rho-\rho_Q|
					 & \lesssim \frac{l(Q)}{r} \left| \frac{1}{|Q|} \int_Q \nabla\tphi - \nabla\tphi_{Q_r} \right|
					\leq \frac{l(Q)}{r} \frac{1}{|Q|} \int_Q |\nabla\phi - \nabla\phi_{Q_r}| \\
					 & \lesssim \frac{l(Q)}{r} \log\left( 2 + \frac{r}{l(Q)} \right) \eta
					\lesssim\eta.
				\end{align*}

				\item Case $Q \subset Q_{2R}$ and $r < l(Q) \leq 2R$.

				This case is a straightforward application of \cref{C:BMO:ext}
				\[
					\frac{|\psi_Q|}{|Q|} \int_Q |\rho-\rho_Q|
					\leq |\psi_Q|
					\lesssim_\epsilon \eta^{1-\epsilon}.
				\]

				\item Case $Q_{2R} \subset Q$ so $l(Q) > 2R$.

				This follows similarly to \cref{S:BMO:ext:diff:large}; let $N$ be defined as there and
				\begin{align*}
					\frac{|\psi_Q|}{|Q|} \int_Q |\rho-\rho_Q|
					 & \leq \frac{1}{|Q|} \left| \int_Q \nabla\phi - \nabla\phi_{Q_{2R}} \right| \\
					 & \leq \frac{2^{2N}}{2^{N(n+1)}} \| \nabla \phi \|_{*, Q_{2R}}
					\leq \eta.
				\end{align*}
			\end{steps}
		\end{steps}
		Therefore by \cref{L:BMO:dyadic}, $\|\rho\psi\|_* \lesssim_\epsilon \eta^{1-\epsilon} + \eta\log(1+\ell)$.\vglue1mm

		It remains to tackle the harder piece $\nabla\rho\theta = \nabla\rho(\tphi - x\cdot\nabla\tphi_{Q_r})$.
		Recall that $\supp(\nabla\rho) = \{r \leq |x|_\infty \leq 2R\}$.

		\item \cref{E:BMO:adj} holds; that is: $\sup_{Q_1, Q_2} |(\nabla\rho\theta)_{Q_1} - (\nabla\rho\theta)_{Q_2}| \leq c(\eta)$ where $Q_1, Q_2$ are dyadic with a touching edge and $l(Q_1)=l(Q_2)$.

		There are two different cases to consider:
		\begin{caselist}
			\item\label{I:ext:nabla:S3:1} 	$Q_1 \cap \supp(\nabla\rho) \neq \emptyset$ and $Q_2 \cap \supp(\nabla\rho) \neq \emptyset$
			\item\label{I:ext:nabla:S3:2} 	$Q_1 \cap \supp(\nabla\rho) \neq \emptyset$ and $Q_2 \cap \supp(\nabla\rho) = \emptyset$
		\end{caselist}
		Again \cref{I:ext:nabla:S3:1} is controlled by $\|\nabla\rho\theta\|_{\dbmo, \, Q_{2R}}$ by \cref{L:BMO:dist} so we only have to deal with \cref{I:ext:nabla:S3:2} and bound $\sup\limits_{Q_1 \text{ dyadic}} |(\nabla\rho\theta)_{Q_1}|$.
		\begin{steps}
			\item Case $Q_1 \subset Q_{2R}$ and $l(Q_1) \lesssim \frac{R\eta}{\ell}$.%
			\label{s:BMO:ext:easy}

			In this case $Q_1$ touches the boundary of the support of $\nabla\rho$ so we have the estimate $\|\nabla\rho\|_{L^{\infty}(Q_1)} \lesssim \frac{l(Q_1)}{R^2}$ since $|\nabla^2 \rho| \lesssim 1/R^2$.
			Also $\phi(0,0) = 0$ and $\phi \in \Lip(1,1/2)$ so $\|\tphi(x,t)\|_{L^{\infty}(Q_1)} \leq \|\phi(x,t)\|_{L^{\infty}(Q_{2R})} \lesssim \ell R$.
			Finally $\| x\cdot \nabla\tphi_{Q_{r}}\|_{L^{\infty}(Q_{2R})} \lesssim \ell R $.
			Therefore
			\begin{align*}
				|(\nabla\rho\theta)_{Q_1}|
				\leq \|\nabla\rho\|_{L^{\infty}(Q_1)} |\theta|_{Q_1}
				 & \lesssim \frac{l(Q_1)}{R^2} \frac{1}{|Q_1|} \int_{Q_1} |\tphi(x,t) - x\cdot\nabla\tphi_{Q_r}|\dx\dt \\
				 & \lesssim \frac{l(Q_1)}{R^2} \ell R
				\lesssim \eta.
			\end{align*}

			\item Case $Q_1 \subset Q_{2R}$ and $\frac{R\eta}{\ell} \lesssim l(Q_1) \leq 2R$.%
			\label{s:BMO:ext:hard}

			By the fundamental theorem of calculus we may write
			\[
				\tphi(x,t) - \tphi\left(r\frac{x}{|x|}, t\right) = x \cdot \int_{r/|x|}^1 \nabla\tphi(\lambda x,t) \dlambda.
			\]
			Therefore
			\begin{align*}
				|(\nabla\rho \theta)_{Q_1}|
				 & \leq |\nabla\rho| |\theta|_{Q_1} \\
				 & = |\nabla\rho| \left| \tphi\left(r\frac{x}{|x|}, t \right) + x\cdot \int_{r/|x|}^1 \left(\nabla\tphi(\lambda x,t) - \nabla\tphi_{Q_r} \right)\dlambda + x\cdot \frac{r}{|x|} \nabla\tphi_{Q_r} \right|_{Q_1} \\
				 & \lesssim \frac{1}{R} \left\|\tphi\left(r\frac{x}{|x|}, t\right)\right\|_{L^\infty(Q_1)} \\
				 & \qquad+ \frac{R}{R} \frac{1}{|Q_1|} \int_{Q_1} \left(\int_{r/|x|}^1 |\nabla\tphi(\lambda x,t) - \nabla\tphi_{Q_r}| \dlambda\right) \dx\dt
				+ \frac{\eta R\ell}{R}.
			\end{align*}
			Since $\tphi$ defined by \cref{E:BMO:ext:time} is tiled and reflected in time on cubes of scale $r$, and $(rx/|x|, 0) \in Q_r$ we control the first term above by
			\[
				\frac{1}{R} \left\|\tphi\left(r\frac{x}{|x|}, t\right) - 0\right\|_{L^\infty(Q_1)}
				\leq \frac{1}{R} \|\phi - \phi(0,0)\|_{L^\infty(Q_r)}
				\lesssim \frac{\ell r}{R}
				\lesssim \ell \eta.
			\]
			Recall that $r \sim \eta R$, $\frac{R\eta}{\ell} \lesssim l(Q_1) \leq 2R$ and $r \leq |x|_\infty \leq 2R$ so $\eta/2 \leq \lambda \leq 1$.
			We apply Fubini to the second term
			\begin{align*}
				\frac{1}{|Q_1|} & \int_{Q_1} \left(\int_{r/|x|}^1 |\nabla\tphi(\lambda x,t) - \nabla\tphi_{Q_r}| \dlambda\right) \dx\dt \\
				                & \leq \frac{1}{|Q_1|} \int_{\eta/2}^1 \int_{Q_1} |\nabla\tphi(\lambda x,t) - \nabla\tphi_{Q_r}| \dx\dt\dlambda.
			\end{align*}

			Let $\tilde{Q}_1$ be the set formed by $Q_1$ under the transformation $(x,t) \mapsto (\lambda x,t)$.
			We may further cover $\tilde{Q}_1$ by $\sim \lambda^{-2}$ translations of $\lambda Q_1$ with $|\lambda Q_1|/|\tilde{Q}_1| \lesssim \lambda^2$.
			Therefore a similar proof to \cref{C:BMO:ext}, using \cref{L:BMO:dist}, gives
			\begin{align*}
				 & \frac{1}{|Q_1|} \int_{Q_1} |\nabla\tphi(\lambda x,t) - \nabla\tphi_{Q_r}| \dx\dt
				= \frac{1}{|\tilde{Q}_1|} \int_{\tilde{Q}_1} |\nabla\tphi - \nabla\tphi_{Q_r}| \\
				 & \lesssim \lambda^{-2} \frac{\lambda^2}{|sQ_1|} \int_{sQ_1} |\nabla\tphi - \nabla\tphi_{Q_r}|
				\lesssim \eta \log\left(2 + \frac{r}{sl(Q_1)}\right)
				\lesssim \eta\log \left(1 + \frac{\ell}{\eta^2}\right)\\
				 & \lesssim_\epsilon \eta^{1-\epsilon} + \eta\log(1 + \ell)
			\end{align*}
			and hence after harmlessly integrating in $\lambda$ we can control the second term by
			\[
				\int_{\eta/2}^1 \eta\log \left(1 + \frac{\ell}{\eta^2}\right) \dlambda
				\lesssim_\epsilon \eta^{1-\epsilon} + \eta\log(1 + \ell).
			\]

			\item Case $l(Q_1) \geq 2R$.

			As before in \cref{S:BMO:ext:diff:large}, $|(\nabla\rho\theta)_{Q_1}| \leq |(\nabla\rho\theta)_{Q_{2R}}|$, which can be further controlled by cubes that tile $\supp(\nabla\rho)$.
			Therefore, this case is bounded as in \cref{s:BMO:ext:hard}.
		\end{steps}\vglue1mm

		\item \cref{E:BMO:dyadic} holds; that is: $\|\nabla\rho\theta\|_{\dbmo} \lesssim c(\eta)$

		Here we have 3 cases to consider:
		\begin{caselist}
			\item\label{I:ext:nabla:S4:1}	$Q \subset Q_{2R}$
			\item\label{I:ext:nabla:S4:2}	$Q \subset \R^n \setminus \supp(\nabla\rho)$
			\item\label{I:ext:nabla:S4:3}	$Q_{2R} \subset Q$
		\end{caselist}

		\Cref{I:ext:nabla:S4:2} is obvious.
		\Cref{I:ext:nabla:S4:3} reduces down to \cref{I:ext:nabla:S4:1} by \cref{S:BMO:ext:diff:large}, the reflection and tiling of $\tphi$, and the $\supp(\nabla\rho)$.

		\Cref{I:ext:nabla:S4:1}: Using \cref{L:St76} this reduces down to showing that
		\begin{enumerate}
			\item[(a)]%
			      $\displaystyle \frac{|\theta_Q|}{|Q|} \int_Q |\nabla\rho - (\nabla\rho)_Q| \lesssim c(\eta)$
			\item[(b)]%
			      $\displaystyle \frac{1}{|Q|} \int_Q |\nabla\rho(\theta - \theta_Q)| \lesssim c(\eta)$
		\end{enumerate}
		for $Q$ dyadic and $Q \subset Q_{2R}$.

		\begin{steps}
			\item  (a) holds for $Q$ dyadic and $Q \subset Q_{2R}$.
			\begin{steps}
				\item Case $Q \subset Q_{2R}$ and $l(Q) \lesssim \frac{R\eta}{\ell}$.

				By the naive bounds in \cref{s:BMO:ext:easy} $|\theta|_Q \lesssim \ell R$.
				If we use the mean value theorem for $\nabla\rho$ similar to \cref{S:BMO:ext:IVT} then
				\[
					\frac{1}{|Q|} \int_Q |\nabla\rho - (\nabla\rho)_Q|
					\lesssim |\nabla^2 \rho| l(Q) \lesssim \frac{l(Q)}{R^2}.
				\]
				Therefore
				\[
					\frac{|\theta_Q|}{|Q|} \int_Q |\nabla\rho - (\nabla\rho)_Q|
					\lesssim \ell R \frac{l(Q)}{R^2}
					\lesssim \eta.
				\]

				\item Case $Q \subset Q_{2R}$ and $\frac{R\eta}{\ell} \lesssim l(Q) \leq 2R$.

				Here we apply the same technique as \cref{s:BMO:ext:hard}
				\[
					\frac{|\theta_Q|}{|Q|} \int_Q |\nabla\rho - (\nabla\rho)_Q|
					\leq |\theta|_Q |\nabla\rho|
					\lesssim_\epsilon \eta^{1-\epsilon} + \eta\log(1+\ell).
				\]
			\end{steps}

			\item  (b) holds for $Q$ dyadic and $Q \subset Q_{2R}$.

			\begin{align*}
				\frac{1}{|Q|} \int_Q |\nabla\rho(\theta - \theta_Q)|
				\lesssim \frac{1}{R} \frac{1}{|Q|} \int_Q |\theta - \theta_Q|.
			\end{align*}
			We split this into the now usual cases.
			\begin{steps}
				\item Case $l(Q) \lesssim \frac{R\eta}{\ell}$.

				By the intermediate and mean value theorems $|\tphi -\tphi_Q| \lesssim l(Q)\ell$ and $|x - x_Q| \lesssim l(Q)$ so
				\begin{align*}
					\frac{1}{R} \frac{1}{|Q|} \int_Q |\theta - \theta_Q|
					 & = \frac{1}{R} \frac{1}{|Q|} \int_Q |\tphi - \tphi_Q - x\cdot\nabla\tphi_{Q_r} + (x\cdot\nabla\tphi_{Q_r})_Q| \\
					 & \lesssim \frac{1}{R} l(Q) \ell
					\lesssim \eta.
				\end{align*}

				\item Case $\frac{R\eta}{\ell} \lesssim l(Q) < 2R$.

				\begin{align*}
					\frac{1}{R} \frac{1}{|Q|} \int_Q |\theta - \theta_Q|
					\lesssim \frac{1}{R} |\theta|_Q
				\end{align*}
				then applying the result from \cref{s:BMO:ext:hard} gives
				\[
					\frac{1}{|Q|} \int_Q |\nabla\rho(\theta - \theta_Q)| \lesssim_\epsilon \eta^{1-\epsilon} + \eta\log(1+\ell).
				\]
			\end{steps}
		\end{steps}
	\end{steps}

	Therefore by \cref{L:BMO:dyadic} we have shown $\nabla\Phi \in \BMO(\R^n)$ and the bound \cref{E:BMO:ext:nabla:bound} holds.
\end{proof}

To finish proving \cref{T:ext} we need to establish \cref{I:ext:Dt}.
\begin{lemma}%
	\label{L:BMO:ext:Dt}

	Let $\Phi : \R^{n-1} \times \R \to \R$ be defined in \cref{E:BMO:ext} with
	\begin{equation}
		\label{E:BMO:ext:Dt:assumption}
		\sup_{\substack{Q_s = J_s \times I_s, \\ Q_s \subset Q_{8d}, \, s \leq r_1}} \frac{1}{|Q_s|}\int_{Q_s}\int_{I_s}
		\frac{|\phi(x,t) - \phi(x,\tau)|^2}{|t-\tau|^2} \dtau\dt\dx \leq \eta^2
	\end{equation}
	then $\Phi$ satisfies
	\begin{equation}
		\sup_{Q_s = J_s \times I_s}  \frac{1}{|Q_s|}\int_{Q_s}\int_{I_s}
		\frac{|\Phi(x,t) - \Phi(x,\tau)|^2}{|t-\tau|^2} \dtau\dt\dx \lesssim \eta^2.
	\end{equation}
\end{lemma}

\begin{proof}
	Trivially since $\Phi$ is defined globally
	\begin{equation*}
		\begin{split}
			&\sup_{Q_s = J_s \times I_s}  \frac{1}{|Q_s|}\int_{Q_s}\int_{I_s}
			\frac{|\Phi(x,t) - \Phi(x,\tau)|^2}{|t-\tau|^2} \dtau\dt\dx  \\
			&\leq \sup_{Q_s = J_s \times I_s}  \frac{1}{|Q_s|}\int_{Q_s}\int_{I_s}
			\frac{|\tphi(x,t) - \tphi(x,\tau)|^2}{|t-\tau|^2} \dtau\dt\dx ,
		\end{split}
	\end{equation*}
	where we interpret the value of $\tphi$ where it is undefined as $0$, i.e.\ $\tphi(x,t) = 0$ when $(x,t) \not\in \supp(\tphi)$.
	It remains to establish
	\begin{equation}
		\label{E:BMO:ext:Dt:bdd}
		\begin{split}
			& \sup_{I_s}  \frac{1}{|I_s|}\int_{I_s}\int_{I_s}
			\frac{|\tphi(x,t) - \tphi(x,\tau)|^2}{|t-\tau|^2} \dtau\dt\dx \\
			&\lesssim  \sup_{I_s \subset I_r}  \frac{1}{|I_s|}\int_{I_s}\int_{I_s}
			\frac{|\phi(x,t) - \phi(x,\tau)|^2}{|t-\tau|^2} \dtau\dt\dx
		\end{split}
	\end{equation}
	pointwise in $x$, where $Q_r = J_r \times I_r$ and is used to define $\Phi$ in~\cref{E:BMO:ext}.
	To simplify our notation we drop the dependance on the spatial variables in $\tphi$ and $\phi$. We also set $A:=I_s$.
	Recall from~\cref{E:BMO:ext:time} that
	\begin{equation*}
		\tilde\phi(t) =
		\begin{cases}
			\phi(t)        & t \in [-r^2,r^2] + 4kr^2, \\
			\phi(2r^2 - t) & t \in [r^2, 3r^2]  + 4kr^2,
		\end{cases}
	\end{equation*}
	for $k \in \Z$.
	Let $I_k = [-r^2,r^2] + 4kr^2$ and $J_k = [r^2,3r^2] + 4kr^2$ be intervals in time for $k \in \Z$.
	We partition $A$ into disjoint pieces $A = \cup_i I_i \cup_j J_j \cup A_1 \cup A_2$, where $A_1$ and $A_2$ are pieces that don't contain either $I_i$ or $J_j$.

	If $A=A_1 \cup A_2$ we may as well assume (by translation an reflection) that $A_1 = [a,r^2]$, $A_2 = [r^2, b]$. Let $\tau', b'$ and $A'_2$ be the images of $\tau, b$ and $A_2$ respectively under the map $\tau \mapsto 2r^2-\tau$.
	Without loss of generality we only consider the case $|A_1| > |A_2|$. Since $|t-\tau| = |t-r^2| + |\tau' - r^2| \geq |t-\tau'|$ we have for $t \in A_1$, $\tau \in A_2$
	\begin{align*}
		 & \int_{A_1}\int_{A_2} \frac{|\tilde\phi(t) - \tilde\phi(\tau)|^2}{|t-\tau|^2} \dtau\dt
		= \int_a^{r^2}\int_{b'}^{r^2} \frac{|\phi(t) - \phi(\tau')|^2}{|t - (2t^2 - \tau')|^2} \dtau'\dt \\
		 & \leq \int_a^{r^2}\int_{b'}^{r^2} \frac{|\phi(t) - \phi(\tau')|^2}{|t - \tau'|^2} \dtau' \dt
		\leq \int_{A_1}\int_{A_1} \frac{|\phi(t) - \phi(\tau')|^2}{|t - \tau'|^2} \dtau' \dt.
	\end{align*}
	Therefore
	\begin{align*}
		 & \frac{1}{|A|} \int_A \int_A \frac{|\tilde\phi(t) - \tilde\phi(\tau)|^2}{|t-\tau|^2} \dtau\dt \\
		 & = \frac{1}{|A|} \left(\int_{A_1}\int_{A_1} + 2\int_{A_1}\int_{A_2} + \int_{A_2}\int_{A_2}\right) \frac{|\phi(t) - \phi(\tau')|^2}{|t - \tau'|^2} \dtau' \dt
		\lesssim \eta^2.
	\end{align*}

	In the general case when  $A = \cup_{i \in \mathcal{I}} I_i \cup_{j \in \mathcal J} J_j \cup A_1 \cup A_2$ we write the double integral over $A$ in terms of integrals
	\[
		\sum_{i,k \in \mathcal I} \int_{I_i} \int_{I_k} \frac{|\tilde\phi(t) - \tilde\phi(\tau)|^2}{|t-\tau|^2} \dtau\dt
		\text{, } \quad
		\sum_{i\in \mathcal I, j \in \mathcal J} \int_{I_i} \int_{J_j} \frac{|\tilde\phi(t) - \tilde\phi(\tau)|^2}{|t-\tau|^2} \dtau\dt
	\]
	and integrals that involve sets $A_1$ or $A_2$ or both (those are handled similar to the earlier calculation).

	Dealing with the first case, if $i \neq k$, $t \in I_i$ and $\tau \in I_k$ then $|t-\tau| \sim r^2|i-k|$;
	if $i = k$ then $|t-\tau|=|t'-\tau'|$.
	Therefore
	\begin{align*}
		 & \sum_{i,k \in \mathcal I} \int_{I_i} \int_{I_k} \frac{|\tilde\phi(t) - \tilde\phi(\tau)|^2}{|t-\tau|^2} \dtau\dt \\
		 & \sim \sum_{i \in \mathcal I} \int_{I_0} \int_{I_0} \frac{|\phi(t) - \phi(\tau)|^2}{|t-\tau|^2} \dtau\dt
		+ \sum_{\substack{i,k \in \mathcal I \\ i \neq k}} \frac{1}{r^4|i-k|^2} \int_{I_0} \int_{I_0} |\phi(t) - \phi(\tau)|^2 \dtau\dt \\
		 & \leq \sum_{i \in \mathcal I} \int_{I_0} \int_{I_0} \frac{|\phi(t) - \phi(\tau)|^2}{|t-\tau|^2} \dtau\dt
		+ \sum_{\substack{i,k \in \mathcal I \\ i \neq k}} \frac{1}{|i-k|^2} \int_{I_0} \int_{I_0} \frac{|\phi(t) - \phi(\tau)|^2}{|t-\tau|^2} \dtau\dt \\
		 & \lesssim |\mathcal I| \int_{I_0} \int_{I_0} \frac{|\phi(t) - \phi(\tau)|^2}{|t-\tau|^2} \dtau\dt.
	\end{align*}
	In the second case
	\begin{align*}
		 & \sum_{i\in \mathcal I, j \in \mathcal J} \int_{I_i} \int_{J_j} \frac{|\tilde\phi(t) - \tilde\phi(\tau)|^2}{|t-\tau|^2} \dtau\dt \\
		 & \lesssim \sum_{\substack{i \in \mathcal I, j \in \mathcal J \\ |i-j| \leq 1}} \int_{I_0} \int_{I_0} \frac{|\phi(t) - \phi(\tau)|^2}{|t-\tau|^2}
		+ \sum_{\substack{i \in \mathcal I, j \in \mathcal J \\ |i-j| \geq 2}} \frac{1}{r^4(|i-j|-1)^2} \int_{I_0} \int_{I_0} |\phi(t) - \phi(\tau)|^2 \dtau\dt \\
		 & \lesssim (|\mathcal I| + |\mathcal J|) \int_{I_0} \int_{I_0} \frac{|\phi(t) - \phi(\tau)|^2}{|t-\tau|^2} \dtau\dt.
	\end{align*}
	Since $|A| \sim (|\mathcal I| + |\mathcal J|)|I_0|$ and $I_0$ is one of the time intervals considered in the supremum of~\cref{E:BMO:ext:Dt:bdd}
	\[
		\frac{1}{|A|} \int_A \int_A \frac{|\phi(t) - \phi(\tau)|^2}{|t-\tau|^2} \dtau\dt
		\sim \frac{1}{|I_0|} \int_{I_0} \int_{I_0} \frac{|\phi(t) - \phi(\tau)|^2}{|t-\tau|^2} \dtau\dt \lesssim \eta^2.\qedhere
	\]
\end{proof}



\end{document}